\documentclass[12pt, reqno]{amsart}
\usepackage{amsmath,amssymb,amsthm}

\usepackage{dsfont}
\usepackage{amsmath, amssymb}
\usepackage{amsthm, amsfonts, mathrsfs}
\usepackage{mathptmx}
\usepackage{fullpage}
\usepackage{amsfonts,graphicx}
\numberwithin{equation}{section}
\usepackage[colorlinks=true, pdfstartview=FitV, linkcolor=blue, citecolor=blue, urlcolor=blue]{hyperref}

\newcommand\blfootnote[1]{%
  \begingroup
  \renewcommand\thefootnote{}\footnote{#1}%
  \addtocounter{footnote}{-1}%
  \endgroup}

\newcommand{\NN}{\mathbb{N}}

\newcommand{\RR}{\mathbb{R}}

\newcommand{\EE}{\varepsilon}

\newcommand{\DD}{\textnormal{D}}

\newcommand{\Div}{\textnormal{div}}

\newcommand{\cur}{\textnormal{curl }}

\newcommand{\supp}{\textnormal{supp }}
\newcommand{\Lip}{\textnormal{Lip}}
\newcommand{\loc}{\textnormal{loc}}
\newcommand{\Curl}{\textnormal{curl}}

\newtheorem{Theo}{Theorem}[section]
\newtheorem{lem}[Theo]{Lemma}

\newtheorem{prop}[Theo]{Proposition}

\theoremstyle{plain}
\theoremstyle{definition}
\newtheorem{defi}[Theo]{Definition}

\theoremstyle{remark}
\newtheorem{Rema}[Theo]{Remark}
\newtheorem*{rema*}{Remark}

\author[Y. Maafa]{Youssouf Maafa}
\address{LEDPA, Universit\'e de Batna --2--\\ Facult\'e des Math\'ematiques et Informatique\\ D\'epartement de Math\'ematiques\\ 05000 Batna Alg\'erie}
\email{y.maafa@univ-batna2.dz}

\author[M. Zerguine]{Mohamed Zerguine}
\address{LEDPA, Universit\'e de Batna --2--\\ Facult\'e des Math\'ematiques et Informatique\\ D\'epartement de Math\'ematiques\\ 05000 Batna Alg\'erie}
\email{m.zerguine@univ-batna2.dz}

\date{}

\begin{document}

\title[Axisymmetric MHD]{On the vanishing viscosity limit for the full viscous MHD system with critical axisymmetric initial data}
\maketitle
\begin{abstract} 
The current paper establishes the global well-posedness issue for the full viscous MHD equations in the axisymmetric setting. Global solutions are obtained in critical Besov spaces uniformly to the viscosity when the resistivity is fixed in the spirit of \cite{AKH,Hassainia,hz0}. Furthermore, strong convergence in the resolution spaces with a rate of convergence is also studied.

\end{abstract}
\noindent 
\blfootnote{{\it keywords and phrases:}
Full viscous MHD system, global well-posedness, axisymmetric solutions, critical Besov spaces, paradifferential calculus, inviscid limit.}  
\tableofcontents
\section{Introduction} 
Magnetohydrodynamics (MHD) is the branch of continuum mechanics which deals with the interaction of electromagnetic
fields and conducting fluids. The modeling consists of a coupling between the equations of continuum fluid mechanics and the Maxwell equations of electromagnetism. This coupled phenomenon governs by the following set equations:  
\begin{equation}\label{MHD(mu-kappa)}
\left\{ \begin{array}{ll} 
\partial_{t} v_{\mu}+v_{\mu}\cdot\nabla v_{\mu}-\mu\Delta v_{\mu}+\nabla p_{\mu}=B_{\mu}\cdot\nabla B_{\mu}\quad
 & \textrm{if $(t,x)\in\RR_+\times\RR^3$,}\\
\partial_{t}B_{\mu}+v_{\mu}\cdot\nabla B_{\mu}-\kappa\Delta B_{\mu}=B_{\mu}\cdot\nabla v_{\mu} & \textrm{if $(t,x)\in\RR_+\times\RR^3$,}\\
\Div v_{\mu}=0,\quad\Div B_{\mu}=0, &\\ 
({v}_{\mu},{B}_{\mu})_{| t=0}=({v}^{\mu}_0,{B}^{\mu}_0). \tag{MHD$_{\mu,\kappa}$}
\end{array} \right.
\end{equation} 
Here, $v(t,x)\in\RR^3$ refers to the velocity vector field of the fluid localized in $x\in\RR^3$ at time $t>0$ and $B(t,x)\in\RR^3$ designate the intensity of the magnetic field, which both are assumed to be solenoidal. The scalar quantity $p(t,x)\in\RR$ is the force of the internal pressure, and it may be determined in terms of $v$ and $B$ using the Calder\'on-Zygmund transform.
\begin{equation*}
p\equiv-\sum_{i,j=1}^{2}\mathfrak{R}_i\mathfrak{R}_j(v^iv^j)+\sum_{i=1}^{2}\mathfrak{R}_i\mathfrak{R}_j(B^i B^j)\triangleq p_v+p_B,
\end{equation*}
where $\mathfrak{R}_i=\frac{\partial_i}{\sqrt{-\Delta}}$ stands to Riesz's operator. The positive parameters $\mu$ and $\kappa$ represent the viscosity and resistivity of the fluid, respectively.

\hspace{0.5cm}Mathematical modeling of the Cauchy problem associated with the magnetohydrodynamic (MHD) has been a subject of great interest during the last few decades since the seminal work of Alfv\'en \cite{Alfven}, due to the many practical applications as a scientific endeavor of great importance in various contexts: geomagnetism and planetary magnetism, astrophysics, nuclear fusion (plasma) physics, and liquid metal technology. For extensive literature about this subject, we refer to \cite{Biskamp}.

\hspace{0.5cm}To derive MHD's governing equations, consider an electrical conductivity fluid; and assume that magnetic fields are prevalent. The electrical conductivity and the prevalence of magnetic fields provide two effects. First, electric currents are generated by the motion of the electrically conducting fluid across the magnetic lines of force. As a result, the associated magnetic fields contribute to changes in the existing fields. Second, the fluid elements carrying current traverse magnetic lines of force contribute to additional forces acting on the fluid particle. This twofold interaction between the motions and the fields is responsible for patterns of behavior that are often unexpected and striking. This interaction gives us the well-known Maxwell's equations,
\begin{eqnarray*}
\textrm{(Gauss)}&\qquad&\Div E=\frac{\rho_c}{\epsilon_0},\quad \Div B=0, \\
\textrm{(Maxwell-Faraday)}&\qquad&\Curl E=-\mu\partial_t B, \\
\textrm{(Maxwell-Amp\`ere)}&\qquad&\Curl B=\mu_0 J+\EE_0\mu_0\partial_t E.
\end{eqnarray*}

In electromagnetic units, $\rho_c$ is the charge density, $J$ is the electric current density, $\EE_0$ the permittivity of free space, and $\mu_0$ the magnetic permeability of free space. We point out that Bellac and L\'evy-Leblond was shown in \cite{Bellac and Levy-Leblond} that Maxwell's equations possess two distinct, still utterly coherent, non-relativistic limits. The two limits correspond to situations where either $|E|\gg c| B|$ (electric limit) or $|E|\ll c|B|$ (magnetic limit). Each of the two limits is Galilei covariant, albeit the transformations of the fields and the density, and current are not the same in the two cases. The electric limit amounts to disregard the time-derivative of the magnetic field in Maxwell-Faraday's law of induction. In contrast, the magnetic limit obtains by neglecting the displacement current in Maxwell-Amp\`ere's equation. Physically, the magnetic limit means that the electromagnetic field is instantaneously at equilibrium, and we can ignore the propagation of information by electromagnetic waves. As a result, Maxwell-Faraday's formula in the magnetic limit leads to $\Div J=0$, so the electric current lines are closed (like the magnetic field lines $\Div B=0$).

\hspace{0.5cm}Visibly the above system is not yet closed. For this purpose, we need to know the expression for the electric current field $J$. On account of the nature of the fluid provided by Ohm's law
\begin{equation*}
J=\sigma(E+\mu v\times B).
\end{equation*}
In the framework of the magnetic limit, by combining the previous equations, we shall obtain after an elementary calculus the second equation in the system \eqref{MHD(mu-kappa)}. When a motion of a conducting fluid induces currents through a magnetic field, a Lorentz force will act on the fluid and modify its motion. This force is given by
\begin{equation*}
\mathcal{L}=\kappa J\times B=\kappa \Curl(B\times B).
\end{equation*}
With some vectorial identity, the force $\mathcal{L}$ can be decomposed in the following way
\begin{equation*}
\mathcal{L}=-\frac12\nabla|B|^2+(B\cdot\nabla)B.
\end{equation*}
The last formula says that the magnetic force density has two parts: the first one occurs in the first equation of \eqref{MHD(mu-kappa)} and is known as the magnetic pressure $\frac12|B|^2$ orthogonal to $B$, and second, an orthogonal curvature force $(B\cdot\nabla)B$. This curvature force acts toward the line's center of curvature and is the magnetic-field-line analog of the force that operates on a curved wire or string under tension. For further information about the derivation of the MHD equations and some dynamical aspects, see \cite{Chandrasekhar}.   
 
\hspace{0.5cm}In the absence of the magnetic field ($B\equiv0$), the system \eqref{MHD(mu-kappa)} boils down to the classical incompressible Navier-Stokes equations denoted by \eqref{NS(mu)} and reads as follows.
\begin{equation}\label{NS(mu)}
\left\{ \begin{array}{ll} 
\partial_{t}v+v\cdot\nabla v+\nabla p-\mu\Delta v=0
 & \textrm{if $(t,x)\in \RR_+\times\RR^3$,}\\ 
\Div v=0, &\\ 
{v}_{| t=0}={v}_0. \tag{NS$_{\mu}$}
\end{array} \right.
\end{equation} 
\hspace{0.5cm}When the viscous forces vanish ($\mu=0$), we still are finding the incompressible Euler equations denoted by \eqref{E} and governed by the system

\begin{equation}\label{E}
\left\{ \begin{array}{ll} 
\partial_{t}v+v\cdot\nabla v+\nabla p=0
 & \textrm{if $(t,x)\in \RR_+\times\RR^3$,}\\ 
\Div v=0, &\\ 
{v}_{| t=0}={v}_0. \tag{E}
\end{array} \right.
\end{equation} 

\hspace{0.5cm}The investigation of the local/global well-posedness issue began early with the celebrated work of Lichtenstein \cite{Lichtenstein} and Gunther \cite{Gunther}, solving the Euler equations locally in time in H\"older's spaces $C^{k,\alpha}$ with $k\in\NN$ and $\alpha\in(0,1)$. The globalization of this solution for 2D Euler in H\"older spaces remounts to Wolibner \cite{Wolibner}. In the same direction, Ebin and Marsden \cite{Ebin-Marsden} explored the manifold structure of certain groups of diffeomorphisms with Sobolev spaces $H^s(\RR^N), s>\frac{N}{2}+1$ to assert that Euler's equations are well-posed. This result was later extend by Bourguignon and Br\'ezis \cite{Bouguignon-Brezis} for $W^{s,p}$ spaces with $s>\frac{N}{p}+1$. In the $\RR^N$ framework, Kato established that Euler's equations admits a unique local solution, for initial data $v_0\in H^m(\RR^N)$ with $m$ is an integer be such that $m>\frac{N}{2}+1$ and $v\in C\big([0,T^\star);H^m(\RR^N)\big)$, where $T^\star$ is the maximal lifespan of the solution which satisfies:

\begin{equation*}
\limsup_{t\uparrow T^\star}\int_{0}^{t}\|\nabla v(\tau)\|_{L^\infty}d\tau=\infty,
\end{equation*} 
Next, Kato and Ponce \cite{Kato-Ponce} enhanced this result for $W^{s,p}$ with $s$ is a real number, $s>1+\frac{N}{p}$, taking into account the following commutator estimate: 
\begin{equation}\label{V-law}
\|\Lambda^s(fg)-f\Lambda^s g\|_{L^p}\le C\|\nabla f\|_{L^\infty}\|\Lambda^{s-1}g\|_{L^p}+\|\Lambda^s f\|_{L^p}\|g\|_{L^\infty},
\end{equation}
where, $\Lambda^s$ refers to the nonlocal operator $(\mathbb{I}-\Delta)^{s/2}$ with $s>0$ and $C=C(N,s,p)$. We mention that, $p$ belongs only to $(1,\infty)$. For $p=\infty$, the problem contributes some hardness, so, it has conjectured subsequently in \cite{Grafakos-Maldonado-Naibo}. Some new Kato-Ponce type inequality developed latterly in a series of papers, e.g \cite{Bourgain-Li-1}. Also, the regularity of the solutions of Euler equations has a close link with the vorticity dynamics. The vorticity variable is denoted by $\omega=\cur v$ and defined as a skew-matrix with entries: 
\begin{equation*}
\omega_{i,j}=\partial_j v^i-\partial_i v^j,\; 1\le i,j\le N.
\end{equation*}
A blow-up vorticity criterion for Kato's solutions following Beale-Kato-Majda \cite{Beale-Kato-Majda} (short BKM) reads as follows: if $T^\star$ is the maximal lifespan time, then we have:\begin{equation}\label{Eq:3}
\limsup_{t\uparrow T^\star}\int_{0}^{t}\|\omega(\tau)\|_{L^\infty}d\tau=\infty.
\end{equation}
In particular, for $N=2$, the vorticity can be identified as a scalar function of the type $\omega=\partial_2 v^1-\partial_1 v^1$ which evolves the following nonlinear transport equation:
\begin{equation}\label{Eq:4}
\partial_{t}\omega+v\cdot\nabla\omega=0,
\end{equation}
which permits us to recover the velocity via Biot-Savart law in the following way:
\begin{equation*}
v(t,x)=\frac{1}{2\pi}\int_{\RR^2}\frac{(x-y)^\perp}{|x-y|^2}\omega(t,x)dx, \quad x^\perp=(-x_2,x_1).
\end{equation*}
This shows that Euler equations have a Hamiltonian structure and, in turn, provides an infinity of conservation laws as $\|\omega(t)\|_{L^p}=\|\omega_0\|_{L^p}$ for all $p\in[1,\infty]$. Ergo, in light of \eqref{Eq:3} the Kato's solutions are globally well-posed in time.

\hspace{0.5cm}The aforementioned functional spaces like Sobolev spaces $W^{s,p}$ with $s\in\RR, s>1+\frac{N}{p}$ are sometimes called sub-critical space for D-dimensional Euler and the space $W^{1+\frac{N}{p},p}$ is called critical. The criticality index $s_{N,p}=1+\frac{N}{p}$ is the minimal prerequisite in some sense to achieve the energy estimate. We underline that for these types of spaces, we do not know whether the BKM criterion works or not. For this purpose, Vishisk \cite{Vishik} developed a new result about global well-posedness for 2D Euler equations in the critical Besov spaces $B^{1+1/p}_{p,1}(\RR^2)$, suggesting a new criterion of continuation of solutions:
\begin{equation}\label{V-log}
\|f\circ\Psi\|_{B^0_{\infty,1}}\le C\big(1+\log\big(\|\Psi\|_{\Lip}\|\Psi^{-1}\|_{\Lip}\big)\big)\|f\|_{B^0_{\infty,1}},
\end{equation}
where, $\Psi:\RR^2\rightarrow\RR^2$ is a diffeomorphism which preserves Lebesgue's measure. This result was extended later by T. Hmidi and S. Keraani \cite{Hmd-Ker} for 2D Navier-Stokes equations, where they used the Lagrangian coordinates to establish that the velocity is a Lipschitz function globally in time and bounded uniformly in viscosity. Moreover, they investigated an inviscid limit result between the Navier-Stokes equations and the Euler one. For 3D Euler's equations, the vorticity is a vectorial function defined by $\omega=\nabla\wedge v$ and satisfies the following system:\begin{equation*}
\partial_t\omega+(v\cdot\nabla)\omega-(\omega\cdot\nabla) v=0.
\end{equation*}
The stretching term $(\omega\cdot\nabla) v$ is the main difficulty that affects the motion of the fluid and, so we don't reach the global regularity of Euler's equations. In contrast, some partial results exist in the framework of so-called axisymmetric flows without swirl. An axisymmetric solution without swirl of Euler's equations, meaning that the solution can be split in the cylindrical coordinates $(r,\theta,z)$ as follows:
\begin{equation*}
v(t, x)=v^{r}(t, r, z)\vec{e}_{r}+v^{z}(t, r, z)\vec{e}_{z},
\end{equation*}
where for every $x=(x_1,x_2,z)\in\RR^3$ we have
\begin{equation*}
x_1=r\cos\theta,\quad x_2=r\sin\theta,\quad r\ge0,\quad 0\le\theta<2\pi.
\end{equation*}
Here, the triplet  $(\vec{e}_{r}, \vec{e}_{\theta}, \vec{e}_{z})$ refers to the usual frame of unit vectors in the radial, azimuthal and vertical directions with the notation:
\begin{equation}\label{(1.7)}
\vec{e}_r=\Big(\frac{x_1}{r},\frac{x_2}{r},0\Big),\quad \vec{e}_{\theta}=\Big(-\frac{x_2}{r},\frac{x_1}{r},0\Big),\quad \vec{e}_{z}=(0,0,1).
\end{equation}
For these type of flows the vorticity $\omega$ 
takes the form $\omega\triangleq\omega_{\theta}\vec{e}_{\theta}$ 
with 
\begin{equation}\label{component(theta-B)}
\omega_{\theta}=\partial_{z}v^r-\partial_{r}v^z.
\end{equation}
We examine that the stretching term $\omega\cdot\nabla v$  
close to $\frac{v^r}{r}\omega_{\theta}$. So, by taking $\Omega=\frac{\omega_{\theta}}{r}$, we check by a straightforward calculus that
 
\begin{equation}\label{omega-E}
\left\{\begin{array}{ll}
\partial_{t}\Omega+v\cdot\nabla \Omega=0, & \\
\Omega_{|t=0}=\Omega_0.
\end{array}
\right. 
\end{equation}
One of the main interests of the axisymmetric geometry is that $\Omega$ is transported through the time by the flow in the sense that $\|\Omega(t)\|_{L^p}=\|\Omega_0\|_{L^p}$ for $t\ge0$ like 2D Euler's equations. Ukhovskii and Yudovich \cite{Ukhovskii-Yudovich} explore this identity to elaborate the global existence and uniqueness for initial data $v_0\in H^s$ with $s>\frac72$.  
This assumption was later relaxed to $\frac{\omega_0}{r}\in L^{3,1}$ by Shirota and Yanagisawa \cite{Sh-Ya}, P. Serfati \cite{Serfati-1} and R. Danchin \cite{Danchin}. We will not be discussing singularity formation for 3D Euler equations in any detail, but we advise the reader to consult \cite{Elgindi}. We point out that their proofs are firmly based on the BKM criterion. For critical Besov regularities ${B}_{p, 1}^{1+\frac{3}{p}},$ with $1\le p\le \infty$, Abidi, Hmidi and Keraani \cite{Abidi-Hmidi-Keraani} succeed to gain the global well-posedness in the absence of BKM criterion. They explored vorticity's special geometric structure, leading to a new decomposition of the vorticity. This idea helps to derive the Lipschitz norm of the velocity. The Navier-Stokes system \eqref{NS(mu)} is also well-explored. Worth mentioning, M. Ukhoviskii and V. Yudovich \cite{Ukhovskii-Yudovich}, independently O. Ladyzhenskaya \cite{ol} succeed to recover \eqref{NS(mu)} globally in time, whenever $v_0\in H^1$ and $\omega_0,\frac{\omega_0}{r}\in L^2\cap L^\infty$. This result was later improved in \cite{LMNP} by S. Leonardi, J. M\`alek, J. Nec\u as and M. Pokorn\'y for $v_0\in H^2$ and in \cite{a} by H. Abidi for $v_0\in H^{\frac{1}{2}}$. In the same direction, Hmidi and the second author \cite{hz0} derived the same result as in \cite{Abidi-Hmidi-Keraani} with a uniform bound of velocity for the viscosity. Furthermore, they studied the inviscid limit between the Navier-Stokes equations and Euler one as long as the viscosity is small enough in the spirit of \cite{Hmidi-Keraani} for 2D Navier-Stokes equations. 

\hspace{0.5cm}About MHD equations, a lot of fundamental mathematical investigations have been made. In the case $\mu,\kappa>0$, we would mention the seminal work of G. Duvaut and J.-L. Lions, where they constructed in \cite{Duvaut-Lions} in 2D a global Leray-Hopf weak solution, while in 3D only a local Leray-Hopf weak solution. Next, M. Sermange and R. Temam \cite{Sermange-Temam} handled with the regularity of weak solutions under the assumption that $(v, B)$ belongs to $L^{\infty}\big(0,T; H^1(\RR^3)\big)$.
By contrast, for $\mu>0$ and $\kappa=0$, Jiu and Niu \cite{Jiu-Niu} succeed to examine a local existence of solutions in 2D as soon as initial data in $H^s$, but only for integer $s\ge3$. They also settled a conditional regularity in 2D in the sense that if $B\in L^p\big(0,T;W^{2,q}\big)$, with $\frac{2}{p}+\frac{1}{q}\le3$ and $1\le p\le\frac43,\;2<q\le\infty$ then $T$ can be extended. This latter result was lately generalized by Zhou and Fan \cite{Zhou-Fan}, assuming that a condition $\nabla B\in L^1(0, T; BMO)$ suffices. In 3D, Fan and Ozawa \cite{Fan-Ozawa} suggested a similar conditional regularity result to extend the solution beyond time $T$ once $\nabla v\in L^{1}\big(0,T;L^{\infty}\big)$. The case $\mu=0$ and $\kappa>0$ explored by Kozono \cite{Kozono} where he showed that in 2D, the weak solution is global for divergence-free initial data in $L^2$. In the same way and for 3D, Fan and Ozawa \cite{Fan-Ozawa} demonstrated again that we can extend the solution beyond time $T$ as soon as $\nabla v\in L^{1}\big(0, T; L^{\infty}\big)$. The serious problem reflects on the uniqueness of the extension of a weak global solution in 2D to the classical one. Thus, the problem with full Laplacian seems critical, and its resolution will be discussed below. For $\mu=\kappa=0$, ideal MHD was explored in \cite{Schmidt, Secchi} by exploiting the commutators estimate following Kato and Ponce \cite{Kato-Ponce}. They proved that ideal MHD is well-posed $v,B\in C\big([0,T^\star);H^s\big)$ for initial data $v_0,B_0\in H^s$, with $s>\frac{N}{2}+1$. In \cite{Hmidi-MHD}, Hmidi proved that we could go beyond Kato's solutions and establish local existence and uniqueness in the framework of smooth vortex patch.



\hspace{0.5cm}For 3D axisymmetric initial data, the global well-posedness topic for \eqref{MHD(mu)} catches much attention. We restrict ourselves to some of them. When $\mu>0,\kappa=0$, Z. Lei \cite{Lei} proved under some specific geometry that there exists a unique global solution for initial data $v_0 \in H^2$ and $B_0 \in H^2$ are both
axisymmetric with $v^\theta = 0$ and $B^r_{0}= B^z_{0}=0$ and $B^\theta_0/r \in L^\infty.$ He suggested that the velocity and magnetic fields have the form
\begin{equation}\label{1.10}
 v(t,x)=v^r(t,r,z) \vec{e}_{r}+ v^z(t,r,z) \vec{e}_{z}, \quad B= B^\theta \vec{e}_{\theta}.
\end{equation}
For more details about this structure, see Section 3 below. In the same way, and for $\mu=0,\kappa>0$, Z. Hassainia \cite{Hassainia} exploited the structure \eqref{1.10} to elaborate two global well-posedness results. The first one deals with initial data in the setting of sub-critical Sobolev spaces, i.e., $(v_0,B_0)\in H^s\times H^ {s-2} $ with $s>\frac52$ and $\frac{B_{0}^{\theta}}{r}\in L^\infty$. In the second she assumed that initial data $(v_0,B_0)\in\mathscr{B}^ {1+\frac3p}_{p,1}\times \mathscr{B}^{1+\frac3p}_{p,1}$. Their proof was deeply based, in particular in the second case in the boundness of the vorticity in the Besov space $\mathscr{B}_{\infty,1}^0$ following Vishik \cite{Vishik} because the breakdown of BKM criterion. In what follows to simplify our presentation we take $\kappa=1$, therefore our system denoted \eqref{MHD(mu)} becomes,
\begin{equation}\label{MHD(mu)}
\left\{ \begin{array}{ll} 
\partial_{t}v_{\mu}+v_{\mu}\cdot\nabla v_{\mu}-\mu\Delta v_{\mu}+\nabla p_{\mu}=B_{\mu}\cdot\nabla B_{\mu}
 & \textrm{if $(t,x)\in\RR_+\times\RR^3$,}\\
\partial_{t}B_{\mu}+v_{\mu}\cdot\nabla B_{\mu}-\Delta B_{\mu}=B_{\mu}\cdot\nabla v_{\mu} & \textrm{if $(t,x)\in\RR_+\times\RR^3$,}\\
\Div v_{\mu}=0,\quad\Div B_{\mu}=0, &\\ 
({v}_{\mu},{B}_{\mu})_{| t=0}=({v}^{\mu}_0,{B}^{\mu}_0). \tag{MHD$_{\mu}$}
\end{array} \right.
\end{equation}

\subsection{Aims and the main results}The main concerns of this paper are twofold. The first part interests to conduct the same result as in \cite{Hassainia} and meanwhile obtain a uniform estimate for the viscous solutions of the system in respect to the viscosity parameter. In the second part, we investigate the inviscid limit for the system \eqref{MHD(mu)} towards the resistive one
\begin{equation}\label{MHD(0)}
\left\{ \begin{array}{ll} 
\partial_{t}v+v\cdot\nabla v+\nabla p={B}\cdot\nabla B
& \textrm{if $(t,x)\in\RR_+\times\RR^3$,}\\
\partial_{t}B+v\cdot\nabla B-\Delta B=B\cdot\nabla v & \textrm{if $(t,x)\in\RR_+\times\RR^3$,}\\ 
\Div v=0,\quad\Div B=0, &\\ 
({v},{B})_{| t=0}=({v}_0,{B}_0) \tag{MHD$_{0}$}
\end{array} \right.
\end{equation}
when the viscosity is small enough and we quantify the rate of convergence between velocities and magnetic fields.    

\hspace{0.5cm}The first main result treats essentially the global and uniqueness topic for the system \eqref{MHD(mu)} with initial data lying to some critical Besosov spaces. More precisely, we will prove the following theorem. 
\begin{Theo}[Uniform boundedness of the velocity and magnetic fields]\label{The:1}
Let $p\in[2, +\infty]$ and $(v^{\mu}_0,B^{\mu}_0)\in \big(L^2\cap \mathscr{B}_{p, 1}^{1+\frac{3}{p}}\big) \times L^2$ be an axisymmetric vector field in divergence-free with $v^{\theta}_{0}=0$ and $B^{r}_{0}=B^{z}_{0}=0$. Assume that
\begin{equation}\label{H2}
\Big(\frac{\omega^{\mu}_{\theta}(0)}{r},\frac{B^{\mu}_{\theta}(0)}{r}\Big) \in L^{3,1}\times \big(L^2\cap L^\infty\big), 
\end{equation}
where, $\omega^{\mu}_{\theta}(0)$ is the angular component of the vorticity $v^{\mu}_0$. We distinguish two cases:
\begin{enumerate}
\item[$\bullet$] {\bf Case} $p=\infty$. If $B^{\mu}_0\in\mathscr{B}_{\sigma, 1}^{-1+\frac{3}{\sigma}},\;\sigma\in[2,\infty)$. Then the system \eqref{MHD(mu)} admits a unique global solution in time $(v,B)$, so have 
\end{enumerate}
\begin{equation*}
(v,B)\in \mathscr{C}\big(\RR_+; \mathscr{B}_{\infty, 1}^{1}\big) \times \Big(\mathscr{C}\big(\RR_+;\mathscr{B}_{\sigma, 1}^{-1+\frac{3}{\sigma}}\big) \cap L^{1}_{\loc}\big(\RR_+; \mathscr{B}_{\infty, 1}^{1}\big)\Big). 
\end{equation*}
\begin{enumerate}
\item[$\bullet$] {\bf Case} $p<\infty.$ If $B^{\mu}_0\in\mathscr{B}_{p, 1}^{-1+\frac{3}{p}}$. Then the system \eqref{MHD(mu)} admits a unique global solution in time $(v,B)$, so have 
\begin{equation*}
(v,B)\in \mathscr{C}\Big(\RR_+; \mathscr{B}_{p, 1}^{1+\frac{3}{p}}\Big) \times \Big( \mathscr{C}\big(\RR_+; \mathscr{B}_{p, 1}^{-1+\frac{3}{p}}\big) \cap L^{1}_{\loc} \big(\RR_+; \mathscr{B}_{p, 1}^{1+\frac{3}{p}}\big)\Big). 
\end{equation*}
\end{enumerate}
In both cases, we have
\begin{equation*}
\Big(\frac{\omega_{\theta}}{r},\frac{B_{\theta}}{r}\Big)\in L^{\infty}_{\loc}\big(\RR_{+};L^{3,1}\big)\times L^{\infty}_{\loc}\big(\RR_{+};L^{2}\cap L^\infty \big).
\end{equation*}
Moreover, there holds 

\begin{equation*}
\Vert v(t)\Vert_{\mathscr{B}_{p, 1}^{1+\frac{3}{p}}} +\| B\|_{\widetilde{L}^\infty_t \mathcal{B}^{\frac{3}{p}-1}_{p,1}}+ \| B\|_{\widetilde{L}^1_t \mathcal{B}^{\frac{3}{p}+1}_{p,1}} +\Vert\omega(t)\Vert_{\mathscr{B}_{p, 1}^{\frac{3}{p}}} \le \Phi_6(t),
\end{equation*}
where 
\begin{equation*}
\Phi_6(t)=\triangleq  C_0  \underbrace{\exp(\ldots\exp}_{6\text{ times}}(C_0(  t^{\frac{5}{4}}))\ldots).
\end{equation*}

\end{Theo}
Before going over the main ideas of the proof, we shall make some comments.
\begin{Rema}
According to Z. Lei \cite{Lei} the motivation of the previous results is the resemblance between \eqref{MHD(mu)} and axisymmetric Navier-Stokes equations. In broad terms, the velocity vector field in the general axisymmetric case is written down:
\begin{equation*}
v=v^r\vec e_{r}+v^z\vec e_{z} \quad\mbox{and}\quad b=v^\theta\vec e_{\theta}.
\end{equation*}
But, the only difference between the Navier-Stokes \eqref{NS(mu)} for $(v, b)$ and \eqref{MHD(mu)} for $(v, B)$ given by \eqref{1.10} is the "-" sign. However, this difference of sign significantly changes the difficulties in solving 3D axisymmetric incompressible equations of \eqref{MHD(mu)}. For details, see Section \ref{Axisymmetric flows}.
\end{Rema}
\begin{Rema}
The condition $v^\theta=0 $ means that the velocity vector field is an axisymmetric vector field without swirl. In addition, $v^\theta=0$ and $B^r_0=B^z_0=0$ persist through the time in the sense that the solution keeps this initial property, $v^\theta=0 $ and $B^r = B^z=0$ and so \eqref{1.10} is satisfied.  
\end{Rema}

\begin{Rema}Note that the Lorentz $L^{3.1}$ space imposed by Danchin in \cite{Danchin} to treat Euler's axisymmetric equations is very technical; we can replace it with the Lebesgue space $L^3$.  
\end{Rema}
\begin{Rema}
We point out that for $p<3$, the assumption \eqref{H2} is a consequence of $v^{\mu}_0\in \mathscr{B}_{p, 1}^{1+\frac{3}{p}}$. More precise, we have:
\begin{equation*}
\Big\|\frac{\omega}{r}\Big\|_{L^{3,1}}\le C\|v\|_{\mathscr{B}_{p, 1}^{1+\frac{3}{p}}}.
\end{equation*} 
\end{Rema}

The proof of Theorem \ref{The:1} requires two main steps. The first one is the boundedness of the vorticity through time by exploiting axisymmetric Biot-Savart law and Lorentz spaces. Unhappily, this is not enough to propagate the Lipschitz norm of the velocity since the BKM criterion is not known to be valid for critical regularities. Therefore, we stand driven to derive a new estimate for the vorticity in the Besov space $\mathscr{B}^{0}_{\infty,1}$ as stated in \eqref{V-log}.  For this purpose, we rewrite \eqref{MHD(mu)} under the vorticity and the magnetic fields as follows:
\begin{equation}\label{VE(mu-1)}
\left\{ \begin{array}{ll} 
\partial_t \omega_{\theta}+v\cdot\nabla \omega_{\theta}-\mu\big(\Delta-\frac{1}{r^2}\big)\omega_{\theta}=\frac{v_r}{r}\omega_{\theta}-\partial_z\frac{(B_{\theta})^2}{r},
 & \\
\partial_{t} B_{\theta}+v\cdot\nabla B_{\theta}-\big(\Delta-\frac{1}{r^2}\big)B_{\theta}=\frac{v_r}{r}B_{\theta}.
\end{array} \right.
\end{equation}

By putting $\Omega=\frac{\omega_{\theta}}{r}$ and $\Sigma=\frac{B_{\theta}}{r}$, we discover that the last system takes the form
\begin{equation}\label{VE(mu-2)}
\left\{ \begin{array}{ll} 
\partial_t \Omega+v\cdot\nabla \Omega-\mu\big(\Delta +\frac{2}{r}\partial_r\big)\Omega=-\partial_z\Sigma^2,
 & \\
\partial_{t}\Sigma+v\cdot\nabla \Sigma-\big(\Delta+\frac{2}{r}\partial_r\big)\Sigma=0.
\end{array} \right.
\end{equation} 
Compared to the resistive MHD ($\mu=0$) recently treated in \cite{Hassainia} we have an additional term in $\Omega-$equation which contributes a hardness at the level to estimate $\Omega$ in the Lorentz space $L^{3,1}$. In fact, to bound $\Omega$ in $L^{3,1}$, we necessitate a priori estimate for the stretching magnetic field $\Sigma$. For this purpose, we use the smoothing effect for the second equation in \eqref{VE(mu-2)}. These results help us to reach the boundedness of the velocity uniformly for the viscosity and so the global well-posedness topic for the system \eqref{MHD(mu)}.

\hspace{0.5cm}Our second main result motivates by establishing the inviscid limit of the system \eqref{MHD(mu)} towards \eqref{MHD(0)} when the viscosity goes to zero. In particular, we quantify the convergence rate between velocities and magnetic vector fields. More precisely, we will prove the following theorem.
\begin{Theo}[Rate of convergence]\label{The:1.2}
Let $(v_{\mu},B_{\mu})$ and $(v,B)$ be the solutions of \eqref{MHD(mu)} and \eqref{MHD(0)} systems respectively with the same initial data, which satisfying the same conditions as in Theorem \ref{The:1}. Then for every $ p\in(2, \infty),$ we have:
\begin{equation}\label{Rate}
\Vert v_{\mu}-v\Vert_{L^\infty_t\mathscr{B}^0_{p,1}}+\Vert B_{\mu}-B\Vert_{L^\infty_t\mathscr{B}^{-1}_{p,1} \cap  L^1_t\mathscr{B}^{1}_{p,1}}
\le \Big( (\mu t)^{\frac{1}{5}+\frac{3}{5p}}+ (\mu t)^{\frac{3}{5p}}+(\mu t)^{\frac{3}{\max(p,6)}}  \Big)\Phi_{6}(t).
\end{equation}
\end{Theo}

\begin{Rema}We note that in $L^2$, the rate of convergence is $(\mu t)$ the same as in \cite{hz0, Wu} concerning the axisymmetric Navier Stokes equations, see Section \ref{In:Limit}. But for MHD is $(\mu t)^{\frac12}$, we refer to J. Wu \cite{Wu-1}. However, for the other values of $p$, the rate of convergence differs entirely compared to \cite{hz0, Wu}. This difference is due to the nature of the system \eqref{MHD(mu)} and the used method. 
\end{Rema}

The direct demonstration of the previous theorem in the Besov spaces seems very difficult since the system stems from the difference between systems \eqref{MHD(mu)} and \eqref{MHD(0)} is hyperbolic no symmetric. For this purpose, we start by handling with Lebesgue space $L^2$. Afterward, we proceed by complex interpolation between $\mathscr{B}^0_{2,1}$ and Besov spaces.\\  

{\bf Structure of the paper}. The layout of the present paper is as follows. Section 2 gives a few results about the Littlewood-Paley theory like dyadic decomposition of the unity, Besov spaces, their properties, and paradifferential calculus. We also state a technical lemma about the persistence regularity in Besov spaces for transport-diffusion equation governs the density and magnetic evolution. Section 3 motivates by treating two parts. The first one concerns the energy estimates for different quantities in Lebesgue space $L^2$ (resp. Besov and Lorentz space $L^ {p, q} $). The second part addresses to derive the Lipschitz norm of the velocity through the vorticity decomposition. Finally, we state in the appendix commutator estimates already used in different situations.

\section{Preparatory and preliminaries}

\subsection{Vocabulary of Littlewood-Paley theory}This subsection starts with the definition of the Lorentz spaces and a brief concise about Littlewood-Paley theory.

\begin{defi}\label{Def:7.1} For a measurable function $f^\star$ we define its increasing rearrangement by
\begin{equation*}
f^\star(t)=\inf\big\{s\in\RR_{+}:\mu(\{x,|f(x)|>s\})\le t\big\},
\end{equation*}
where $\mu$ refers to the usual Lebesgue measure. For $(p,q)\in[1,\infty]^2$, the Lorentz space $L^{p,q}$ is the set of functions $f$ such that $\|f\|_{L^{p,q}}<\infty$, with 
\begin{equation*}
\|f\|_{L^{p,q}}\triangleq\left\{\begin{array}{ll}
\bigg(\int_{0}^{\infty}\big(t^{\frac1p}f^\star(t)\big)^{q}\frac{dt}{t}\bigg)^{\frac1q} &\textrm{for $q\in[1,\infty)$,}\\
\sup_{t>0} t^{\frac1p}f^{\star}(t) & \textrm{for $q=\infty$.}
\end{array}
\right.
\end{equation*}
\end{defi}
We can also define the Lorentz space $L^{p,q}$ by the real interpolation process from the Lebesgue spaces
\begin{equation*}
L^{p,q}=[L^1,L^\infty]_{1-\frac{1}{p},q}.
\end{equation*}
where $(p,q)\in ]1,\infty[\times[1,\infty]$. These spaces are characterized by the following properties:
\begin{enumerate}
\item[{\bf(i)}]$L^{p,p}=L^p$,
\item[{\bf(ii)}]$L^{p,q_1}\hookrightarrow L^{p,q_2}$ for every $ 1 \le q_1 \le q_2 \le \infty$,
\item[{\bf(iii)}]$\|uv\|_{L^{p,q}} \le \|u\|_{L^\infty} \|v\|_{L^{p,q}}.$
\end{enumerate}

Next, we state a few phrases about the so-called Littlewood-Paley and some of its properties.
\begin{defi}
Let $\chi\in\mathscr{D}(\RR^3)$ be a reference cut-off function, monotonically decaying along rays and so that: $\chi\equiv 1$ on $B(0,\frac12)$ and $0\le\chi\le1$ on $B(\frac12,1)$. Define $\varphi(\xi)\triangleq\chi(\frac{\xi}{2})-\chi(\xi)$. We obviously check that $\varphi\ge0$ and $$\supp\varphi\subset\mathscr{C}\triangleq\{\xi\in\RR^3:\frac12  \le\|\xi\|\le1\}.$$ 
\end{defi}
We have the following elementary properties, see for example \cite{Bahouri-Danchin-Chemin, che1}.
\begin{prop} Let $\chi$ and $\varphi$ be as above. Then the following assertions are hold.
\begin{enumerate}
\item[{\bf(i)}] \textnormal{Decomposition of the unity:} 
$$
\forall\xi\in\RR^3,\quad \chi(\xi)+\sum_{q\ge0}\varphi(2^{-q}\xi)=1.
$$
\item[{\bf(ii)}] \textnormal{Almost orthogonality in the sense of $\ell^2$:}
$$
\forall\xi\in\RR^3,\quad \frac{1}{2}\le\chi^2(\xi)+\sum_{q\ge0}\varphi^2(2^{-q}\xi)\le1.
$$
\end{enumerate}
\end{prop}

The Littlewood-Paley or cut-off operators or dyadic blocks are defined as follows.
\begin{defi} For every $u\in\mathscr{S}'(\RR^3)$, setting 
\begin{equation*}
\Delta_{-1}u\triangleq\chi(\DD)u,\quad \Delta_{q}u\triangleq\varphi(2^{-q}\DD)u\quad \mbox{if}\;q\in\NN,\quad S_{q}u\triangleq\sum_{j\le q-1}\Delta_{j}u\quad\mbox{for}\; q\ge0.
\end{equation*}
\end{defi}

Some properties of $\Delta_q$ and $S_q$ are listed in the following proposition.
\begin{prop} Let $u,v\in\mathscr{S}'(\RR^3)$ we have
\begin{enumerate}
\item[{\bf(i)}] $\vert p-q\vert\ge2\Longrightarrow\Delta_p\Delta_q u\equiv0$,
\item[{\bf(ii)}] $\vert p-q\vert\ge4\Longrightarrow\Delta_q(S_{p-1}u\Delta_p v)\equiv0$,
\item[{\bf(iii)}] $\Delta_q, S_q: L^p\rightarrow L^p$ uniformly with respect to  $q$ and $p$.
\item[{\bf(iv)}] 
$$
u=\sum_{q\ge-1}\Delta_q u.
$$
\end{enumerate}  
\end{prop}

Likewise, the homogeneous operators $\dot{\Delta}_{q}$ and $\dot{S}_{q}$ are defined by
\begin{equation}\label{Hom}
\forall{q}\in \mathbb{Z}\quad\dot{\Delta}_{q}=\varphi(2^{q}D)u, \quad \dot{S}_{q}=\sum_{ j\le q-1}\dot{\Delta}_{j}v.
\end{equation}
Now, we define the Besov spaces in the following way.
\begin{defi} For $(s,p,r)\in\RR\times[1,  +\infty]^2$. The inhomogeneous Besov space $B_{p,r}^s$ (resp. the homogeneous Besov space $\dot{B}_{p,r}^s$) is the set of all tempered distributions $u\in\mathscr{S}^{'}$ (resp. $u\in\mathscr{S}^{'}_{|{\bf P}})$ such that
\begin{eqnarray*}
&&\Vert u\Vert_{\mathscr{B}_{p, r}^{s}}\triangleq\Big(2^{qs}\Vert \Delta_{q} u\Vert_{L^{p}}\Big)_{\ell^r}<\infty. \\
&&\big(\mbox{resp. }\Vert u\Vert_{\dot{\mathscr{B}}_{p, r}^{s}}\triangleq\ \Big(2^{qs}\Vert \dot\Delta_{q} u\Vert_{L^{p}}\Big)_{\ell^r(\mathbb{Z})}<\infty\big), 
\end{eqnarray*}
where ${\bf P}$ denotes the set of polynomials.
\end{defi}


The celebrate  {\it Bony's} decomposition \cite{b} enables us to split formally the product of two tempered distributions $u$ and $v$ into three pieces. In what follows, we shall adopt the following definition for paraproduct and remainder:
\begin{defi} For a given $u, v\in\mathscr{S}'$ we have
 $$
uv=T_u v+T_v u+\mathscr{R}(u,v),
$$
with
$$T_u v=\sum_{q}S_{q-1}u\Delta_q v,\quad  \mathscr{R}(u,v)=\sum_{q}\Delta_qu\widetilde\Delta_{q}v  \quad\hbox{and}\quad \widetilde\Delta_{q}=\Delta_{q-1}+\Delta_{q}+\Delta_{q+1}.
$$
\end{defi}

The mixed space-time spaces are stated as follows. 
\begin{defi} Let $T>0$ and $(s,\beta,p,r)\in\RR\times[1, \infty]^3$.  We define the spaces $L^{\beta}_{T}\mathscr{B}_{p,r}^s$ and $\widetilde L^{\beta}_{T}\mathscr{B}_{p,r}^s$ respectively by: 
$$
L^\beta_{T}\mathscr{B}_{p,r}^s\triangleq\Big\{u: [0,T]\to\mathscr{S}^{'}; \Vert u\Vert_{L_{T}^{\beta}B_{p, r}^{s}}=\big\Vert\big(2^{qs}\Vert \Delta_{q}u\Vert_{L^{p}}\big)_{\ell^{r}}\big\Vert_{L_{T}^{\beta}}<\infty\Big\},
$$
$$
\widetilde L^{\beta}_{T}\mathscr{B}_{p,r}^s\triangleq\Big\{u:[0,T]\to\mathscr{S}^{'}; \Vert u\Vert_{\widetilde L_{T}^{\beta}\mathscr{B}_{p, r}^{s}}=\big(2^{qs}\Vert \Delta_{q}u\Vert_{L_{T}^{\beta}L^{p}}\big)_{\ell^{r}}<\infty\Big\}.
$$
The relationship between these spaces is given by the following embeddings. Let $ \varepsilon>0,$ then 
\begin{equation}\label{embeddings}
\left\{\begin{array}{ll}
L^\beta_{T}\mathscr{B}_{p,r}^s\hookrightarrow\widetilde L^\beta_{T}\mathscr{B}_{p,r}^s\hookrightarrow L^\beta_{T}\mathscr{B}_{p,r}^{s-\varepsilon} & \textrm{if  $r\geq \beta$},\\
L^\beta_{T}\mathscr{B}_{p,r}^{s+\varepsilon}\hookrightarrow\widetilde L^\beta_{T}\mathscr{B}_{p,r}^s\hookrightarrow L^\beta_{T}\mathscr{B}_{p,r}^s & \textrm{if $\beta\geq r$}.
\end{array}
\right.
\end{equation}
\end{defi}
The main interest of the {\it Bernstein's} inequalities is that the derivatives (or more generally the Fourier multipliers) act in a very special way on distributions the Fourier transform of which is supported in a ball or a ring. The proof can be found in \cite{Bahouri-Danchin-Chemin,che1}.
\begin{lem}\label{Bernstein} There exists a constant $C>0$ such that for $1\le a\le b\le\infty$, for every function $u$ and every $q\in\NN\cup\{-1\}$, we have
\begin{enumerate}
\item[{\bf(i)}]
\begin{equation*}
\sup_{\vert\alpha\vert=k}\Vert\partial^{\alpha}S_{q}u\Vert_{L^{b}}\le C^{k}2^{q\big(k+3\big(\frac{1}{a}-\frac{1}{b}\big)\big)}\Vert S_{q}u\Vert_{L^{a}},\\
\end{equation*}
\item[{\bf(ii)}]
\begin{equation*}
C^{-k}2^{qk}\Vert\Delta_{q}u\Vert_{L^{a}}\le\sup_{\vert\alpha\vert=k}\Vert\partial^{\alpha}\Delta_{q}u\Vert_{L^{a}}\le C^{k}2^{qk}\Vert\Delta_{q}u\Vert_{L^{a}}.
\end{equation*}
\end{enumerate}
\end{lem}

As a consequence of Bernstein inequality {\bf(i)} is the embedding $
\mathscr{B}_{p,r}^{s}\hookrightarrow \mathscr{B}^{\widetilde s}_{\widetilde p,\widetilde r}\quad \textnormal{whenever}\; \widetilde{p}\ge p$, with $\widetilde s<s-2\big(\frac{1}{p}-\frac{1}{\widetilde p}\big)$ or $\widetilde s=s-2\big(\frac{1}{p}-\frac{1}{\widetilde{p}}\big)$ and $\widetilde r\le r$.

We end this paragraph by the persitence of Besov spaces for the following transport-diffusion equation:
\begin{equation}\label{TD(mu)}
\left\{ \begin{array}{ll} 
\partial_{t}f+v\cdot\nabla f-\mu\Delta f=g, \\
g_{| t=0}=g^0.  
\end{array} \right.\tag{TD$_\mu$}
\end{equation} 
\begin{prop}\label{Prop:1.1} Let $(s, r, p)\in(-1, 1)\times[1, \infty]^2$ and $v$ be a smooth vector field in divergence-free. We assume that $f^0\in\mathscr{B}_{p, r}^{s}$ and $g\in L_{loc}^{1}(\RR_+; \mathscr{B}_{p, r}^{s})$. Then for every smooth solution $a$  of \eqref{TD(mu)} and  $t\geq0$ we have
\begin{equation*}
\Vert f(t)\Vert_{\mathscr{B}_{p, r}^{s}}\le Ce^{CV(t)}\bigg(\Vert f^0\Vert_{\mathscr{B}_{p, r}^{s}}+\int_{0}^{t}e^{-CV(\tau)}\Vert g(\tau)\Vert_{\mathscr{B}_{p, r}^{s}}d\tau\bigg),
\end{equation*}
with
\begin{equation*}
V(t)=\int_{0}^t\|\nabla v(\tau)\|_{L^\infty}d\tau
\end{equation*}
and $C$  a constant which depends only on $s$ and not on the viscosity. For the limit case 
\begin{equation*}
s=-1, r=\infty \mbox{ and } p\in[1, \infty]\quad\mbox{ or }\quad s=1, r=1 \mbox{ and } p\in[1, \infty]
\end{equation*}
the above estimate remains true despite we change $V(t)$  by $Z(t)\triangleq\Vert v\Vert_{L_{t}^{1}\mathscr{B}_{\infty, 1}^{1}}$. In addition if $f=\textnormal{curl } v$, then the above estimate holds true for all $s\in[1, +\infty)$.
\end{prop}
\section{Axisymmetric flows for MHD}\label{Axisymmetric flows}
In this section we develop the axisymmetric geometry for the MHD equations \eqref{MHD(mu)} and we
derive their cylindrical form. For this purpose, let $ x = (x_1, x_2,x_3) \in  \RR^3$ define $r =(x^2_1+x^2_2)^{1/2}$. An axisymmetric solution for 3D MHD equations \eqref{MHD(mu)} is a triplet $(v,B, p)$ configured as:
\begin{equation*} 
\left\{ \begin{array}{ll}
 v(t,x)=v^r(t,r,z) \vec{e}_{r}+v^\theta(t,r,z) \vec{e}_{\theta} +v^z(t,r,z) \vec{e}_{z}\\
 B(t,x)=B^r(t,r,z) \vec{e}_{r}+B^\theta(t,r,z) \vec{e}_{\theta} +B^z(t,r,z) \vec{e}_{z}\\
p(t,x)=p(t,r,z),
\end{array} \right.
\end{equation*}
where, $(\vec{e}_{r},\vec{e}_{\theta} ,\vec{e}_{z})$ refers to the cylindrical basis of $\RR^3$ expressed in \eqref{(1.7)}. Let us mention that the components $v^r, v^\theta$ and $v^z$ (resp. $B^r, B^\theta$  and $B^z$) do not depend on the angular (swirl) variable. So, the general axisymmetric version of \eqref{MHD(mu)} is given by


\begin{equation}\label{Cylindrical-version}
\left\{\begin{array}{llllll}
\partial_t v^r+v^r\partial_r v^r +v^z \partial_z v^r -\frac{(v^\theta)^2}{r}-\mu(\Delta-\frac{1}{r^2})v^r +\partial_r p=B^r\partial_r B^r +B^z \partial_z B^r-\frac{(B^{\theta})^2}{r}, &\\
\partial_t v^{\theta}+v^r\partial_r v^\theta +v^z \partial_z v^\theta +\frac{v^r}{r}v^{\theta}-\mu(\Delta-\frac{1}{r^2})v^\theta=B^r\partial_r B^\theta +B^z \partial_z B^\theta+\frac{B^r}{r}B^\theta,&\\
\partial_t v^z+v^r\partial_r  v^z +v^z \partial_z v^{z} -\mu \Delta v^z+\partial_z p=B^r\partial_r B^z+B^z \partial_z B^z, &\\
\partial_t B^r+v^r\partial_r B^r +v^z \partial_z B^r -\Delta B^r -(\Delta-\frac{1}{r^2})B^r=B^r\partial_r v^r+B^z \partial_z v^r, &\\
\partial_t B^{\theta}+v^r\partial_r B^\theta +v^z \partial_z B^\theta +\frac{B^r}{r}v^{\theta}-(\Delta-\frac{1}{r^2})B^\theta=B^r\partial_r v^\theta+B^z \partial_z v^\theta +\frac{v^r}{r}B^\theta,&\\
\partial_t B^z+v^r\partial_r B^z +v^z \partial_z B^z-\Delta B^z =B^r\partial_r v^z+B^z \partial_z v^z ,
\end{array}
\right.
\end{equation}
with, the incompressibility requirement $\Div v = 0$ and $\Div B = 0$:
\begin{equation}\label{Divv,DivB}
\partial_r v^r + \frac{v^r}{r}+\partial_z v^z=0, \quad \partial_r B^r + \frac{B^r}{r}+\partial_z B^z=0.
\end{equation}

\hspace{0.5cm}If $(v^{\theta}_{0},B^{r}_{0},B^{z}_{0})$ vanishes. An elementary cumputation claims that $(v^\theta,B^r,B^z)$ is also vanishes for all time, so, $(v,B)$ takes the form:  
 \begin{equation}\label{special-structure} 
\left\{ \begin{array}{ll}
 v(t,x)=v^r(t,r,z) \vec{e}_{r} +v^z(t,r,z) \vec{e}_{z}\\
 B(t,x)=B^\theta(t,r,z) \vec{e}_{\theta} ,
\end{array} \right.
\end{equation}
Further, the triplet $(v,B,p)$ given by \eqref{Cylindrical-version} and \eqref{Divv,DivB} obeys the following parabolic equations:
\begin{equation*}
\left\{ \begin{array}{ll}
\partial_{t}v^{r}+ v^r\partial_r  v^r +v^z \partial_z v^{r}-\frac{(v^\theta)^2}{r}+\partial_r p -\mu\Big(\Delta -\frac{1}{r^{2}}\Big)v^{r}=-\frac{(B^\theta)^2}{r},\\
\partial_{t}v^{z}+v^r\partial_r  v^z +v^z \partial_z v^{z} +\partial_z p -\mu\Delta v_{z}=0,\\
\partial_{t} B^\theta+ v^r\partial_r B^\theta +v^z \partial_z B^\theta -(\Delta-\frac{1}{r^{2}}) B^\theta =\frac{v^r B^\theta }{r},
\end{array} \right.
\end{equation*}
together with the incompressible condition
$$\partial_r v^r+\frac{v^r}{r} +\partial_z v^z =0.$$

\hspace{0.5cm}In terms of vorticity $\omega= \nabla \times v $ we have:
$$\omega(t,x)=\omega ^r(t,r,z) \vec{e}_{r}+\omega ^\theta(t,r,z) \vec{e}_{\theta} +\omega ^z(t,r,z) \vec{e}_{z},$$
where
$$
\omega^r =-\partial_z v^\theta, \quad  \omega^\theta =\partial_r v^z-\partial_z v^r, \quad \omega^z=\frac{1}{r}\partial_r(rv^\theta).  
$$
The axisymmetric geormetry without swirl implies that $\omega^r=\omega_z=0$ and $\omega^{\theta}=\partial_r v^r-\partial_z v^z$ (see, Proposition \ref{Prop-1} below). A straightforward calculus claims that $(\omega^\theta,B^\theta)$ satifies
\begin{equation*}
\left\{ \begin{array}{ll} 
\partial_t \omega_{\theta}+v\cdot\nabla \omega_{\theta}-\mu\big(\Delta-\frac{1}{r^2}\big)\omega_{\theta}=\frac{v_r}{r}\omega_{\theta}-\partial_z\frac{(B_{\theta})^2}{r},
 & \\
\partial_{t} B_{\theta}+v\cdot\nabla B_{\theta}-\big(\Delta-\frac{1}{r^2}\big)B_{\theta}=\frac{v_r}{r}B_{\theta}
\end{array} \right.
\end{equation*} 
Define the quantities  $\Omega=\frac{\omega_\theta}{r} $ and $ \Sigma=\frac{B^\theta}{r}$, we discover that $(\Omega,\Sigma) $ is governed by
\begin{equation*}
\left\{ \begin{array}{ll} 
\partial_t \Omega+v\cdot\nabla \Omega-\mu\big(\Delta +\frac{2}{r}\partial_r\big)\Omega=-\partial_z\Sigma^2,
 & \\
\partial_{t}\Sigma+v\cdot\nabla \Sigma-\big(\Delta+\frac{2}{r}\partial_r\big)\Sigma=0.
\end{array} \right.
\end{equation*} 
Really, the previous system will be helpful to set up some important a priori estimates which arising in the boudedness of the vorticity in $L^\infty$ and $\mathscr{B}^0_{\infty,1}$ respectively. These estimates permit us to derive the Lipschitz norm of the velocity and consequently the system \eqref{MHD(mu)} is globally well-posed in time, see Subsection \ref{Axi- estimates}. 
\begin{Rema}We notice that the axisymmetric Navier-Stokes equations ($B\equiv0)$ read as follows:
\begin{equation}\label{axisym-NS}
\left\{ \begin{array}{ll}
\partial_{t}v^{r}+ v^r\partial_r  v^r +v^z \partial_z v^{r}+\partial_r p -\mu(\Delta -\frac{1}{r^{2}})v_{r}=0,\\
\partial_t v^{\theta}+v^r\partial_r v^\theta +v^z \partial_z v^\theta +\frac{v^r}{r}v^{\theta}-\mu(\Delta-\frac{1}{r^2})v^\theta=0, &\\
\partial_{t}v^{z}+v^r\partial_r  v^z +v^z \partial_z v_{z} +\partial_z p -\mu\Delta v_{z}=0.\\
\end{array} \right.
\end{equation}
By denoting 
\begin{equation*}
v=v^r\vec e_{r}+v^z\vec e_{z} \quad\mbox{and}\quad b=v^\theta\vec e_{\theta},
\end{equation*}
we discover that $(v,b)$ obyes the system
\begin{equation}\label{NS-mu}
\left\{\begin{array}{lll}
\partial_t v+v\cdot\nabla v-\mu\Delta v+\nabla p=-b\cdot\nabla b, &\\
\partial_t b+v\cdot\nabla b-\Delta b=-b\cdot\nabla v, &\\
\Div v=0,\quad\Div b=0.\tag{NS$_\mu$}
\end{array}
\right.
\end{equation}
We remark that there is a resseblamce between the MHD system \eqref{MHD(mu)} and \eqref{NS-mu}. Consequently, we can recover \eqref{MHD(mu)} from \eqref{NS-mu} by chanching the sign of the terms $b\cdot\nabla b$ and $b\cdot\nabla v$. Notice that the sign of $b\cdot\nabla v$ plays a crucial role in the a priori estimates for \eqref{MHD(mu)}.
\end{Rema}

\section{Study of a vorticity like equation}


This section concerns the treatment of some geometrical estimates of any solution carrying-out a vorticity like equation formulated as follows:
\begin{equation} \label{Eq-II-1}
\left\{ \begin{array}{ll}
\partial_{t}\zeta+(v\cdot\nabla)\zeta-\Delta\zeta=(\zeta\cdot\nabla)v+\Curl(B\cdot\nabla B),\\
\zeta_{\mid t=0}=\zeta^{0}.
\end{array} \right.
\end{equation}
Here, $\zeta=(\zeta_{1},\zeta_{2},\zeta_{3})$ is an unknown vector-valued function.  Our main result in this section reads as follows:
\begin{prop}\label{Prop-1} Let $v$ be an axisymmetric smooth vector field in divergence-free and $\zeta$ be  the solution of \eqref{Eq-II-1} with smooth initial data $\zeta^0.$ Then the following assertions hold.
\begin{itemize}
\item[{\bf(i)}]Assume that  $\Div\zeta_{0}=0$, then for every $t\in\RR_{+}$, we have $\Div\zeta(t,x)=0$.
\item[{\bf(ii)}]Assume that $\zeta^{0}=\zeta_{\theta}^0(r,z)\vec{e}_{\theta},$ then for every $t\in\RR_{+}$ we have $\zeta(t,x)=\zeta_{\theta}(t,r,z)\vec{e}_{\theta}$.

\item[{\bf(iii)}]Under assumption {\bf(ii)} we have $\zeta_{1}(t, x_{1}, 0, z)=\zeta_{2}(t, 0, x_{2}, z)$ and
\begin{equation*}
\partial_{t}\zeta_{\theta}+(v\cdot\nabla)\zeta_{\theta}-\Big(\partial_{r}^{2}+\partial_{z}^{2}+\frac{1}{r}\partial_{r}-\frac{1}{r^{2}}\Big)\zeta_{\theta}=\frac{v^{r}}{r}\zeta_{\theta}-\partial_z\frac{(B_{\theta})^2}{r}.
\end{equation*}
\end{itemize}
\end{prop}
\begin{proof} {\bf (i)} To reach this identity, apply the divergence operator $"\Div"$ to \eqref{Eq-II-1}. An elementay calculus, on account $\Div v=0$ and the algebreic identity $\Div(\Curl)=0$ provide us
\begin{equation}\label{Eq-II-2}
\left\{ \begin{array}{ll}
\partial_{t}\Div\zeta+(v\cdot\nabla)\Div\zeta-\Delta\zeta\Div\zeta=0,\\
\Div\zeta_{\mid t=0}=0.
\end{array} \right.
\end{equation}
So, the maximum principle leads to
\begin{equation*}
\Vert\Div\zeta\Vert_{L^{\infty}}\le\Vert\Div\zeta^{0}\Vert_{L^{\infty}},
\end{equation*}
this yields the aimed result.

{\bf (ii)} We designate by $(\zeta_{r}, \zeta_{\theta}, \zeta_{z})$ the coordinates of $\zeta$ in cylindrical basis $(\vec e_r,\vec e_\theta, \vec e_z)$. We perform the result in two steps: we claim first that the components of $\zeta$ does not rely upon on the angular variale $\theta$. Next, we claim that the components $\zeta_{r}$ and $\zeta_{z}$ vanish.

\hspace{0.5cm}We intend to check that \eqref{Eq-II-1} is invariant under rotation transform. Doing so, let us consider the rotation with angle $\alpha$ and axis $(oz)$, $\mathcal{R}_{\alpha}$ and defined by
\begin{equation*}
\mathcal{R}_{\alpha} =\left(
\begin{array}{ccc}
\cos \alpha& -\sin\alpha& 0 \\
\sin\alpha & \cos\alpha &  0 \\
0 & 0 & 1
\end{array}
\right),\quad \alpha\in\RR.
\end{equation*}
Setting ${\zeta}_{\alpha}(t, x)=\mathcal{R}_{\alpha}^{-1}\zeta{(t, \mathcal{R}_{\alpha}x)}$. We show that ${\zeta}_{\alpha}$ satisfies also \eqref{Eq-II-1}. The fact that the operator $\Delta$ commutes with rotations, then it follows
\begin{equation*}
\Delta(\zeta(t,\mathcal{R}_{\alpha}x))=(\Delta\zeta)(t,\mathcal{R}_{\alpha}x),
\end{equation*}
one deduce that
\begin{equation}\label{Eq-II-3}
\Delta\zeta_{\alpha}(t,x)=\mathcal{R}_{\alpha}^{-1}(\Delta\zeta)(t,\mathcal{R}_{\alpha}x).
\end{equation}
About the advection term $v\cdot\nabla\zeta$, we check that
\begin{equation*}
(\mathcal{R}_{\alpha}^{-1}(v\cdot\nabla\zeta))(t,\mathcal{R}_{\alpha}x)=(v\cdot\nabla\,\mathcal{R}_{\alpha}^{-1}\zeta)(t,\mathcal{R}_{\alpha}x).
\end{equation*}
Consequently,
\begin{equation*}
v(t,x)\cdot\nabla\,(\mathcal{R}_{\alpha}^{-1}\zeta(t,\mathcal{R}_{\alpha}x))=\big(\mathcal{R}_{\alpha}v(t,\mathcal{R}_{\alpha}^{-1}x)\cdot\nabla\mathcal{R}_{\alpha}^{-1}\zeta\big)(t,\mathcal{R}_{\alpha}(x)).
\end{equation*}
Since the velocity is axisymmetric vector field, thus we obtain
\begin{equation*}
v(t,x)\cdot\nabla\,\zeta_{\alpha}(t,x)=\big(v\cdot\nabla\mathcal{R}_{\alpha}^{-1}\zeta\big)(t,\mathcal{R}_{\alpha}(x))).
\end{equation*}
Join together the last two estimates, it hlods
\begin{equation}\label{Eq-II-4}
(\mathcal{R}_{\alpha}^{-1}(v\cdot\nabla\zeta))(t,\mathcal{R}_{\alpha}x)=v(t,x)\cdot\nabla\zeta_{\alpha}(t,x).
\end{equation}
For the stretching term $\zeta\cdot\nabla v$ step by step we find that
\begin{eqnarray}\label{Eq-II-5}
\nonumber(\mathcal{R}_{\alpha}^{-1}(\zeta\cdot\nabla v))(t,\mathcal{R}_{\alpha}x)&=&(\zeta\cdot\nabla\,\mathcal{R}_{\alpha}^{-1}v)(t,\mathcal{R}_{\alpha}x)\\
\nonumber&=&\big(\zeta\cdot\nabla\,v(t,\mathcal{R}_{\alpha}^{-1}x\big)(t,\mathcal{R}_{\alpha}x)\\
&=&\zeta_{\alpha}\cdot\nabla\,v(t,x).
\end{eqnarray}
Collecting \eqref{Eq-II-3}, \eqref{Eq-II-4} and \eqref{Eq-II-5} and insert them in \eqref{Eq-II-1} to infer that
\begin{equation*}
\partial_{t}\zeta_{\alpha}+v\cdot\nabla\zeta_{\alpha}-\mu\Delta\zeta_{\alpha}=\zeta_{\alpha}\cdot\nabla v+\Curl(B\cdot\nabla B) .
\end{equation*}
The fact that $\zeta_{\alpha}^0(x)=\zeta^0(x)$ and the uniqueness issue confirm that $\zeta_{\alpha}(t,x)=\zeta(t,x)$. Consequently, the components of $\zeta$ don't depend on the angle $\theta$, so have the first step is reached.

\hspace{0.5cm}Now, our task is to prove that the components $\zeta_{r}$ and $\zeta_{z}$ vanish. To do this, we write their equations and taking  the inner product in $\RR^3$ for $\zeta$'s equation with $\vec{e}_{r}$. Exploring the fact $v$ is axisymmetric, so a straightforward computations allow us to write
\begin{eqnarray*}
(v\cdot\nabla\zeta)\cdot\vec{e}_{r}&=&v_{r}\partial_{r}\zeta_{r}+v_{z}\partial_{z}\zeta_{r}\\
&=&v \cdot\nabla\zeta_{r},
\end{eqnarray*}
and
\begin{eqnarray*}
(\zeta\cdot\nabla v)\cdot\vec{e}_{r}&=&\zeta_{r}\partial_{r}v_{r}+\frac{1}{r}\zeta_{\theta}\partial_{\theta}v_{r}+\zeta_{z}\partial_{z}v_{r}\\
&=&\zeta_{r}\partial_{r}v_{r}+\zeta_{z}\partial_{z}v_{r}.
\end{eqnarray*}
For the dissipative term we have by definition and by the first step, 
\begin{eqnarray*}
\mu\Delta\zeta\cdot\vec{e}_{r}&=&\mu\Big(\partial_{r}^{2}\zeta\cdot\vec{e}_{r}+\frac{1}{r}\partial_{r}\zeta_{r}\cdot\vec{e}_{r}+\frac{1}{r^{2}}\partial_{\theta}^{2}\zeta\cdot\vec{e}_{r}+\partial_{z}^{2}\zeta\cdot\vec{e}_{r}\Big)\\
&=&\mu\Big(\partial_{r}^{2}(\zeta\cdot\vec{e}_{r})+\frac{1}{r}\partial_{r}(\zeta\cdot\vec{e}_{r})+\frac{1}{r^{2}}\Big(\partial_{\theta}^{2}(\zeta\cdot\vec{e}_{r})+(\zeta\cdot\vec{e}_{r})\\
&&-2\partial_{\theta}\zeta\cdot\vec{e}_{\theta}\Big)+\partial_{z}^{2}(\zeta\cdot\vec{e}_{r})\Big)\\
&=&\mu\Big(\partial_{r}^{2}\zeta_{r}+\frac{1}{r}\partial_{r}\zeta_{r}+\frac{1}{r^{2}}\partial_{\theta}^{2}\zeta_{r}-\frac{1}{r^{2}}\zeta_{r}-\frac{2}{r^2}\partial_{\theta}\zeta_{\theta}+\partial_{z}^{2}\zeta_{r}\Big)\\
&=&\mu(\Delta-\frac{1}{r^{2}})\zeta_{r}.
\end{eqnarray*}
For the last term $\Curl(B\cdot\nabla B)$, we combine \eqref{special-structure} and $\Curl(B\cdot\nabla B)=-\partial_z\big(\frac{B^\theta}{r}B\big)$ to obtain
\begin{equation*}
\Curl(B\cdot\nabla B)\cdot\vec e_r=0.
\end{equation*}
Finally, we discover that $\zeta_r$ obeys the following Cauchy problem:
\begin{equation}\label{Eq-II-6}
\left\{ \begin{array}{ll}
\partial_{t}\zeta_{r}+v\cdot\nabla\zeta_{r}-\mu(\Delta-\frac{1}{r^{2}})\zeta_{r}=\zeta_{r}\partial_{r}v_{r}+\zeta_{z}\partial_{z}v_{r},\\
{\zeta_{r}}_{\vert t=0}=\zeta_{r}^0.
\end{array} \right.
\end{equation}
Similarly, we can find that the component $\zeta_{z}$ satisfies the following equation:
\begin{equation}\label{Eq-II-7}
\left\{ \begin{array}{ll}
\partial_{t}\zeta_{z}+v\cdot\nabla\zeta_{z}-\mu\Delta\zeta_{z}=\zeta_{r}\partial_{r}v_{z}+\zeta_{z}\partial_{z}v_{z},\\
{\zeta_{z}}_{\vert t=0}=0.
\end{array} \right.
\end{equation}
At this stage, we develop and $L^p$ estimate for $\zeta_r$ (resp. $\zeta_z$) by multiplying \eqref{Eq-II-6} (resp. \eqref{Eq-II-7}) by $\vert\zeta_{r}\vert^{p-2}\zeta_{r}$ (resp. $\vert\zeta_{r}\vert^{p-2}\zeta_{r}$). Thereafter, integrating by parts over $\RR^3$, account $\Div v=0$ and H\"older's inequality give
\begin{eqnarray*}
\frac{1}{p}\frac{d}{dt}\Vert\zeta_{r}(t)\Vert_{L^{p}}^{p}&+&\mu (p-1)\int_{\mathbb{R}^3}|\nabla\zeta_{r}|^2|\zeta_{r}|^{p-2}dx+\mu\int_{\mathbb{R}^3}\frac{|\zeta_{r}|^p}{r^2}dx\\
&\le&\int_{\mathbb{R}^{3}}|\zeta_{r}|^{p}\partial_{r}v_{r}dx
+\int_{\mathbb{R}^{3}}\Sigma_{z}|\zeta_{r}|^{p-2}\zeta_{r}\partial_{z}v_{r}dx\\
&\leq&\Big(\Vert\zeta_{r}\Vert_{L^{p}}^{p}+\Vert\zeta_{z}\Vert_{L^{p}}\Vert\zeta_{r}\Vert_{L^{p}}^{p-1}
\Big)\Vert\nabla v\Vert_{L^{\infty}},
\end{eqnarray*}
Therefore
\begin{equation*}
\Vert\zeta_{r}(t)\Vert_{L^{p}}\leq\Vert\zeta_{ r}^0\Vert_{L^{p}}+\int_{0}^{t}\big(\Vert\zeta_{r}(\tau)\Vert_{L^{p}}+\Vert\zeta_{z}(\tau)\Vert_{L^{p}}\big)\Vert\nabla v(\tau)\Vert_{L^{\infty}}d\tau,
\end{equation*}
respectively
\begin{equation*}
\Vert\zeta_{z}(t)\Vert_{L^{p}}\leq\Vert\zeta_{ z}^0\Vert_{L^{p}}+\int_{0}^{t}\big(\Vert\zeta_{r}(\tau)\Vert_{L^{p}}+\Vert\zeta_{z}(\tau)\Vert_{L^{p}}\big)\Vert\nabla v(\tau)\Vert_{L^{\infty}}d\tau.
\end{equation*}
Putting together the last two estimates, Gronwall's inequality asserts for every $p\in[2,\infty]$ that
\begin{equation*}
\Vert\zeta_{r}(t)\Vert_{L^{p}}+\Vert\zeta_{z}(t)\Vert_{L^{p}}\leq\Big(\Vert\zeta_{r}^0\Vert_{L^{p}}+\Vert\zeta_{z}^0\Vert_{L^{p}}\Big)e^{2\int_{0}^t\|\nabla v(\tau)\|_{L^\infty}d\tau}.
\end{equation*}
Since $\zeta_{r}^0=\zeta_{z}^0=0$ then we deduce that  $\zeta_{r}(t)=\zeta_{z}(t)=0$ for all $t\in\mathbb{R}_{+}$, so {\bf(ii)} is then achieved.  

{\bf (iii)} We emphasize that the first statement is a direct consequence of $\zeta\wedge\vec{e}_{\theta}=\vec{0}.$ Next, we will claim the exact equation that governs the angular component $\zeta_{\theta}$. For this aim, we develop an analog calculus as above to clear that 
\begin{eqnarray*}
\Delta\zeta\cdot\vec{e}_{\theta}&=&\Delta(\zeta_{\theta}\vec{e}_{\theta})\cdot\vec{e}_{\theta}\\
&=&\Delta\zeta_{\theta}-\frac{\zeta_{\theta}}{r^2},
\end{eqnarray*}
and
\begin{eqnarray*}
(v\cdot\nabla\zeta)\cdot\vec{e}_{\theta}=v\cdot\nabla\zeta_{\theta},\quad
(\zeta\cdot\nabla v)\cdot\vec{e}_{\theta}=\frac{v_{r}}{r}\zeta_{\theta}.
\end{eqnarray*}
Thanks to \eqref{Eq-II-1}, the angular components $\zeta_{\theta}$ evolves
\begin{equation}\label{m9}
\left\{ \begin{array}{ll}
\partial_{t}\zeta_{\theta}+v\cdot\nabla\zeta_{\theta}-\mu(\Delta-\frac{1}{r^2})\zeta_{\theta}=\frac{v_{r}}{r}\zeta_{\theta}-\partial_z\frac{(B_{\theta})^2}{r} ,
\\
{\zeta_{\theta}}_{\vert t=0}=\zeta_{\theta}^{0},
\end{array} \right.
\end{equation}
 which accomplish the proof of the proposition.
\end{proof}

For another properties of the vorticity in the axisymmetric context we state the following proposition, which its proof can be found  for example in \cite{AKH}.
\begin{prop}\label{prop3} 
Assume that $v$ is an axisymmetric vector field in divergence-free and $\omega=\nabla\land v$ its vorticity. Then the following properties hold true
\begin{itemize}
\item[{\bf(i)}]  The vector $\omega$ satisfies
\begin{equation*}
\vec{\omega}\land\vec{e}_{\theta}=\vec{0}.
\end{equation*}
\end{itemize}
In particular, for every $(x_{1}, x_{2}, z)$ in $\mathbb{R}^{3}$ we have
\begin{equation*}
\omega_{3}=0, \quad \omega_{1}(x_{1}, 0, z)=\omega_{2}(0, x_{2}, z)=0.
\end{equation*}
\begin{itemize}
\item[{\bf(ii)}] For every $q\ge-1, \;\Delta_{q}v$ is axisymmetric and
\begin{equation*}
\Delta_{q}\omega\wedge\vec{e}_{\theta}=\vec{0}.
\end{equation*}
\end{itemize}
\end{prop}

\section{A priori estimate}
This section comprises three important parts. The first one addresses to the energy estimates for different quantities like $v, B$ and $\frac{B_{\theta}}{r}$,\ldots, in various functional spaces. The second part, cares with an axisymmetric estimates for the same quantities, in particular, the vorticity $\omega$ in $L^\infty-$space and $\frac{v^r}{r}$ in Lorentz space $L^{3,1}$. In the third part we focus on the Lipschitz norm of the velocity by employing a special geometric decomposition of the vorticity.   

\hspace{0.5cm}In the sequel we will agree the following notation: we denote by $\Phi_k(t)$ any function of the form
$$ \Phi_k(t) = C_0  \underbrace{\exp(\ldots\exp}_{k\text{ times}}(C_0(  t^{\frac{5}{4}}))\ldots),$$
where $C_0$ depends on the involved norms of the initial data and its value varies from line up some absolute constants, we will make  an intensive use (without mentioning it) of the following trivial facts  
$$\int_0^t \Phi_k(\tau)d\tau \leq \Phi_k(t) \quad \textnormal{and}  \quad \exp\Big(  \int_0^t \Phi_k(\tau)d\tau\Big) \leq \Phi_{k+1}(t)  .$$
\subsection{Energy estimates }
\begin{prop}\label{prop:4.1}
Let $(v,B)$ be a smooth solution of \eqref{MHD(mu)}. The following assertions hold.
\begin{enumerate}
\item[{\bf(i)}] For $(v^0, B^0)\in L^2 \times L^2 $ and $t \geq 0$,  we have: 
\begin{equation*}
\|v(t)\|^2_{L^2} + \|B(t)\|^2_{L^2} + 2\mu \|\nabla v\|^2_{L^2_t L^2} + 2 \|\nabla B\|^2_{L^2_t L^2} = \|v^0\|^2_{L^2}+ \|B^0\|^2_{L^2}.
\end{equation*} 
\item[{\bf(ii)}] Assume that $(v,B)$ is an axisymmetric solution of        \eqref{MHD(mu)}. Then for every $t \geq 0$, we have:
\begin{equation*}
\bigg{\|} \frac{B_\theta}{r} \bigg{\|}^2_{L^{\infty}_{t} L^2} + \bigg{\|}\nabla \bigg{(}\frac{B_\theta}{r} \bigg{)} \bigg{\|}^2_{L^{2}_{t} L^2} \le \bigg{\|} \frac{{B^0_\theta}}{r} \bigg{\|}^2_{L^2}.
\end{equation*}
\item[{\bf(iii)}] Assume that $(v,B)$ is an  axisymmetric solution of     \eqref{MHD(mu)}. Then for every $p \in (1,\infty] ,q \in [1,\infty] $ and $ t \geq 0$,
\begin{equation*}
\bigg{\|} \frac{B_\theta}{r} \bigg{\|}_{L^\infty_t L^{p,q}} \le \bigg{\|} \frac{{B^0_\theta}}{r} \bigg{\|}_{L^{p,q}}.
\end{equation*}
\item[{\bf(iv)}] Assume that $(v,B)$ is an  axisymmetric solution of     \eqref{MHD(mu)} such that $(v^0, B^0)\in L^2 \times L^2  $ and $\frac{{B^0_\theta}}{r} \in L^m $ with $m>6.$ Then, 
\begin{itemize}
\item[{\bf(iv.a)}] for every $t > 0$, we have: 
\begin{equation*}
\| B_\theta(t) \|_{L^\infty}\le C_{0}\big( t^{-\frac{3}{4}}+t^{\frac{1}{4}} \big).
\end{equation*}
\item[{\bf(iv.b)}]  for every $p \in(2,\infty]$, and  $t\geq 0$, we have: 
\begin{equation*}
\| B_\theta \|_{L^{1}_{t} L^p}\le C_0 \big( t^{\frac{1}{4}}+t^{\frac{5}{4}} \big).
\end{equation*}
\end{itemize}
\end{enumerate}
\end{prop} 
\begin{proof}We will establish just {\bf(i)}. The proof of other assertions can be found in details in Proposition 4.1 of  \cite{Hassainia}. Performing $L^2$ inner product for the first equation of \eqref{MHD(mu)}, integrating by parts and using the fact that $\Div v=0$, then we have
$$\frac{d}{dt}\|v(t)\|^2_{L^2} + 2\mu\int_{\RR^3} |\nabla v(t,x)|^2dx  = -\int_{\RR^3} (B\cdot\nabla v) \cdot B  dx.$$
Similarly, an $L^2-$estimate of $B-$equation gives
$$\frac{d}{dt}\|B_{\mu}(t)\|^2_{L^2} +2\int_{\RR^3} |\nabla B(t,x)|^2dx  = \int_{\RR^3} (B\cdot\nabla v) \cdot B  dx.$$
Collecting the last two estimates, yields 
 \begin{equation*}
\frac{d}{dt}\|v(t)\|^2_{L^2} +\frac{d}{dt}\|B(t)\|^2_{L^2} + 2\mu \|\nabla v\|^2_{ L^2} + 2 \|\nabla B\|^2_{L^2} = 0.
\end{equation*} 
Finally integrating  the above estimate with respect to time we find the desired estimate.  
\end{proof}
\subsection{Axisymmetric estimates }\label{Axi- estimates}

\begin{prop}\label{prop:4.2} Let $\sigma \in(3,\infty)$ and $(v,B)$ be an axisymmetric smooth solution of \eqref{MHD(mu)} in divergence-free. Then the following assertions are hold.
\begin{enumerate}
\item[{\bf(i)}] For $(v^0,B^0) \in  L^2 \times L^2, \;(\omega^0_\theta /r,B^0_{\theta} /r) \in  L^{3,1} \times (L^{2}\cap L^\infty) $ and $t \ge 0$, we have:
$$\Big\|\frac{\omega}{r}\Big\|_{L^{\infty}_{t} L^{3,1}}+\Big\| \frac{B_{\theta}}{r}\Big\|_{L^{1}_{t}\mathcal{B}^{\frac{3}{2}}_{2,1}} +\Big\Vert\frac{v_{r}}{r}\Big\Vert_{L^\infty_t L^{\infty}} \le \Phi_1(t).$$
\item[{\bf(ii)}] For $(v^0,B^0,\omega^0) \in  L^2 \times \big(L^2\cap \mathcal{B}^{\frac{3}{\sigma}-1}_{\sigma,1}\big) \times L^\infty,\;(\omega^{0}_{\theta} /r, B^0_{\theta} /r) \in  L^{3,1} \times L^{2}\cap L^\infty $ and $t \ge 0$, we have:
$$\Vert  B \Vert_{L^{\infty}_{t} \mathcal{B}^{\frac{3}{\sigma}-1}_{\sigma,1}}+ \Vert  B \Vert_{L^{1}_{t} \mathcal{B}^{\frac{3}{\sigma}+1}_{\sigma,1}} +\Vert\omega\Vert_{L^{\infty}_{t}  L^{\infty}} +\Vert  B \Vert_{L^{2}_{t} \mathcal{B}^{0}_{\infty,1}}  \le \Phi_2(t).$$
\item[{\bf(iii)}] For $(v^0,B^0,\omega^0) \in  (L^2\cap L^\infty \times \big(L^2\cap \mathcal{B}^{\frac{3}{\sigma}-1}_{\sigma,1}\big)  \times L^\infty,\;(\omega^{0}_{\theta} /r  ,B^0_{\theta} /r) \in  L^{3,1} \times (L^{2}\cap L^{\infty}) $ and  $t \ge 0$, we have: 
$$\Vert v\Vert_{L^{\infty}_{t}L^{\infty}}\le  \Phi_3(t).$$
\end{enumerate}
\end{prop}
\begin{proof} {\bf (i)} Setting $\Omega=\frac{\omega_{\theta}}{r}$ and $\Sigma=\frac{B_{\theta}}{r}$. We check that $\Omega$ satisfies the following inhomogeneous parabolic equation:
\begin{equation*}
\partial_{t}\Omega+v\cdot\nabla\Omega-\mu(\Delta+\frac{2}{r}\partial_{r})\Omega=-\partial_{z}\Sigma^2  .
\end{equation*}
The dissipative term has a good sign and thus we have for all $p\in[1,\infty]$ (see, also \cite{Ukhovskii-Yudovich})
\begin{equation*}
\|\Omega(t)\|_{L^p}\leq\|\Omega^{0}\|_{L^p}+ \int_0^t \|\partial_z\Sigma^2(\tau)\|_{L^p}d\tau.
\end{equation*}
By a real interpolation we get for $1<p<\infty$ and $q\in[1,\infty],$
\begin{equation*}
\|\Omega(t)\|_{L^{p,q}}\leq\|\Omega^{0}\|_{L^{p,q}}+ \int_0^t \|\partial_z\Sigma^2(\tau)\|_{L^{p,q}}d\tau.
\end{equation*}
So, for $p=3,q=1$, combine $\partial_z\Sigma^2=2\Sigma \partial_z \Sigma $ with {\bf(iii)-}Definition \ref{Def:7.1} yield
 \begin{equation}
\|\Omega(t)\|_{L^{3,1}}\leq\|\Omega^{0}\|_{L^{3,1}}+ 2\int_0^t \|\Sigma(\tau)\|_{L^{\infty}} \|\partial_z\Sigma(\tau)\|_{L^{3,1}}d\tau.
\end{equation}
Since $\|\Sigma(\tau)\|_{L^{\infty}} \le \|\Sigma^0\|_{L^{\infty}} $, so from {\bf{(iii)}}-Proposition \ref{prop:4.1}, embedding $ {\mathcal{B}^{\frac{3}{2}-1}_{2,1}} \hookrightarrow L^{3,1}$ and the continuity of $\nabla :\mathcal{B}^{\frac{3}{2}}_{2,1} \rightarrow \mathcal{B}^{\frac{3}{2}-1}_{2,1}$, we get 
 \begin{equation}\label{Eq:4.10}
\|\Omega(t)\|_{L^{3,1}}\leq\|\Omega^{0}\|_{L^{3,1}}+ 2\|\Sigma^0\|_{L^{\infty}}\int_0^t  \|\Sigma(\tau)\|_{\mathcal{B}^{\frac{3}{2}}_{2,1}}d\tau.
\end{equation}
To treat the term $ \|\Sigma\|_{L^1_t\mathcal{B}^{\frac{3}{2}}_{2,1}}$, we localize in frequency the $\Sigma-$equation by taking $\Sigma_q= \Delta_q \Sigma$  for $q\in \mathbb{N}$. Then we obtain
$$ \Big(\partial_t +v \cdot \nabla -\big(\Delta+\frac{2}{r}\partial_{r}\big)\Big) \Sigma_q = -[\Delta_q,v \cdot \nabla ]\Sigma. $$
Multiplying the above equation by $|\Sigma_q|^{p-2}\Sigma_q$, integrating over $\RR^3$. On account $\Div v=0$ and H\"older's inequality, we find 
\begin{equation*}
\frac{1}{p}\frac{d}{dt}\|\Sigma_q(t) \|^p_{L^p} - \int_{\RR^3} (\Delta  \Sigma_q ) |\Sigma_q|^{p-2}\Sigma_q  dx \le\|\Sigma_q(t) \|^{p-1}_{L^p}   \| [\Delta_q,v\cdot \nabla]\Sigma(t) \|_{L^{p}}  .
\end{equation*}
The fact  $-\int_{\RR^3} \frac{\partial_r \Sigma_q}{r}   |\Sigma_q|^{p-2}\Sigma_q dx \geq0  $ and the generalized Bernstein's inequality, see \cite{Chen.Miao.Zhang, Danchin},
\begin{equation}\label{Bs-G}
- \int_{\RR^3} (\Delta f_q ) |f_q|^{p-2}f_q  dx \ge
\left\{\begin{array}{ll}
c2^{2q} \| f_q(t)\|^p_{L^p} \quad& \textrm{if $q \ge 0,$}  \\ \quad  0    &\textrm{if $q =-1$}.
\end{array}
\right.
\end{equation}
give 
$$ \frac{d}{dt} \| \Sigma_q(t) \|_{L^p} +c 2^{2q}\| \Sigma_q(t) \|_{L^p}  \lesssim   \| [\Delta_q,v\cdot \nabla]\Sigma(t) \|_{L^{p}}  .$$
It follows 
$$\frac{d}{dt} \Big( e^{c t2^{2q}} \| \Sigma_q(t) \|_{L^p}\Big) v\lesssim e^{c t2^{2q}}   \| [\Delta_q,v\cdot \nabla]\Sigma(t) \|_{L^{p}}   .$$ 
This implies that 
$$  \| \Sigma_q(t) \|_{L^p}   \lesssim  e^{-c t2^{2q}} \|\Sigma_q(0)\|_{L^p} + \int_0^t e^{-c (t-\tau)2^{2q}} \|[\Delta_q,v\cdot \nabla]\Sigma(\tau)\|_{L^{p}} d\tau  .$$
Taking the $L^1_t-$norm to obtain
$$  \| \Sigma_q \|_{L^\infty_tL^p} +c 2^{2q}\| \Sigma_q\|_{L^{1}_{t} L^p}  \lesssim   \| \Sigma_q(0)\|_{L^p}   +\|[\Delta_q,v\cdot \nabla]\Sigma \|_{L^{1}_{t} L^{p}}  .$$ 
Multiplying both sides of above inequality by $2^{-2q}$, leading to
\begin{equation}\label{Eq:4-12}
 \| \Sigma_q\|_{L^{1}_{t} L^p}  \lesssim   2^{-2q} \| \Sigma_q(0)\|_{L^p}   +2^{-2q}\|[\Delta_q,v\cdot \nabla]\Sigma \|_{L^{1}_{t}L^{p}}.
 \end{equation}  
For $p=2$, \eqref{Eq:4-12} takes the form
\begin{equation*}
 \| \Sigma_q\|_{L^{1}_{t} L^2}  \lesssim   2^{-2q} \| \Sigma_q(0)\|_{L^2}+2^{-2q}\|[\Delta_q,v\cdot \nabla]\Sigma \|_{L^{1}_{t}L^{2}}.
 \end{equation*}  
So, by definition of $\mathcal{B}^{\frac32}_{2,1}$ and Proposition \ref{prop:2.1} in appendix, it happens 
\begin{eqnarray*}
\| \Sigma\|_{L^1_t\mathcal{B}^{\frac{3}{2}}_{2,1}} &\le&   \|\Delta_{-1} \Sigma \|_{L^{1}_{t}L^2} + \sum_{q \geq 0} 2^{q \frac{3}{2}} \| \Sigma_q\|_{L^{1}_{t}L^2} \\ &\lesssim& \| \Sigma \|_{L^{1}_{t}L^2} +   \| \Sigma^{0}\|_{L^2} \sum_{q \geq 0}   2^{-q \frac{1}{2}} +\sum_{q \geq 0}   2^{-q \frac{1}{2}}  \|[\Delta_q,v\cdot \nabla]\Sigma \|_{L^{1}_{t} L^{2}}  \\ &\lesssim& \| \Sigma \|_{L^{1}_{t}L^2} +   \| \Sigma^{0}\|_{L^2}    + \int_0^t \| \Omega(\tau)\|_{L^{3,1}}  \big(\| B_\theta(\tau) \|_{L^6} +\| \Sigma(\tau)\|_{ L^2}\big)d\tau .
\end{eqnarray*}
Furthermore, from {\bf{(ii)}-}Proposition \ref{prop:4.1}, we get 
\begin{eqnarray}\label{Eq:4.13}
\| \Sigma\|_{L^{1}_{t}\mathcal{B}^{\frac{3}{2}}_{2,1}} &\le& \| \Sigma^0\|_{L^2}(1+t) + \int_0^t \| \Omega(\tau)\|_{L^{3,1}} \Big(\| B(\tau) \|_{L^6} +\| \Sigma^0\|_{L^2}  \Big)d\tau .
\end{eqnarray}
Inserting the above estimate in \eqref{Eq:4.10}, we thus have 
\begin{eqnarray*}
\|\Omega(t)\|_{L^{3,1}} &\le&   C_0(1+t) +C_0 \int_0^t \| \Omega(\tau)\|_{L^{3,1}}  \big( \| B(\tau) \|_{L^6} +1\big)d\tau .
\end{eqnarray*}
Gronwall's inequality combined with Proposition \ref{prop:4.1} ensures that
\begin{equation}\label{Eq:4.14}
\|\Omega(t)\|_{L^{3,1}} \le \Phi_1(t).
\end{equation}
Since $\| \frac{\omega_\theta}{r}\vec e_\theta\|_{L^{3,1}}=\|\frac{\omega}{r}\|_{L^{3,1}},$ we infer that 
\begin{equation}\label{Eq:4.8}
\Big\|\frac{\omega(t)}{r}\Big\|_{L^{3,1}} \le \Phi_1(t).
\end{equation}
Substituting \eqref{Eq:4.14} in \eqref{Eq:4.13}, it follows that 
\begin{equation}\label{Eq:4.9}
\| \Sigma\|_{L^1_t\mathcal{B}^{\frac{3}{2}}_{2,1}} \le \Phi_1(t).
\end{equation}
On the other hand, due to T. Shirota and T. Yanagisawa \mbox{\cite{Sh-Ya}}, we have 
\begin{equation}
\Big|\frac{v^r}{r}\Big|\lesssim \frac{1}{{\vert . \vert}^{2}}\star \Big\vert\frac{\omega_{\theta}}{r}\Big\vert.
\end{equation} 
As $\frac{1}{{\vert . \vert}^{2}}\in L^{\frac{3}{2}, \infty},$ then from the convolution laws $L^{p, q}\star L^{p', q'}\rightarrow L^{\infty}$, we have
\begin{equation}\label{Eq:4-10}
\Big\Vert\frac{v_{r}}{r}\Big\Vert_{L^\infty_t  L^{\infty}}\lesssim \Vert \Omega(t)\Vert_{L^{3, 1}}\le \Phi_1(t) .
\end{equation}
We collect \eqref{Eq:4.8}, \eqref{Eq:4.9} and \eqref{Eq:4-10} we find {\bf(i)}.\\
{\bf (ii)} Since $\omega= \omega_{\theta}\vec e_{\theta}$, then from the first equation of \eqref{VE(mu-1)}, the vorticity  $\omega$ satisfies 
\begin{equation}\label{Eq:Om}
\left\{ \begin{array}{ll}
\partial_{t}\omega+v\cdot\nabla\omega-\mu\Delta\omega=\dfrac{v_{r}}{r}\omega -\partial_z \big( \Sigma B\big) ,\\
{\omega}_{\vert t=0}=\omega^{0}.
\end{array} \right.
\end{equation}
From maximum's principle, {\bf(iii)-}Proposition \ref{prop:4.1} and the embedding $ \mathcal{B}^{\frac{3}{\sigma}+1}_{\sigma,1} \hookrightarrow \Lip$ yield  
\begin{eqnarray}\label{Es :Om-L3.1}
\Vert\omega(t)\Vert_{L^{\infty}}&\le&\Vert\omega^{0}\Vert_{L^{\infty}}+\int_0^t\| \partial_z \big( \Sigma(\tau)B(\tau)\big)\Vert_{L^{\infty}}d\tau +\int_{0}^{t}\Big\Vert\frac{v_{r}}{r}(\tau)\omega(\tau)\Big\Vert_{L^{\infty}}d\tau \nonumber  \\
&\lesssim&\Vert\omega^{0}\Vert_{L^{\infty}}+\int_0^t \|\partial_z  B(\tau)\|_{L^\infty} \| \Sigma (\tau)\|_{L^\infty} d\tau +\int_{0}^{t} \Big\Vert\frac{v}{r} (\tau)\Big\Vert_{L^{\infty}} \Vert\omega(\tau)\Vert_{L^{\infty}}d\tau \nonumber\\
&\lesssim& \Vert\omega^{0}\Vert_{L^{\infty}}+\|\Sigma^0\|_{L^\infty}  \| B\|_{L^1_t \mathcal{B}^{\frac{3}{\sigma}+1}_{\sigma,1} }  +\int_{0}^{t} \Big\Vert\frac{v}{r} (\tau)\Big\Vert_{L^{\infty}} \Vert\omega(\tau)\Vert_{L^{\infty}}d\tau  .
\end{eqnarray}
To bound the term $\| B\|_{L^1_t \mathcal{B}^{\frac{3}{\sigma}+1}_{\sigma,1}}$, first, we combine $B=B_{\theta} \vec{e}_{\theta}$ with the second equation of \eqref{VE(mu-1)}, we find that $B$ satisfies 
\begin{equation*}
\partial_{t}B+v\cdot\nabla B-\Delta B=\dfrac{B}{r}v_{r}.
\end{equation*}
Second, we localize in frequency the above equation by taking  $\Delta_q  B \triangleq B_q$ and  $G_q \triangleq \Delta_q ( B\cdot \nabla v) =\Delta_q ( \dfrac{B}{r}v_{r})$, for $q \in \mathbb{N}\cup \{-1\}$. Thus, we have
$$ (\partial_t +v \cdot \nabla -\Delta) B_q =G_q -[\Delta_q,v \cdot \nabla ]B_q. $$ 
Multiplying by $|B_q|^{p-2}B_q$. So, H\"older's inequality tells   
\begin{equation*}
\frac{1}{p}\frac{d}{dt}\|B_q(t) \|^p_{L^p} - \int_{\RR^3} (\Delta  B_q ) |B_q|^{p-2}B_q  dx \le\|B_q(t) \|^{p-1}_{L^p} \big( \|G_q(t)\|_{L^{p}} + \| [\Delta_q,v\cdot \nabla]B(t) \|_{L^{p}} \big).
\end{equation*}
Then from \eqref{Bs-G}, it follows that
$$
\frac{d}{dt} \| B_q(t) \|_{L^p} +c2^{2q}\| B_q(t) \|_{L^p}  \lesssim  \|G_q(t)\|_{L^{p}} + \| [\Delta_q,v\cdot \nabla]B(t) \|_{L^{p}}. $$
This gives 
$$
\frac{d}{dt} \Big( e^{ct2^{2q}} \| B_q(t) \|_{L^p}\Big) \lesssim e^{ct2^{2q}} \Big( \|G_q(t)\|_{L^{p}} + \| [\Delta_q,v\cdot \nabla]B(t) \|_{L^{p}} \Big).
$$ 
Consequently,
$$\| B_q(t) \|_{L^p}   \lesssim  e^{-c t2^{2q}} \| B_q(0)\|_{L^p} + \int_0^t e^{-c(t-\tau)2^{2q}} \Big( \|G_q(\tau)\|_{L^{p}}  +\|[\Delta_q,v\cdot \nabla]B(\tau)\|_{L^{p}}\Big) d\tau.
$$
By the classical Young's convolution inequality in time, it holds
\begin{equation}\label{Eq:4.21}
\| B_q \|_{L^\infty_tL^p} +2^{\frac{2q}{\eta}}\| B_q\|_{L^\eta_t L^p}  \lesssim   \| B_q(0)\|_{L^p} + \|G_q\|_{L^1_t L^{p}}  +\|[\Delta_q,v\cdot \nabla]B\|_{L^1_t L^{p}}  .
\end{equation}
For the term $\|G_q\|_{L^{1}_{t}L^{p}}$, we explore the continuity  of $\Delta_q$ on $L^p$ into itself and  H\"older's inequality in time to deduce that 
$$  
\| B_q \|_{L^\infty_tL^p} + 2^{\frac{2q}{\eta}} \| B_q\|_{L^{\eta}_t L^p}  \lesssim   \| B_q(0)\|_{L^p} +\Big{\|}\frac{v^r}{r}\Big{\|}_{L^\infty_{t} L^{\infty}} \|B\|_{{L^1_t } L^{p}}  +\|[\Delta_q,v\cdot \nabla]B\|_{L^1_t L^{p}}.
$$ 
For the commutator term in the r.h.s. we make use Proposition's \ref{prop:7.1} in appendix to get
\begin{align}\label{4.22}
\| B_q \|_{L^\infty_tL^p} + 2^{\frac{2q}{\eta}} \| B_q\|_{L^\eta _t L^p}  \lesssim &   \| B_q(0)\|_{L^p} + \| B\|_{L^{1}_t L^p} \Big( \Big{\|}\frac{v^r}{r}\Big{\|}_{L^{\infty}_t L^{\infty}} +\|v\|_{L^{\infty} _t L^{2}}\Big) \nonumber \\+&  (q+2)  \int_0^t \| B(\tau)\|_{L^p} \|\omega(\tau) \|_{ L^{\infty}}d \tau .
\end{align}  
In particular, for $ \eta =1$  and $p=\sigma>3$, we obtain 
\begin{align*}
\| B\|_{L^\infty_t \mathcal{B}^{\frac{3}{\sigma}-1}_{\sigma,1}}+  \| B\|_{L^1_t\mathcal{B}^{ \frac{3}{\sigma}+1}_{\sigma,1}} \lesssim&  \|\Delta_{-1} B\|_ {L^1_tL^\sigma} +   \| B^0\|_{ \mathcal{B}^{\frac{3}{\sigma}-1}_{\sigma,1}}\\& +\| B\|_{L^{1}_t L^\sigma} \Big( \|v\|_{L^{\infty}_t L^{2}} +\Big{\|}\frac{v^r}{r}\Big{\|}_{L^{\infty}_t L^{\infty}} \Big)\Big(\sum_{q\geq 0} 2^{q(\frac{3}{\sigma}-1)} (q+2) \Big) \\&+   \int_0^t \| B(\tau)\|_{L^\sigma} \|\omega(\tau) \|_{ L^{\infty}}d \tau  \Big(\sum_{q\geq 0} 2^{q(\frac{3}{\sigma}-1)} (q+2) \Big).   
\end{align*}
The fact that $\sigma>3$, the serie $\sum_{q\geq 0} 2^{q(\frac{3}{\sigma}-1)} (q+2)$ converges, we find that
\begin{eqnarray}\label{Es:B-B}
\| B\|_{L^\infty_t \mathcal{B}^{\frac{3}{\sigma}-1}_{\sigma,1}}+  \| B\|_{L^1_t\mathcal{B}^{ \frac{3}{\sigma}+1}_{\sigma,1}} &\lesssim&   \| B^0\|_{ \mathcal{B}^{\frac{3}{\sigma}-1}_{\sigma,1}} +    \| B\|_{L^{1}_t L^\sigma} \Big(1+\|v\|_{L^{\infty}_t L^{2}} +\Big{\|}\frac{v^r}{r}\Big{\|}_{L^{\infty}_t L^{\infty}} \Big)\\
&&+\nonumber\int_0^t \| B(\tau)\|_{L^\sigma} \|\omega(\tau) \|_{ L^{\infty}}d \tau,
\end{eqnarray}
comined with \eqref{Es :Om-L3.1}, we get 
\begin{align*}
\Vert\omega(t)\Vert_{L^{\infty}}&\le   C_{0}+ C_{0} \bigg(1+\| B\|_{L^{1}_t L^\sigma} \Big(1+\|v\|_{L^{\infty}_t L^{2}} +\Big{\|}\frac{v^r}{r}\Big{\|}_{L^{\infty}_t L^{\infty}}   \Big)\bigg)\\& +\int_{0}^{t}\Big(  \Big\Vert\frac{v^r}{r} (\tau)\Big\Vert_{L^{\infty}}+ \| B(\tau)\|_{L^\sigma} \Big) \Vert\omega(\tau)\Vert_{L^{\infty}}d\tau .
\end{align*}
Gronwall's inequality leads to
\begin{align*}
\Vert\omega(t)\Vert_{L^{\infty}} \le  C_{0} \bigg(1+\| B\|_{L^{1}_t L^\sigma} \Big(1+\|v\|_{L^{\infty}_t L^{2}} +\Big{\|}\frac{v^r}{r}\Big{\|}_{L^{\infty}_t L^{\infty}}   \Big)\bigg) e^{C_0\big(  \Vert\frac{v}{r}\Vert_{L^1_tL^{\infty}}+ \| B\|_{L^1_tL^\sigma} \big)} .
\end{align*}
Thanks to {\bf (i)}, {\bf (iv.b)}-Proposition \ref{prop:4.1} and \eqref{Eq:4-10}, we deduce that
\begin{eqnarray}\label{Eq:4.12}
\Vert \omega(t)\Vert_{L^{\infty}} \le \Phi_2(t).
\end{eqnarray}
Inserting the last estimate in \eqref{Es:B-B}, so by employing {\bf (i)}, {\bf (iv.b)}-Proposition \ref{prop:4.1} and  \eqref{Eq:4-10},\eqref{Eq:4.12}, we end up with 
\begin{eqnarray}\label{Es:B,b}
\| B\|_{L^\infty_t \mathcal{B}^{\frac{3}{\sigma}-1}_{\sigma,1}}+  \| B\|_{L^1_t\mathcal{B}^{ \frac{3}{\sigma}+1}_{\sigma,1}}  \le \Phi_2(t).
\end{eqnarray}
To close our claim, we must estimate $\| B\|_{L^2_t \mathcal{B}^{0}_{\infty,1}}$. For this purpose, we combine Bernstein's inequality for $p=\sigma >3 ,$ with \eqref{4.22} for $\eta=2$ to obtain 
\begin{eqnarray*}
\| B\|_{L^2_t \mathcal{B}^{0}_{\infty,1}}  &\le &   \|\Delta_{-1} B \|_ {L^2_tL^\infty} +\sum_{q\geq 0} 2^{q} \|\Delta_{q} B\|_ {L^2_tL^\infty} \\ 
 &\lesssim&  \| B\|_ {L^2_tL^2}  +\sum_{q\geq 0} 2^{q\frac{3}{\sigma}} \|\Delta_{q} B\|_ {L^2_tL^\sigma}   \\ &\lesssim&  \| B\|_ {L^2_tL^2}  +  \| B^0\|_{ \mathcal{B}^{\frac{3}{\sigma}-1}_{\sigma,1}} +    \| B\|_{L^{1}_t L^\sigma} \Big(\|v\|_{L^{\infty}_t L^{2}} +\Big{\|}\frac{v^r}{r}\Big{\|}_{L^{\infty}_t L^{\infty}} + \Vert \omega\Vert_{L^\infty_t L^{\infty}} \Big).
\end{eqnarray*} 
Finally, in view of {\bf (i)},{\bf (iv.b)}-Proposition \ref{prop:4.1} and \eqref{Eq:4-10}, \eqref{Eq:4.12}, we conclude that 
\begin{eqnarray}\label{Est:B-0}
\| B\|_{L^2_t \mathcal{B}^{0}_{\infty,1}}  \le \Phi_2(t).
\end{eqnarray}
The desired estimate is then proved.\\
{\bf (iv)} This item will be done by using an argument of \mbox{P. Serfati \cite{Serfati}}. From homogeneous  Littlewood-Paley  decomposition, we write
\begin{equation*}
\Vert v(t)\Vert_{L^{\infty}}\le \Vert \dot{S}_{-N}v\Vert_{L^{\infty}}+\sum_{q\ge -N}\Vert \dot{\Delta}_qv\Vert_{L^{\infty}},
\end{equation*}
where $N$ is a parameter that  will be judiciously  chosen later. Using the fact $2^{q} \Vert \dot{\Delta}_q \omega \Vert_{L^{\infty}} \approx \Vert \dot{\Delta}_q v\Vert_{L^{\infty}} $ for $q\in \mathbb{N}$  we get,
\begin{equation}\label{m10}
\sum_{q\ge -N}\Vert \dot{\Delta}_q v(t)\Vert_{L^{\infty}}\lesssim \sum_{q\ge -N} 2^{-q}\Vert \dot{\Delta}_q \omega(t) \Vert_{L^{\infty}} \lesssim 2^{N}\Vert\omega(t)\Vert_{L^{\infty}}.
\end{equation}
Since $ \dot{S}_{-N}v$ satisfies the equation,
\begin{equation*}
(\partial_{t}-\mu\Delta)\dot{S}_{-N}v=\dot{S}_{-N}\mathbb{P}\big( B\cdot\nabla B\big)- \dot{S}_{-N}\mathbb{P}\big(v\cdot\nabla v\big),
\end{equation*}
where $\mathbb{P}$ refers to Leray's operator. Then we get 
\begin{eqnarray*}
\Vert \dot{S}_{-N}v(t)\Vert_{L^{\infty}}&\le& \Vert \dot{S}_{-N}v^{0}\Vert_{L^{\infty}}+\int_{0}^{t} \Big(\Vert\dot{S}_{-N}\mathbb{P}(v\cdot\nabla v)(\tau) \Vert_{L^{\infty}}+\Vert\dot{S}_{-N}\mathbb{P}(B\cdot\nabla B)(\tau)\Vert_{L^{\infty}}\Big)d\tau
\\&\le& \Vert \dot{S}_{-N}v^{0}\Vert_{L^{\infty}} +\sum_{q\le -N}\int_{0}^{t}\Big( \Vert\dot{\Delta}_q \mathbb{P}(v\cdot\nabla v)(\tau) \Vert_{L^{\infty}} +\Vert \dot{\Delta}_q  \mathbb{P}(B\cdot\nabla B)(\tau)\Vert_{L^{\infty}} \Big) d\tau,
\end{eqnarray*}
Using the fact $v\cdot\nabla v= \Div( v\otimes v), B\cdot\nabla B= \Div( B\otimes B) $ and Bernstein's inequality yield
\begin{eqnarray*}
\Vert \dot{S}_{-N}v(t)\Vert_{L^{\infty}}&\lesssim&  \Vert \dot{S}_{-N}v^{0}\Vert_{L^{\infty}} +\sum_{q\le -N} 2^{q} \int_{0}^{t}\Big( \Vert \dot{\Delta}_q \mathbb{P}( v\otimes v)(\tau)\Vert_{L^{\infty}}d\tau+\int_{0}^{t}\Vert \dot{\Delta}_q  \mathbb{P}( B\otimes B)(\tau)\Vert_{L^{\infty}} \Big) d\tau.
\end{eqnarray*} 
The fact that $\dot\Delta_{q}\mathbb{P}$ maps continuously $L^\infty$ into itself. Thus we deduce,
\begin{equation*}\label{m11}
\Vert \dot{S}_{-N}v\Vert_{L^{\infty}}\lesssim \Vert v^{0}\Vert_{L^{\infty}}+2^{-N}\int_{0}^{t}\Vert v(\tau)\Vert_{L^{\infty}}^2d\tau +\int_{0}^{t}\Vert B(\tau)\Vert_{L^{\infty}}^2d\tau,
\end{equation*}
combined with (\ref{m10}), we get
\begin{equation}\label{m:12}
\Vert v(t)\Vert_{L^{\infty}}\lesssim \Vert v^{0}\Vert_{L^{\infty}}+2^{N}\Vert\omega(t) \Vert_{L^{\infty}}+2^{-N}\int_{0}^{t}\Vert v(\tau)\Vert_{L^{\infty}}^2d\tau+\int_{0}^{t}\Vert B(\tau)\Vert_{L^{\infty}}^2d\tau. 
\end{equation}
We choose $N$ such that
\begin{equation*}
2^{2N}\approx 1+\Vert\omega(t)\Vert_{L^{\infty}}^{-1}\int_{0}^{t}\Vert v(\tau)\Vert_{L^{\infty}}^2d\tau,
\end{equation*}
Then the estimate (\ref{m:12}) becomes
\begin{equation*}
\Vert v(t)\Vert_{L^{\infty}}^{2}\lesssim\Vert v^{0}\Vert_{L^{\infty}}^{2} +\bigg(\int_{0}^{t}\Vert B(\tau)\Vert_{L^{\infty}}^2d\tau \bigg)^2+\Vert\omega(t){\Vert}^2_{L^{\infty}}+\Vert\omega(t)\Vert_{L^{\infty}}\bigg(\int_{0}^{t}\Vert v(\tau)\Vert_{L^{\infty}}^{2}d\tau \bigg).
\end{equation*}
Granwall's inequality enubles us to write
\begin{equation*}\label{m13}
\Vert v(t)\Vert_{L^{\infty}}^{2} \lesssim e^{Ct\Vert\omega\Vert_{L^{\infty}_tL^\infty}} \Big(\Vert v^{0}\Vert_{L^{\infty}}^{2}+\| B\|^4_{L^2_t \mathcal{B}^{0}_{\infty,1}}+\Vert\omega(t){\Vert}^2_{L^{\infty}} \Big).
\end{equation*}
Finally, \eqref{Eq:4.12} and \eqref{Est:B-0} yield 
\begin{equation*}
\Vert v(t)\Vert_{L^{\infty}}\le \Phi_4(t).
\end{equation*}
The proof is now completed.
\end{proof}

\subsection{Vorticity decomposition and Lipschitz bound}\label{Lipschitz-V}
The following result is the principal step to bounding the Lipschitz norm of the velocity.  We will establish a new decomposition of the vorticity based on the special structure of axisymmetric flows.  We notice that this result was first proved for the Euler case in \cite{AKH} and generalized later in \cite{hz0} for the viscous case uniformly to the viscosity. Roughly, we will prove the following result.
\begin{prop}\label{Prop:4.3} Let $\omega$ be the vorticity of the viscous axisymmetric solution. Then there exists a decomposition $\{\widetilde{\omega}_{q}\}_{q\ge -1}$ of the vorticity $\omega$ such that for every $t\in\RR_{+},$
\begin{itemize}
\item[{\bf(i)}]$\omega(t, x)=\sum_{q\ge-1}\widetilde{\omega}_{q}(t, x)$.
\item[{\bf(ii)}]$\Div\widetilde\omega_{q}(t, x)=0$.
\item[{\bf(iii)}]$\forall q\ge-1,\; \Vert\widetilde\omega_{q}(t)\Vert_{L^{\infty}}\le\Big(\Vert\Delta_{q}\omega^{0}\Vert_{L^{\infty}}+2^{q}\Big\|\Delta_{q}\Big(\frac{B_{\theta}}{r}B\Big)\Big\|_{L^1_t L^\infty}\Big)\Phi_2(t) $.
\item[{\bf(iv)}] There exists a constant $C>0$ independent on the viscosity such that for every $k, q\ge-1$
\end{itemize}
\begin{equation*}
\Vert\Delta_{k}\widetilde\omega_{q}(t)\Vert_{L^{\infty}}\le C2^{|q-k|}e^{CZ(t)}\Big(\Vert\Delta_{q}\omega^{0}\Vert_{L^{\infty}}+2^{q}\Big\|\Delta_{q}\Big(\frac{B_{\theta}}{r}B^{1}\Big)\Big\|_{L^{1}_{t}  L^\infty}\Big).
\end{equation*}
\end{prop}
\begin{proof}
{\bf(i)} The main idea is to linearize the vortitity's equation. To do so, let $q\ge -1$ and define $\widetilde\omega_{q}$ as the solution of the following linear Cauchy problem
\begin{equation}\label{m14}
\left\{ \begin{array}{ll}
\partial_{t}\widetilde\omega_{q}-\mu\Delta\widetilde\omega_{q}+(v\cdot\nabla)\widetilde\omega_{q}=\widetilde\omega_{q}\cdot\nabla v-\partial_{z}\Big(\Delta_{q}\Big(\frac{B_{\theta}}{r}B\Big)\Big). \\
\widetilde\omega_{q}\vert_{t=0}=\Delta_{q}\omega^{0}.
\end{array} \right.
\end{equation}

The fact that the cut-off operators commutes with $\Div$ operator, that is $\Div(\Delta_q \omega^0)=0$, one deduce in view of {\bf(ii)}-Proposition \ref{Prop-1} for $t\in\RR_{+}$ that $\Div(\widetilde{\omega}_q)=0$. Hence the linearity and uniqueness enable us to write $\omega(t, x)=\sum_{q\ge-1}\widetilde\omega_{q}(t, x)$. Taking advantage to {\bf(ii)}-Proposition \ref{prop3}, we have $\Delta_{q}\omega^0\wedge\vec{e}_{\theta}=\vec{0}$. Then {\bf(ii)}-Proposition \ref{Prop-1} implies that this  property is preserved through the time and
\begin{equation}\label{m15}
\left\{ \begin{array}{ll}
\partial_{t}\tilde\omega_{q}-\mu\Delta\widetilde\omega_{q}+(v\cdot\nabla)\widetilde\omega_{q}=\frac{v^r}{r}\widetilde\omega_{q}-\partial_{z}\Big(\Delta_{q}\Big(\frac{B_{\theta}}{r}B\Big)\Big), \\
\widetilde\omega_{q}\vert_{t=0}=\Delta_{q}\omega_{0}.
\end{array} \right.
\end{equation}

Maximum's principle leads to
\begin{eqnarray*}
\Vert\widetilde\omega_{q}(t)\Vert_{L^{\infty}}&\le&\Vert\Delta_{q}\omega^{0}\Vert_{L^{\infty}}+\int_{0}^{t}\Big\Vert \frac{v^{r}}{r}{(\tau)}\Big\Vert_{L^{\infty}}\Vert \widetilde{\omega}_q{(\tau)}\Vert_{L^{\infty}}d\tau+\int_{0}^{t}\Big\|\partial_{z}\Big(\Delta_{q}\Big(\frac{B_{\theta}}{r}B\Big)\Big)(\tau)\Big\|_{L^\infty}d\tau  .
\end{eqnarray*}
Account, {\bf(i)}-Proposition \ref{prop:4.2} and Bernstein lemma, one has 
\begin{eqnarray*}
\Vert\widetilde\omega_{q}(t)\Vert_{L^{\infty}}&\le&\Vert\Delta_{q}\omega^{0}\Vert_{L^{\infty}} +\Phi_{1}(t) \int_{0}^{t}\Vert\widetilde\omega_{q}(\tau)\Vert_{L^{\infty}}d\tau+2^{q}\int_{0}^{t}\Big\|\Delta_{q}\Big(\frac{B_{\theta}}{r}B\Big)(\tau)\Big\|_{L^\infty}d\tau   .
\end{eqnarray*}
Gronwall's inequality leads 
\begin{equation*}
\Vert\widetilde\omega_{q}(t)\Vert_{L^{\infty}}\le \Phi_2(t) \bigg(\Vert\Delta_{q}\omega^{0}\Vert_{L^{\infty}}+2^{q}\Big\|\Delta_{q}\Big(\frac{B_{\theta}}{r}B\Big)\Big\|_{L^1_t L^\infty}\bigg).
\end{equation*}
The item {\bf(iv)} will be done in two-steps as below. 
\begin{equation}\label{m16}
\Vert\Delta_{k}\widetilde\omega_{q}(t)\Vert_{L^{\infty}}\le C2^{ k-q}e^{CZ(t)}\bigg(\Vert\Delta_{q}\omega_{0}\Vert_{L^{\infty}}+2^{q}\Big\|\Delta_{q}\Big(\frac{B_{\theta}}{r}B\Big)\Big\|_{L^1_t L^\infty}\bigg).
\end{equation}
and
\begin{equation}\label{m17}
\Vert\Delta_{k}\widetilde\omega_{q}(t)\Vert_{L^{\infty}}\le C2^{q-k}\bigg(\Vert\Delta_{q}\omega_{0}\Vert_{L^{\infty}}+2^{q}\Big\|\Delta_{q}\Big(\frac{B_{\theta}}{r}B\Big)\Big\|_{L^1_t L^\infty}\bigg).
\end{equation}
First, we handle with \eqref{m16} by employing Proposition \ref{Prop:1.1} in limit case $s=-1$ to \eqref{m14}, one obtains
\begin{eqnarray}\label{Eq:4,21}
e^{-CZ(t)}\Vert\widetilde\omega_{q}\Vert_{\mathscr{B}_{\infty, \infty}^{-1}}&\lesssim& \Vert\Delta_{q}\omega^{0}\Vert_{\mathscr{B}_{\infty, \infty}^{-1}}+\int_{0}^{t}e^{-CZ(\tau)}\Vert \widetilde\omega_{q}\cdot\nabla v(\tau)\Vert_{\mathscr{B}_{\infty, \infty}^{-1}}d\tau\\
&&+\int_{0}^{t}e^{-CZ(\tau)}\Big\|\partial_z\Big(\Delta_{q}\Big(\frac{B_{\theta}}{r}B\Big)\Big)(\tau)\Big\|_{\mathscr{B}_{\infty, \infty}^{-1}}d\tau.\nonumber
\end{eqnarray}
For the term $\|\widetilde\omega_{q}\cdot\nabla v\Vert_{\mathscr{B}_{\infty, \infty}^{-1}}$ Bony's decomposition allows us to write 
\begin{equation*}
\widetilde\omega_{q}\cdot\nabla v=T_{\widetilde\omega_{q}}\cdot\nabla v+T_{\nabla v}\cdot\widetilde\omega_{q}+{R}(\widetilde\omega_{q}^i,\partial_i v).
\end{equation*}
Consequently,
\begin{eqnarray*}
\Vert\widetilde\omega_{q}\cdot\nabla v\Vert_{\mathscr{B}_{\infty, \infty}^{-1}}&\le& \Vert T_{\widetilde\omega_{q}}\cdot\nabla v\Vert_{\mathscr{B}_{\infty, \infty}^{-1}}+\Vert T_{\nabla v}\cdot\widetilde\omega_{q}\Vert_{\mathscr{B}_{\infty, \infty}^{-1}}+\Vert{R}(\widetilde\omega_{q}\cdot,\nabla v)\Vert_{\mathscr{B}_{\infty, \infty}^{-1}}\\
&\lesssim&\Vert\nabla v\Vert_{L^\infty}\Vert\widetilde\omega_{q}\Vert_{\mathscr{B}_{\infty, \infty}^{-1}}+\Vert{R}(\widetilde\omega_{q}^i,\partial_i v)\Vert_{\mathscr{B}_{\infty, \infty}^{-1}}.
\end{eqnarray*}
Exploring explicitly {\bf(ii)} and $\Div\widetilde{\omega}_q(t,x)=0$ we get to 
\begin{eqnarray*}
\Vert{R}(\widetilde\omega_{q}^i, \partial_iv)\Vert_{\mathscr{B}_{\infty, \infty}^{-1}}&=&\Vert\partial_i{R}(\widetilde\omega_{q}^i, v)\Vert_{\mathscr{B}_{\infty, \infty}^{-1}}\\
&\lesssim&\sup_{k}\sum_{j\ge k-3}\Vert\Delta_{j}\widetilde\omega_{q}\Vert_{L^{\infty}}\Vert\widetilde\Delta_{j}v\Vert_{L^{\infty}}\\
&\lesssim& \Vert\widetilde\omega_q\Vert_{\mathscr{B}_{\infty, \infty}^{-1}}\Vert v\Vert_{\mathscr{B}_{\infty, 1}^{1}}.
\end{eqnarray*}
Consequently,
\begin{equation*}
\Vert \widetilde{\omega}_{q}\cdot\nabla v\Vert_{\mathscr{B}_{\infty, \infty}^{-1}}\lesssim\Vert v\Vert_{\mathscr{B}_{\infty,1}^{1}}\Vert \widetilde{\omega}_{q}\Vert_{\mathscr{B}_{\infty, \infty}^{-1}}.
\end{equation*}
We turn to $\int_{0}^{t}e^{-CZ(\tau)}\Big\|\partial_z\Delta_{q}\Big(\frac{B_{\theta}}{r}B\Big)(\tau)\Big\|_{{B}_{\infty, \infty}^{-1}} d\tau$. The fact that $\partial_z:\mathscr{B}_{\infty,\infty}^{0}\rightarrow \mathscr{B}_{\infty,\infty}^{-1}$ is continuous, $ L^\infty\hookrightarrow \mathscr{B}_{\infty,\infty}^{0}$ and H\"older's inequality with respect to time give us

\begin{eqnarray*}
\int_{0}^{t}e^{-CZ(\tau)}\Big\|\partial_z\Big(\Delta_{q}\Big(\frac{B_{\theta}}{r}B\Big)\Big)\Big\|_{\mathscr{B}_{\infty, \infty}^{-1}}d\tau &\lesssim& \bigg\|\Delta_{q}\Big(\frac{B_{\theta}}{r}B\Big) \Big\|_{L^1_t \mathscr{B}_{\infty, \infty}^{0}}\\
&\lesssim& \Big\|\Delta_{q}\Big(\frac{B_{\theta}}{r}B\Big)\Big\|_{L^1_t L^\infty}.
\end{eqnarray*}
Gathering the last two estimates and plugging them in \eqref{Eq:4,21}, so Gronwall's inequality leads to
\begin{eqnarray*}
\Vert\widetilde\omega_q(t)\Vert_{\mathscr{B}_{\infty, \infty}^{-1}}&\lesssim&e^{CZ(t)} \bigg(\Vert\Delta_{q}{\omega}^{0}\Vert_{\mathscr{B}_{\infty, \infty}^{-1}}+\Big\|\Delta_{q}\Big(\frac{B_{\theta}}{r}B\Big)\Big\|_{L^1_t L^\infty}\bigg)\\
&\lesssim& e^{CZ(t)} \bigg(2^{-q}\Vert\Delta_{q}{\omega}^{0}\Vert_{L^{\infty}}+\Big\|\Delta_{q}\Big(\frac{B_{\theta}}{r}B\Big)\Big\|_{L^1_t L^\infty}\bigg).
\end{eqnarray*}
Then it follows
\begin{equation}\label{17}
\Vert\Delta_{k}\widetilde{\omega}_{q}(t)\Vert_{L^{\infty}}\lesssim 2^{ k-q}e^{CZ(t)} \bigg(\Vert\Delta_{q}{\omega}^{0}\Vert_{L^{\infty}}+2^{q} \Big\|\Delta_{q}\Big(\frac{B_{\theta}}{r}B\Big)\Big\|_{L^1_t L^\infty}\bigg).
\end{equation}
Let us now move to estimate (\ref{m17}). Since $v^{\theta}=0$, then
\begin{equation*}
\frac{v^r}{r}=\frac{v^1}{x_1}=\frac{v^2}{x_2},
\end{equation*}
where $(v^1, v^2, 0)$ are the components of $v$ in cartesian basis. According to Proposition \ref{Prop-1} the vector-valued solution $\widetilde\omega_q$ has two components in cartesian basis $\widetilde\omega_{q}^1$ and $\widetilde\omega_{q}^2$. We restrict ourselves to the proof of estimate of the first component. The second one is done in the same way. Since $\tilde\omega_{q}^1$ obeys\begin{equation} \label{18}
\left\{ \begin{array}{ll}
\partial_{t}\widetilde\omega_{q}^1-\mu\Delta\widetilde\omega_{q}^1+(v\cdot\nabla)\widetilde\omega_{q}^1=\dfrac{v^2}{x_2}\widetilde\omega_{q}^1- \partial_z\Big(\Delta_{q}\Big(\frac{B_{\theta}}{r}B^1\Big)\Big),\\
\widetilde\omega_{q}^1\vert_{\vert t=0}=\Delta_{q}\omega_{0}^1,
\end{array} \right.
\end{equation}
where $(B^1,B^2,0)$ denotes the components of $B$ in cartesian basis. Taking advantage again to Proposition \ref{Prop:1.1} for $r=1$ and $p=\infty$ to write
\begin{eqnarray}\label{m19}
e^{-CZ(t)}\Vert\widetilde\omega_{q}^1(t)\Vert_{\mathscr{B}_{\infty, 1}^{1}}&\lesssim&\Vert\Delta_{q}\omega^1(0)\Vert_{\mathscr{B}_{\infty, 1}^{1}}+\int_{0}^{t}e^{-CZ(\tau)}\Big\Vert v^2\frac{\widetilde\omega_{q}^1}{x_2}(\tau)\Big\Vert_{\mathscr{B}_{\infty, 1}^{1}}d\tau\\
&&+\int_{0}^{t}e^{-CZ(\tau)}\Big\|\partial_z\Big(\Delta_{q}\Big(\frac{B_{\theta}}{r}B^1\Big)\Big)(\tau)\Big\|_{\mathscr{B}_{\infty, 1}^{1}}d\tau.\nonumber
\end{eqnarray}
For the second term of the r.h.s. Bony's decomposition implies 
\begin{eqnarray}\label{m20}
\Big\Vert v^2\frac{\widetilde\omega_{q}^1}{x_2}\Big\Vert_{\mathscr{B}_{\infty, 1}^{1}}&\le &\Big\Vert T_{\frac{\widetilde\omega_{q}^1}{x_2}}v^2\Big\Vert_{\mathscr{B}_{\infty, 1}^{1}}+\Big\Vert T_{v^2}\frac{\widetilde\omega_{q}^1}{x_2}\Big\Vert_{\mathscr{B}_{\infty, 1}^{1}}+\Big\Vert{R}\Big(v^2, \frac{\widetilde\omega_{q}^1}{x_2}\Big)\Big\Vert_{\mathscr{B}_{\infty, 1}^{1}}\\
\nonumber&\triangleq& \mathrm{I}_1+\mathrm{I}_2+\mathrm{I}_3.
\end{eqnarray}
To bound $\mathrm{I}_1$, the paraproduct laws and Besov spaces provide us
\begin{equation}\label{m21}
\mathrm{I}_1\lesssim\sum_{k\ge-1}2^k\Big\Vert S_{k-1}\Big(\frac{\widetilde\omega_{q}^1}{x_2}\Big)\Big\Vert_{L^\infty}\Vert\Delta_{k}v^2\Vert_{L^\infty}\lesssim\Vert v\Vert_{\mathscr{B}_{\infty, 1}^{1}}\Big\Vert\frac{\widetilde\omega_{q}^1}{x_2}\Big\Vert_{L^\infty}.
\end{equation}
The term $\mathrm{I}_3$, will be done by a similar way as above. More precisely, we have
\begin{equation}\label{m22}
\mathrm{I}_3\lesssim\sum_{l\ge k-3}2^k\Vert\Delta_{l}v^2\Vert_{L^\infty}\Big\Vert\widetilde{\Delta}_{l}\frac{\widetilde{\omega}_{q}^1}{x_2}\Big\Vert_{L^{\infty}}\lesssim\Vert v\Vert_{{B}_{\infty, 1}^{1}}\Big\Vert\frac{\widetilde{\omega}_{q}^1}{x_2}\Big\Vert_{L^{\infty}}.
\end{equation}
Let us move to estimate $\mathrm{I}_2$ in the following way.
\begin{equation}\label{m23}
\mathrm{I}_2\lesssim\sum_{m\in\NN}2^m\Big\Vert S_{m-1}v^2(x)\Delta_{m}\Big(\frac{\widetilde{\omega}_{q}^1(x)}{x_2}\Big)\Big\Vert_{L^\infty}.
\end{equation}
We check that
\begin{eqnarray*}
\Big\Vert S_{m-1}v^2(x)\Delta_{m}\Big(\frac{\widetilde\omega_{q}^1(x)}{x_2}\Big)\Big\Vert_{L^\infty}&\le&\Big\Vert S_{m-1}v^2(x)\frac{\Delta_{m}\widetilde\omega_{q}^1(x)}{x_2}\Big\Vert_{L^\infty}\\
&&+\Big\Vert S_{m-1}v^2(x)\Big[\Delta_{m}, \frac{1}{x_2}\Big]\widetilde\omega_{q}^1\Big\Vert_{L^\infty}.
\end{eqnarray*}
For the first term of r.h.s., we have
 \begin{equation}\label{v-2}
\Big\Vert S_{m-1}v^2(x)\frac{\Delta_{m}\widetilde\omega_{q}^1(x)}{x_2}\Big\Vert_{L^\infty}\le  \Big\Vert \frac{S_{m-1}v^2}{x_2}\Big\Vert_{L^\infty} \Vert \Delta_{m}\widetilde\omega_{q}^1(x)\Vert_{L^\infty}.
\end{equation}
Proposition's \ref{prop3} tells us that $S_{m-1}v$ is axisymmetric and consequently $S_{m-1}v^2(0,x_{2}, z)=0$. Thus, Taylor's formula yields
$$ \frac{S_{m-1}v^2(x_1,x_2,z)}{x_2} =\int_{0}^{1}(\partial_{x_2}S_{m-1}v^2)(x_1, \tau x_2, x_3)d\tau .$$
From Lemma \ref{Lem:6.1}, we find  that
\begin{eqnarray*}
\Big\Vert \frac{S_{m-1}v^2}{x_2}\Big\Vert_{L^\infty} &\lesssim& \int_{0}^{1}\Vert\partial_{x_2} S_{m-1}v^2 (\cdot,\tau\cdot, \cdot)\Vert_{L^\infty }(1-\log\tau)d\tau\\
&\lesssim&\Vert \nabla  v \Vert_{L^\infty}\int_{0}^{1}(1-\log\tau)d\tau\\
&\lesssim&\Vert \nabla  v \Vert_{L^\infty}.
\end{eqnarray*}
Inserting the last estimate into \eqref{v-2}, we deduce that 
\begin{equation}\label{v-4}
\Big\Vert S_{m-1}v^2(x)\frac{\Delta_{m}\widetilde\omega_{q}^1(x)}{x_2}\Big\Vert_{L^\infty}\le\Vert\nabla v\Vert_{L^{\infty}}\Vert\Delta_{m}\widetilde{\omega}_{q}^1\Vert_{L^{\infty}}.
\end{equation}
Therefore 
\begin{equation}\label{m24}
\sum_{m\in\NN}2^m\Big\Vert S_{m-1}v^2(x)\frac{\Delta_{m}\widetilde{\omega}_{q}^1(x)}{x_2}\Big\Vert_{L^\infty}\le \Vert\nabla v\Vert_{L^{\infty}}\Vert\widetilde\omega_{q}^1\Vert_{{B}_{\infty,1}^{1}}.
\end{equation}
The commutator term be dealt as follows:
\begin{eqnarray*}
S_{m-1}v^2(x)\Big[\Delta_{m}, \frac{1}{x_2}\Big]\widetilde\omega_{q}^1&=&\frac{S_{m-1}v^2}{x_2}2^{3m}\int_{\RR^3}h(2^m(x-y))(x_2-y_2)\frac{\widetilde\omega_{q}^1(y)}{y_2}dy\\
&=&2^{-m}\Big(\frac{S_{m-1}v^2}{x_2}\Big)2^{3m}\widetilde{h}(2^m\cdot)\star\Big(\frac{\widetilde\omega_q^1}{x_2}\Big)(x),
\end{eqnarray*}
where $\widetilde{h}(x)=x_2h(x)$. The following property holds true for every $f\in{\mathscr{S}}'(\mathbb{R}^3)$.
\begin{equation*}
2^{3m}\widetilde{h}(2^m\cdot)\star f=\sum_{|m-k|\leq 1}2^{3m}\widetilde{h}(2^m\cdot)\star\Delta_{k}f.
\end{equation*}
Really, we have
$\widehat{\widetilde{h}}=i\partial_{\xi_2}\hat{h}=i\partial_{\xi_2}\varphi(\xi)$.  Meaning that
$\supp\widehat{\widetilde{h}}\subset\supp\varphi$, so we have
\begin{equation*}
2^{3m}\widetilde{h}(2^m\cdot)\star\Delta_{m}f\equiv0,\, \textnormal{ for }\vert m-k\vert\ge2.
\end{equation*}
Consequently,
\begin{eqnarray}\label{m25}
\sum_{m\in\NN}2^m\Big\Vert S_{m-1}v^2(x)\Big[\Delta_{m}, \frac{1}{x_2}\Big]\widetilde\omega_{q}^1\Big\Vert_{L^\infty}&\lesssim&\sum_{\vert m-k\vert\le1}\Big\Vert\frac{S_{m-1}v^2}{x_2}\Big\Vert_{L^{\infty}}\Big\Vert\Delta_k\Big(\frac{\widetilde{\omega}_{q}^1}{x_2}\Big)\Big\Vert_{L^{\infty}}\\
\nonumber&\lesssim&\Vert\nabla v\Vert_{L^{\infty}}\Big\Vert\frac{\widetilde{\omega}_{q}^1}{x_2}\Big\Vert_{\mathscr{B}_{\infty,1}^0}.
\end{eqnarray}
Adding \eqref{m24} and \eqref{m25} to deduce that
\begin{equation}\label{m26}
\mathrm{I}_{2}\lesssim\Vert\nabla v\Vert_{L^{\infty}}\Big(\Vert\widetilde{\omega}_{q}^1\Vert_{\mathscr{B}_{\infty, 1}^1}+\Big\Vert\frac{\widetilde{\omega}_{q}^1}{x_2}\Big\Vert_{\mathscr{B}_{\infty, 1}^0}\Big).
\end{equation}
Collecting \eqref{m21}, \eqref{m22} and \eqref{m26}, we can write
\begin{equation*}
\Big\Vert v^2\frac{\widetilde{\omega}_{q}^1}{x_2}\Big\Vert_{\mathscr{B}_{\infty, 1}^1}\lesssim\Vert{\nabla} v\Vert_{\mathscr{B}_{\infty, 1}^1}\Big(\Vert\widetilde{\omega}_{q}^1\Vert_{\mathscr{B}_{\infty, 1}^1}+\Big\Vert\frac{\widetilde{\omega}_{q}^1}{x_2}\Big\Vert_{\mathscr{B}_{\infty, 1}^0}\Big).
\end{equation*}
Thanks to (\ref{m19}) and the above estimate one has
\begin{eqnarray}\label{m27}
e^{-CZ(t)}\Vert\widetilde\omega_{q}^1(\tau)\Vert_{\mathscr{B}_{\infty, 1}^{1}}&\lesssim&\Vert\widetilde\omega_{q}^1(0)\Vert_{\mathscr{B}_{\infty, 1}^{1}}
+\int_{0}^{t}e^{-CZ(\tau)}\Vert v(\tau)\Vert_{\mathscr{B}_{\infty, 1}^{1}}
\Vert \widetilde\omega_{q}^1(\tau)\Vert_{\mathscr{B}_{\infty, 1}^{1}}d\tau\\
\nonumber&+&\int_{0}^{t}e^{-CZ(\tau)}\Vert v(\tau)\Vert_{\mathscr{B}_{\infty, 1}^{1}}\Big\Vert \frac{\widetilde\omega_{q}^1(\tau)}{x_2}\Big\Vert_{\mathscr{B}_{\infty, 1}^{0}}d\tau
 +\int_{0}^{t}e^{-CZ(\tau)}\Big\|\partial_z\Big(\Delta_{q}\Big(\frac{B_{\theta}}{r}B^1\Big)\Big)\|_{\mathscr{B}_{\infty, 1}^{1}}d\tau.
\end{eqnarray}
We now analyze the term $\Big\Vert \frac{\widetilde\omega_{q}^1}{x_2}\Big\Vert_{\mathscr{B}_{\infty, 1}^{0}}$. Thanks to Proposition \ref{prop3} and $\widetilde\omega_{q}^{1}(x_1, 0,z)=0$, we claim by Taylor's formula and the Lemma \ref{Lem:6.1} in appendix that
\begin{equation*}
\frac{\widetilde\omega_{q}^{1}}{x_2}=\int_{0}^{1}\partial_{y}\widetilde\omega_{q}^{1}(x_1, \tau x_2, x_3)d\tau.
\end{equation*}
Hence
\begin{eqnarray*}
\Big\Vert \frac{\widetilde\omega_{q}^1}{x_2}\Big\Vert_{\mathscr{B}_{\infty, 1}^{0}}&\lesssim&\int_{0}^{1}\Vert\partial_{y}\widetilde\omega_{q}^{1}\Vert_{{B}_{\infty, 1}^{0}}(1-\log\tau)d\tau\\
&\lesssim&\Vert\widetilde\omega_{q}^{1}\Vert_{\mathscr{B}_{\infty, 1}^{1}}\int_{0}^{1}(1-\log\tau)d\tau\\
&\lesssim&\Vert\widetilde\omega_{q}^{1}\Vert_{\mathscr{B}_{\infty, 1}^{1}}.
\end{eqnarray*}
To close \eqref{m19} we must estimate $\int_{0}^{t}e^{-CZ(\tau)}\big\|\partial_z\big(\Delta_{q}\big(\frac{B_{\theta}}{r}B^1\big)\big)\big\|_{\mathscr{B}_{\infty, 1}^{1}}d\tau$. By H\"older's inequality, Bernstein's lemma and the fact $\partial_{z}:\mathscr{B}^{2}_{\infty,1}\rightarrow\mathscr{B}^{1}_{\infty,1}$ yield
\begin{eqnarray*}
\int_{0}^{t}e^{-CZ(\tau)}\Big\|\partial_z\Big(\Delta_{q}\Big(\frac{B_{\theta}}{r}B^1\Big)\Big)\Big\|_{\mathscr{B}_{\infty, 1}^{1}}d\tau&\lesssim& \Big\|\Delta_{q}\Big(\frac{B_{\theta}}{r}B^1\Big)\Big\|_{L^1_{t}\mathscr{B}_{\infty, 1}^{2}}\\
&\lesssim&2^{2q} \Big\|\Delta_{q}\Big(\frac{B_{\theta}}{r}B^1\Big)\Big\|_{L^1_{t}  L^\infty},\nonumber
\end{eqnarray*}
combined with \eqref{m27}, it happens
\begin{eqnarray*}
e^{-CZ(t)}\Vert\widetilde\omega_{q}^1(\tau)\Vert_{\mathscr{B}_{\infty, 1}^{1}}\lesssim\Vert\widetilde\omega_{q}^1(0)\Vert_{\mathscr{B}_{\infty, 1}^{1}}+\int_{0}^{t}e^{-CZ(\tau)}\Vert v(\tau)\Vert_{\mathscr{B}_{\infty, 1}^{1}}\Vert \widetilde\omega_{q}^1(\tau)\Vert_{\mathscr{B}_{\infty, 1}^{1}}d\tau+2^{2q}\Big\|\Delta_{q}\Big(\frac{B_{\theta}}{r}B^1\Big)\Big\|_{L^1_{t}  L^\infty}.
\end{eqnarray*}
Via, Gronwall's lemma we obtain 
\begin{equation*}
\Vert\widetilde\omega_{q}^1(\tau)\Vert_{\mathscr{B}_{\infty, 1}^{1}}\lesssim e^{CZ(t)}\Big(2^q\Vert\widetilde\omega_{q}^1(0)\Vert_{L^\infty}+2^{2q}\Big\|\Delta_{q}\Big(\frac{B_{\theta}}{r}B^1\Big)\Big\|_{L^1_{t}  L^\infty}\Big).
\end{equation*}
This gives in particular the estimate (\ref{m17})
\begin{equation*}
\Vert\Delta_{k}\widetilde\omega_{q}(t)\Vert_{L^{\infty}}\le C2^{q-k}e^{CZ(t)}\Big(\Vert\Delta_{q}\omega^{0}\Vert_{L^{\infty}}+2^{q}\Big\|\Delta_{q}\Big(\frac{B_{\theta}}{r}B^1\Big)\Big\|_{L^1_{t}  L^\infty}\Big).
\end{equation*}
The proof of item {\bf(iv)} is now achieved.
\end{proof}
At this stage, to reach the Theorem \ref{The:1.2} we require to propagate the persistence of the initial regularity uniformly on the viscosity. Especially, we will prove the following result.

\begin{prop}\label{Prop:4.4}
Let $p\in [2,\infty]$ and $(v^0,B^0)\in (L^2 \cap \mathcal{B}_{p,1}^{1+\frac{3}{p}})  \times (L^2\cap\mathcal{B}_{p,1}^{\frac{3}{p}-1})$   be two  axisymmetric vector fields in divergence-free such that $\omega^0 \in L^\infty,$ and $(\omega^{0}_\theta/r ,B^{0}_\theta/r ) \in L^{3,1} \times ( L^2 \cap L^\infty )$. Then for any smooth solution $(v,B)$ of \eqref{MHD(mu)}, the following assertions are hold.
\begin{enumerate}

\item[{\bf(i)}]For $p=\infty$ and  $t\geq 0 $, we have
\begin{equation*}
\Vert v(t)\Vert_{\mathscr{B}_{\infty, 1}^{1}} +\Vert\omega(t)\Vert_{\mathscr{B}_{\infty, 1}^{0}} \le \Phi_4(t). 
\end{equation*}
\item[{\bf(ii)}]For every $[2,\infty)$ and $t\geq 0 $, we have
\begin{equation*}
\Vert v(t)\Vert_{\mathscr{B}_{p, 1}^{1+\frac{3}{p}}} +\| B\|_{\widetilde{L}^\infty_t \mathcal{B}^{\frac{3}{p}-1}_{p,1}}+ \| B\|_{\widetilde{L}^1_t \mathcal{B}^{\frac{3}{p}+1}_{p,1}} +\Vert\omega(t)\Vert_{\mathscr{B}_{p, 1}^{\frac{3}{p}}} \le \Phi_6(t).
\end{equation*}
\end{enumerate}
\end{prop}

\begin{proof}{\bf(i)} We fix an integer $N$ which will be judiciously chosen later. With the help of {\bf(i)}-Proposition \ref{Prop:4.3}, we write
\begin{eqnarray}\label{m28}
\Vert\omega(t)\Vert_{\mathscr{B}_{\infty, 1}^{0}}&\le& \sum_{k}\Big\Vert\Delta_{k}\sum_{q}\widetilde{\omega}_{q}(t)\Big\Vert_{L^{\infty}}\\
\nonumber&\le&\sum_{k\ge-1}\sum_{\vert k-q\vert\ge N}\Vert\Delta_{k}\widetilde{\omega}_{q}(t)\Vert_{L^{\infty}}+\sum_{k\ge-1}\sum_{\vert k-q\vert<N}\Vert\Delta_{k}\widetilde{\omega}_{q}(t)\Vert_{L^{\infty}}\triangleq \textnormal{II}_1+\textnormal{II}_2.
\end{eqnarray}
In order to analyze $\mathrm{II}_1$, the item {\bf(iv)}-Proposition \ref{Prop:4.3} claims that
\begin{eqnarray}\label{m29}
\mathrm{II}_1&=& \sum_{k\ge-1}\sum_{\vert k-q\vert\ge N}\Vert\Delta_{k}\widetilde{\omega}_{q}(t)\Vert_{L^{\infty}}  \\ \nonumber
&\lesssim&2^{-N}\bigg(\Vert\omega^{0}\Vert_{\mathscr{B}_{\infty, 1}^{0}}+\Big\|\frac{B_{\theta}}{r}B\Big\|_{L^1_t \mathscr{B}^{1}_{\infty,1}}\bigg)e^{CZ(t)},
\end{eqnarray}
with $Z(t)=\|v\|_{L^1_t\mathscr{B}^{1}_{\infty,1}}$. For the term $\big\|\frac{B_{\theta}}{r}B\big\|_{L^1_t \mathscr{B}^{1}_{\infty,1}}$, {\bf(ii)-}Proposition \ref{prop:7.10} gives 
\begin{eqnarray*}
\Big\|\frac{B_{\theta}}{r}B\Big\|_{L^1_t \mathscr{B}^{1}_{\infty,1}} &\lesssim& \Big\|\frac{B_{\theta}}{r}\Big\|_{L^\infty_t L^\infty} \| B \|_{L^1_t\mathscr{B}^{1}_{\infty,1}}  .
\end{eqnarray*}
Moreover, combining the embedding $  {\mathscr{B}_{\sigma, 1}^{1+\frac{3}{\sigma}}} \hookrightarrow {\mathscr{B}_{\infty, 1}^{1}}$ with {\bf(ii)-}Proposition \ref{prop:4.2} and {\bf(ii)-}Proposition \ref{prop:4.1} to find that
\begin{eqnarray}\label{B-r,B}
\Big\|\frac{B_{\theta}}{r}B\Big\|_{L^1_t \mathscr{B}^{1}_{\infty,1}}  \le \Phi_2(t)  .
\end{eqnarray}
Inserting this estimate into \eqref{m29} leads
\begin{eqnarray}\label{II,1}
\mathrm{II}_1 &\lesssim&2^{-N}\Big(\Vert\omega^{0}\Vert_{\mathscr{B}_{\infty, 1}^{0}}+\Phi_2(t)\Big)e^{CZ(t)}.
\end{eqnarray} 
We bound $\mathrm{II}_2$ by exploiting {\bf(iii)}-Proposition \ref{Prop:4.3} to state
\begin{eqnarray}\label{m29-i}
\mathrm{II}_2&=& \sum_{k\ge-1}\sum_{\vert k-q\vert\le N}\Vert\Delta_{k}\widetilde{\omega}_{q}(t)\Vert_{L^{\infty}}\\ 
\nonumber\\
\nonumber&\lesssim& N\Phi_2(t)\bigg(\Vert\omega^{0}\Vert_{\mathscr{B}_{\infty, 1}^{0}}+\Big\|\frac{B_{\theta}}{r}B\Big\|_{L^1_t\mathscr{B}^{1}_{\infty,1}}\bigg).
\end{eqnarray}
In view of \eqref{B-r,B}, we find  
\begin{eqnarray*}
\mathrm{II}_2  &\lesssim& N\Phi_2(t)\Big(\Vert\omega^{0}\Vert_{\mathscr{B}_{\infty, 1}^{0}} +\Phi_2(t) \Big) .
\end{eqnarray*}
Combining the last estimate with \eqref{II,1}, it follows
\begin{equation*}
\mathrm{II}_1+\mathrm{II}_2\lesssim \big(2^{-N}e^{CZ(t)}+N\Phi_2(t)\big)\big(\Vert\omega^{0}\Vert_{\mathscr{B}_{\infty,1}^0}+\Phi_2(t)\big).
\end{equation*}
Taking $N$ equals to $\big[CZ(t)\big]+1$ to get
\begin{equation}\label{Eq-30}
\Vert\omega(t)\Vert_{\mathscr{B}_{\infty, 1}^{0}}\lesssim\big(Z(t)+1\big)\Phi_2(t).
\end{equation}
To finalize this item, we shall estimate $\Vert v\Vert_{\mathscr{B}_{\infty, 1}^{1}}$. Doing so, the fat that $2^{q}\|\Delta_q v\|_{L^\infty}\approx \|\Delta_q\omega\|_{L^\infty}$ for $q \in \mathbb{N} $, estimate \eqref{Eq-30} and {\bf(iii)-}Proposition \ref{prop:4.2} provide us 
\begin{eqnarray*}
\Vert v(t) \Vert_{\mathscr{B}_{\infty, 1}^{1}}&=&\sum_{q\ge-1}\Vert \Delta_q v(t) \Vert_{L^{\infty}}\\
&\lesssim&\|\Delta_{-1}v (t) \|_{L^\infty}+\sum_{q\ge0}2^q\Vert\Delta_q v(t) \Vert_{L^\infty}\\&\lesssim&\|v(t) \|_{L^\infty}+\sum_{q\ge0}\Vert\Delta_q\omega(t)\|_{L^\infty}\\
&\lesssim &\|v(t)\|_{L^\infty}+\|\omega(t)\|_{\mathscr{B}_{\infty,1}^{0}}
\le \Phi_3(t)\Big(1+\int_0^t\|v(\tau)\|_{\mathscr{B}^{1}_{\infty,1}}d\tau\Big).
\end{eqnarray*}
Hence we get by Gronwall's lemma that
\begin{equation}\label{v-B{1}}
\Vert v(t)\Vert_{\mathscr{B}_{\infty, 1}^{1}}\le \Phi_4(t),
\end{equation}
Combining the last estimate with the embedding  ${\mathscr{B}_{\infty, 1}^{1}} \hookrightarrow \Lip(\mathbb{R}^3) $ we get
\begin{equation}\label{b,v-Lip}
\Vert \nabla v(t)\Vert_{L^\infty} \le \Phi_{4}(t)
\end{equation}
and consequently from \eqref{Eq-30}, we infer that
\begin{equation}\label{Ome-B{1}}
\Vert\omega(t)\Vert_{\mathscr{B}_{\infty, 1}^{0}}\lesssim \Phi_4(t).
\end{equation}
This explains that {\bf(i)} is achieved.

{\bf(ii)}. The dyadic decomposition and the fact $2^{q}\Vert\Delta_{q}v\Vert_{L^p} \approx \Vert\Delta_{q} \omega \Vert_{L^p}$ for $q\in\NN$ enuble us to write
\begin{eqnarray}\label{Eq:6.16}
\Vert v(t)\Vert_{\mathscr{B}_{p, 1}^{1+\frac{3}{p}}} \nonumber &\lesssim&\Vert\Delta_{-1}v\Vert_{L^p}+\sum_{q\in\NN}2^{q\frac{3}{p}}2^{q}\Vert\Delta_{q}v\Vert_{L^p}\\
&\lesssim&\Vert v(t)\Vert_{L^p}+\Vert \omega(t)\Vert_{\mathscr{B}_{p, 1}^{\frac{3}{p}}}.
\end{eqnarray}
To close our claim, we must estimate $\Vert v(t)\Vert_{L^p}$. For this aim, we rewrite $v-$equation as 
\begin{equation*}
(\partial_t -\mu\Delta)v =\mathbb{P}(B\cdot\nabla B-v\cdot\nabla v),
\end{equation*}
where $\mathbb{P}$ refers to Leray's operator. Since for $p\in(1, \infty)$, Riesz's transform maps continuously $L^p$ into itself, then
\begin{eqnarray*}
\Vert v(t)\Vert_{L^p}&\le&\Vert v^0\Vert_{L^p}+\int_0^t\|(B\cdot\nabla B)(\tau)\|_{L^p}+\int_{0}^{t}\Vert (v\cdot\nabla v)(\tau)\Vert_{L^p}d\tau.
\end{eqnarray*}
Since $B=B_\theta \vec{e}_{\theta}$ we have $B\cdot\nabla B =-\frac{(B_\theta)^2}{r}\vec{e}_{r}$, so {\bf(ii)-}Proposition \ref{prop:4.1} and the embedding $\mathscr{B}_{p, 1}^{\frac{3}{p}+1} \hookrightarrow L^p$ help us to write
\begin{equation}
\Vert v(t)\Vert_{L^p} \le \Vert v^0\Vert_{L^p}+\Big\|\frac{B_\theta^0}{r}\Big\|_{L^p}\| B\|_{L^1_t \mathscr{B}_{p, 1}^{\frac{3}{p}+1}}+\int_{0}^{t}\Vert v(\tau)\Vert_{L^p}\Vert\nabla v(\tau)\Vert_{L^\infty}d\tau.
\end{equation} 

Then, Gronwall's inequality gives
\begin{eqnarray}\label{EQ:4.41}
\Vert v(t)\Vert_{L^p}&\lesssim&\Big(\Vert v^0\Vert_{L^p}+\Big\|\frac{B_\theta^0}{r}\Big\|_{L^p}\| B\|_{L^1_t \mathscr{B}_{p, 1}^{\frac{3}{p}+1}} \Big)e^{CV(t)},
\end{eqnarray}

where $V(t)=\|\nabla v\|_{L^1_tL^\infty}$, combined with the last estimate and \eqref{Eq:6.16}, we get
\begin{equation}\label{Eq:v-lp}
\Vert v(t)\Vert_{\mathscr{B}_{p, 1}^{1+\frac{3}{p}}} \le C\Big(\Vert v^0\Vert_{L^p}+\Big\|\frac{B_\theta^0}{r}\Big\|_{L^p}\| B\|_{L^1_t \mathscr{B}_{p, 1}^{\frac{3}{p}+1}} \Big)e^{CV(t)}+\Vert \omega(t)\Vert_{\mathscr{B}_{p, 1}^{\frac{3}{p}}}.
\end{equation}
Now, we treat the quantity $\|\omega\|_{\widetilde{L}^\infty_t\mathcal{B}^\frac{3}{p}_{p,\infty}}.$ For this, we apply Proposition \ref{Prop:1.1} to $\omega-$equation for $s=\frac{3}{p},r=1$ to ensure 
\begin{equation}\label{Es-5}
\| \omega(t)  \|_{\mathcal{B}^\frac{3}{p}_{p,1}} \le Ce^{CV(t)} \bigg(\|  \omega^0 {\|}_{\mathcal{B}^0_{p,1}} + \int_0^t e^{-CV(\tau)}\Big(\| \omega\cdot \nabla v (\tau )  \|_{\mathcal{B}^\frac{3}{p}_{p,1}}+\| \cur( B\cdot \nabla B)(\tau )  \|_{\mathcal{B}^\frac{3}{p}_{p,1}}\Big)d\tau \bigg). 
\end{equation}
First, we prove that 
\begin{equation}\label{m31}
\Vert\omega\cdot\nabla v \Vert_{\mathscr{B}_{p, 1}^{\frac{3}{p}}}\lesssim\Vert\omega\Vert_{\mathscr{B}_{p, 1}^{\frac{3}{p}}}\Vert \nabla v\Vert_{L^{\infty}}.
\end{equation}
For this aim, we explore Bony's decomposition to write 
\begin{equation}\label{m32}
\Vert\omega\cdot\nabla v\Vert_{\mathscr{B}_{p, 1}^{\frac{3}{p}}}\le\Vert T_{\nabla v}\cdot\omega\Vert_{\mathscr{B}_{p, 1}^{\frac{3}{p}}}+\Vert T_{\omega}\cdot\nabla v\Vert_{\mathscr{B}_{p, 1}^{\frac{3}{p}}}+\Vert {R}\big(\omega^i, \partial_i v\big)\Vert_{\mathscr{B}_{p, 1}^{\frac{3}{p}}}.
\end{equation}
By definition of $\Vert {R}\big(\omega^i, \partial_i v\big)\Vert_{\mathscr{B}_{p, 1}^{\frac{3}{p}}}$, we have
\begin{eqnarray*}
\Vert {R}\big(\omega, \nabla v\big)\Vert_{\mathscr{B}_{p, 1}^{\frac{3}{p}}}&\le&\sum_{q\in\NN}2^{q\frac{3}{p}}\sum_{j\ge q-3}\Vert\Delta_{j}\omega\Vert_{L^p}\Vert\Delta_{j}\nabla v\Vert_{L^{\infty}}\\
&\lesssim&\Vert\nabla v\Vert_{L^{\infty}}\sum_{j\ge q-3}2^{(q-j)\frac{3}{p}}2^{j\frac{3}{p}}\Vert\Delta_{j}\omega\Vert_{L^p}\\
&\lesssim&\Vert\nabla v\Vert_{L^{\infty}}\Vert\omega\Vert_{\mathscr{B}_{p, 1}^{\frac{3}{p}}}.
\end{eqnarray*}
We can get the following
\begin{equation*}
\Vert T_{\nabla v}\cdot\omega\Vert_{\mathscr{B}_{p, 1}^{\frac{3}{p}}}\lesssim \|\nabla v\|_{L^\infty}\Vert\omega\Vert_{\mathscr{B}_{p, 1}^{\frac{3}{p}}}.
\end{equation*}
For the third term we write,
\begin{eqnarray*}
\Vert T_{\omega}\cdot\nabla v\Vert_{\mathscr{B}_{p, 1}^{\frac{3}{p}}}&\lesssim&\sum_{q\in\NN}2^{q\frac{3}{p}}\Vert S_{q-1}\omega\Vert_{L^{\infty}}\Vert\nabla\Delta_{q}v\Vert_{L^p}\\
&\lesssim&\Vert\omega\Vert_{L^{\infty}}\sum_{q\in\NN}2^{q\frac{3}{p}}\Vert\Delta_{q}\omega\Vert_{L^p}\\
&\lesssim&\Vert\nabla v\Vert_{L^{\infty}}\Vert\omega\Vert_{{B}_{p, 1}^{\frac{3}{p}}}.
\end{eqnarray*}
Collecting the three last estimates we find \eqref{m31}. Plugging \eqref{m31} in \eqref{Es-5}, one gets
\begin{equation}\label{Eq-31}
\Vert\omega\Vert_{\mathscr{B}_{p, 1}^{\frac{3}{p}}}\lesssim e^{CV(t)} \bigg(\Vert\omega_{0}\Vert_{\mathscr{B}_{p, 1}^{\frac{3}{p}}}+\int_{0}^{t}e^{-CV(t)}\Vert\nabla v(\tau)\Vert_{L^{\infty}}\Vert\omega(\tau)\Vert_{\mathscr{B}_{p, 1}^{\frac{3}{p}}}d\tau+\int_{0}^{t}e^{-CV(\tau)}\Vert \Curl(B\cdot\nabla B)(\tau)\Vert_{\mathscr{B}_{p, 1}^{\frac{3}{p}}}d\tau \bigg).
\end{equation}
The fact that $\Curl( B\cdot \nabla B) =-\partial_z\big(\frac{B_{\theta}}{r} B\big)$ and the continuity of $\partial_z:{\mathcal{B}^{\frac{3}{p}+1}_{p,1}}\rightarrow \mathcal{B}^\frac{3}{p}_{p,1} $ leads to 
\begin{equation*}
\| \cur( B\cdot \nabla B )(\tau )  \|_{\mathcal{B}^\frac{3}{p}_{p,1}} =\Big\|  \partial_z(\frac{B_{\theta} }{r} B)(\tau )\Big\|_{\mathcal{B}^\frac{3}{p}_{p,1}} \lesssim \Big\| \frac{B_{\theta}}{r} B(\tau )\Big\|_{\mathcal{B}^{\frac{3}{p}+1}_{p,1}}.
\end{equation*}
According to {\bf(ii)-}Proposition \ref{prop:7.10} in appendix, one has 
\begin{equation*}
\| \Curl( B\cdot \nabla B )(\tau )  \|_{\mathcal{B}^\frac{3}{p}_{p,1}}\lesssim \Big\| \frac{B_{\theta}}{r}(\tau )\Big\|_{L^\infty}\| B(\tau )\|_{\mathcal{B}^{\frac{3}{p}+1}_{p,1}}.
\end{equation*}
Thanks to {\bf(iii)}-Proportion \ref{prop:4.1}, that is $\| \frac{B^{\theta }}{r}(\tau )\|_{L^\infty} \le \| \frac{B_{\theta}^0}{r}\|_{L^\infty}$, it follows 
\begin{equation}\label{curl(B-B)}
\| \cur( B\cdot \nabla B)(\tau )  \|_{\mathcal{B}^\frac{3}{p}_{p,1}}\lesssim \Big\| \frac{{{B_{\theta}^0}}}{r} \Big\|_{L^\infty}\| B(\tau )\|_{\mathcal{B}^{\frac{3}{p}+1}_{p,1}} .
\end{equation}
Inserting the last estimate in \eqref{Eq-31}, it happens 
\begin{equation*}
e^{-CV(t)}\Vert\omega\Vert_{\mathscr{B}_{p, 1}^{\frac{3}{p}}}\lesssim   \bigg( \Vert\omega^{0}\Vert_{\mathscr{B}_{p, 1}^{\frac{3}{p}}}+\Big\|\frac{B^0}{r}\Big\|_{L^\infty}\|B\|_{L^1_t \mathscr{B}_{p, 1}^{1+\frac{3}{p}}} +\int_{0}^{t}e^{-CV(t)}\Vert\nabla v(\tau)\Vert_{L^{\infty}}\Vert\omega(\tau)\Vert_{\mathscr{B}_{p, 1}^{\frac{3}{p}}}d\tau \bigg).
\end{equation*}
Further, Gronwall's lemma leads
\begin{equation}\label{Eq:6.00.}
\Vert\omega\Vert_{\mathscr{B}_{p, 1}^{\frac{3}{p}}}\le Ce^{CV(t)} \bigg(\Vert\omega^{0}\Vert_{\mathscr{B}_{p, 1}^{\frac{3}{p}}}+\Big\|\frac{B^0}{r}\Big\|_{L^\infty}\|B\|_{L^1_t \mathscr{B}_{p, 1}^{1+\frac{3}{p}}}\bigg).
\end{equation}
By plugging the last estimate into \eqref{Eq:v-lp} yields
\begin{equation}\label{Eq:6.22}
\Vert v(t)\Vert_{\mathscr{B}_{p, 1}^{1+\frac{3}{p}}}\le Ce^{CV(t)} \Big(\Vert v^0\Vert_{L^p}+ \Vert\omega^{0}\Vert_{\mathscr{B}_{p, 1}^{\frac{3}{p}}}+ \Big\|\frac{B_\theta^0}{r}\Big\|_{L^p}\| B\|_{L^1_t \mathscr{B}_{p, 1}^{\frac{3}{p}+1}} \Big).
\end{equation}
To bound the term $\|B\|_{L^1_t \mathscr{B}_{p, 1}^{1+\frac{3}{p}}}$, we distinguish two cases \\
{\bf $\bullet$ Case}: $p>3$. From {\bf(iii)} Proposition \eqref{prop:4.2} and \eqref{b,v-Lip}, we infer that
\begin{eqnarray}\label{4.56}
\| v(t) \|_{ \mathcal{B}^{1+\frac{3}{p}}_{p,1}}  +\| B\|_{\widetilde{L}^\infty_t \mathcal{B}^{\frac{3}{p}-1}_{p,1}}+ \| B\|_{\widetilde{L}^1_t \mathcal{B}^{\frac{3}{p}+1}_{p,1}}  \le \Phi_5(t).
\end{eqnarray}
Plugg the last estimate in \eqref{Eq:6.00.} and using \eqref{b,v-Lip} to conclude that
 $$\|\omega(t) \|_{ \mathcal{B}^{\frac{3}{p}}_{p,1}} \le \Phi_6(t) .$$
{\bf $\bullet$ Case}$: 2 \le p \le 3.$
According to \eqref{Eq:4.21}, we have
\begin{eqnarray}\label{Eq:6.23}
\| B\|_{\widetilde{L}^\infty_t \mathcal{B}^{\frac{3}{p}-1}_{p,1}}+ \| B\|_{\widetilde{L}^1_t \mathcal{B}^{\frac{3}{p}+1}_{p,1}}  &\lesssim&   \| B^0\|_{ \mathcal{B}^{\frac{3}{p}-1}_{p,1}} +    \|\Delta_{-1} B\|_ {L^1_tL^p} + \| B\cdot \nabla v \|_{\widetilde {L}^1_t \mathcal{B}^{\frac{3}{p}-1}_{p,1}} \\&+&\sum_{q\geq 0} 2^{q(\frac{3}{p}-1)} \|[\Delta_q,v\cdot \nabla]B\|_{L^1_t L^{p}}\nonumber .
\end{eqnarray}
For the term $ \| B\cdot \nabla v \|_{ \mathcal{B}^{\frac{3}{p}-1}_{p,1}}$. If $ 2 \le p<3$, we have immediately $\frac{3}{p}-1 >0$,  then $ \mathcal{B}_{p,1}^{\frac{3}{p}-1}\cap L^\infty$ is an algebra, this yields 
\begin{equation*}
\| B\cdot \nabla v \|_{ \mathcal{B}^{\frac{3}{p}-1}_{p,1}} \lesssim \|  \nabla v \|_{L^{\infty}} \| B \|_{ \mathcal{B}^{\frac{3}{p}-1}_{p,1}} + \| B \|_{L^{\infty}}  \|  \nabla v\|_{ \mathcal{B}^{\frac{3}{p}-1}_{p,1}}  
\end{equation*}
If $p=3$, in view of Lemma 5.2 in \cite{hz0} page 20, we have
\begin{equation*}
\|B\cdot\nabla v\|_{\mathscr{B}_{p,1}^{0}}\lesssim \|v\|_{\mathscr{B}_{\infty,1}^{1}}\| B \|_{ \mathcal{B}^{\frac{3}{p}-1}_{p,1}}.
\end{equation*}
Combining the two last estimates and in view of Bernstein's inequality and the embedding $ \mathcal{B}^{\frac{3}{p}+1}_{p,1} \hookrightarrow  \mathcal{B}^{\frac{3}{p}}_{p,1}$, we find 
\begin{equation}\label{61,1}
\|B\cdot\nabla v\|_{\mathcal{B}^{\frac{3}{p}-1}_{p,1}}\lesssim \|v\|_{\mathscr{B}_{\infty,1}^{1}}\| B \|_{ \mathcal{B}^{\frac{3}{p}-1}_{p,1}}+ \| B \|_{L^{\infty}}  \| v\|_{ \mathcal{B}^{\frac{3}{p}+1}_{p,1}}  .
\end{equation}
Let us move to the commutator term in r.h.s. of \eqref{Eq:6.23}. By virtue of Lemma \ref{Lem:2.3} in appendix with $-1<\frac{3}{p}-1) \le 1/2$ provides 
\begin{equation}\label{61,2}
\sum_{q\geq 0} 2^{q(\frac{3}{p}-1)} \|[\Delta_q,v\cdot \nabla]B\|_{L^1_t L^{p}} \lesssim \int_0^t  \|\nabla v(\tau)\|_{L^\infty} \| B(\tau) \|_{ \mathcal{B}^{\frac{3}{p}-1}_{p,1}} d \tau 
\end{equation}
Plugging \eqref{61,1}, \eqref{61,2} into \eqref{Eq:6.23} and using the embeddings $\mathcal{B}^{\frac{3}{p}-1}_{p,1}\hookrightarrow L^p $ and $\mathcal{B}^{1}_{\infty,1}\hookrightarrow \Lip (\RR^3) $ to get 
\begin{eqnarray*}
\| B\|_{\widetilde{L}^\infty_t \mathcal{B}^{\frac{3}{p}-1}_{p,1}}+ \| B\|_{\widetilde{L}^1_t \mathcal{B}^{\frac{3}{p}+1}_{p,1}}  &\lesssim&   \| B^0\|_{ \mathcal{B}^{\frac{3}{p}-1}_{p,1}} +\int_0^t \| B(\tau)\|_{L^\infty} \| v(\tau) \|_{ \mathcal{B}^{\frac{3}{p}+1}_{p,1}} d\tau\\
&& +\int_0^t \big(1+\| v(\tau)\|_{ \mathcal{B}^{1}_{\infty,1}} \big)\| B(\tau) \|_{ \mathcal{B}^{\frac{3}{p}-1}_{p,1}}d\tau.
\end{eqnarray*}
Using Gornwall's inequality, we obtain
\begin{eqnarray}\label{B.b}
\| B\|_{\widetilde{L}^\infty_t \mathcal{B}^{\frac{3}{p}-1}_{p,1}}+ \| B\|_{\widetilde{L}^1_t \mathcal{B}^{\frac{3}{p}+1}_{p,1}}  \lesssim e^{C{\| v\|_{L^1_t \mathcal{B}^{1}_{\infty,1}}}} \bigg(  \| B^0\|_{ \mathcal{B}^{\frac{3}{p}-1}_{p,1}}  +\int_0^t \| B(\tau)\|_{L^\infty}  \| v(\tau) \|_{ \mathcal{B}^{\frac{3}{p}+1}_{p,1}}  d\tau   \bigg).
\end{eqnarray}
Substituting the last estimate in \eqref{Eq:6.22}, we find
\begin{eqnarray*}
\| v(t) \|_{ \mathcal{B}^{\frac{3}{p}-1}_{p,1}} \lesssim e^{C{\| v\|_{L^1_t \mathcal{B}^{1}_{\infty,1}}}} \Big( \Vert v^0\Vert_{L^p}+ \Vert\omega^{0}\Vert_{\mathscr{B}_{p, 1}^{\frac{3}{p}}}+   \|B_\theta^0/r\|_{L^p}\Big( \| B^0\|_{ \mathcal{B}^{\frac{3}{p}-1}_{p,1}} + \int_0^t \| B(\tau)\|_{L^\infty}  \| v(\tau) \|_{ \mathcal{B}^{\frac{3}{p}+1}_{p,1}}  d\tau \Big).
\end{eqnarray*}
According to Gronwall's inequality, we infer that 
\begin{eqnarray*}
\| v(t) \|_{ \mathcal{B}^{1+\frac{3}{p}}_{p,1}}  \le C_{0}e^{C{\| v\|_{L^1_t \mathcal{B}^{1}_{\infty,1}}}\| B\|_{L^1_t L^\infty} } e^{\exp(  C{\| v\|_{L^1_t \mathcal{B}^{1}_{\infty,1}}} )} .
\end{eqnarray*}
Or, from the embedding $ \mathcal{B}^{1+\frac{3}{p}}_{p,1} \hookrightarrow  \mathcal{B}^{1+\frac{3}{\sigma}}_{\sigma,1},$ with $\sigma >3$ and   {\bf(iv.b)-}Proposition \eqref{prop:4.1} and {\bf(i)}, we deduce that 
$$\| v(t) \|_{ \mathcal{B}^{\frac{3}{p}+1}_{p,1}} \le \Phi_6(t)$$
Inserting the last  estimate in \eqref{B.b} consequently,
\begin{eqnarray}\label{4.56}
\| B\|_{\widetilde{L}^\infty_t \mathcal{B}^{\frac{3}{p}-1}_{p,1}} + \| B\|_{\widetilde{L}^1_t \mathcal{B}^{\frac{3}{p}+1}_{p,1}}  \le \Phi_6(t).
\end{eqnarray}
Finally, plugging \eqref{4.56} in \eqref{Eq:6.00.}, we conclude that
 $$\|\omega(t) \|_{ \mathcal{B}^{\frac{3}{p}}_{p,1}} \le \Phi_6(t) .$$
 This finishes the proof.
\end{proof}

\section{Inviscid limit}\label{In:Limit}
\textnormal{This section addresses the inviscid limit of the viscous system \eqref{MHD(mu)} to the inviscid one \eqref{MHD(0)} as soon as the viscosity tends to zero, and quantify the rate of convergence between velocities and magnetic fields. }

\subsection{Proof of Theorem \ref{The:1.2}}

Let $(v_\mu,B_\mu,p_\mu)$ and $(v,B,p)$ be a solution of \eqref{MHD(mu)} and \eqref{MHD(0)} respectively. Setting $\overline{v}_\mu=v_\mu -v, \overline{B}_\mu =B_\mu -B $ and $\overline{p}_\mu =p_\mu-p$. So, an elementary calculus claims that the triplet $(\overline{v}_\mu,\overline{B}_\mu,\overline{p}_\mu )$ satisfies the following evolution system
\begin{displaymath}\label{D-mu}
\left\{ \begin{array}{ll} 
\partial_{t} \overline{v}_\mu+v_\mu \cdot\nabla \overline{v}_\mu=\mu\Delta  v_\mu - \overline{v}_\mu \cdot\nabla v+\overline{B}_\mu\cdot\nabla B +B_\mu \cdot\nabla \overline{B}_\mu -\nabla\overline{p}_\mu,  & \\
\partial_{t}\overline{B}_\mu +v_\mu\cdot\nabla \overline{B}_\mu -\Delta \overline{B}_\mu =-\overline{v}_\mu \cdot\nabla B +\overline{B}_\mu\cdot\nabla v  +B_\mu \cdot\nabla\overline{v}_\mu,  & \\ 
\Div \overline{v}_\mu =0,\Div \overline{B}_\mu=0, &\\ 
({\overline{v}_\mu},{\overline{B}_\mu})_{| t=0}=({\overline{v}^{0}_\mu},{\overline{B}{^0}_\mu}).\tag{D$_\mu$} 
\end{array} \right.
\end{displaymath}
\begin{Rema}
If we think to apply the approach in \cite{hz0} for the axisymmetric Navier-Stokes equations in critical Besov spaces, we find a difficulty at the level of estimation $\|B_{\mu}\cdot\nabla\overline{v}_{\mu}\|_{L^{1}_{t} \mathscr{B}_{p,1}^{0}}$. Furthermore, to bound $\nabla\overline{v}_\mu$ in $\mathscr{B}_{p,1}^{0}$, we need an additional regularity for $\overline{v}_\mu$. Unfortunately, we do not have this advantage, so we shall start by developing the rate of convergence in $\mathcal{B}^{0}_{2,1} $, next we will explore the complex interpolation.
\end{Rema}
\begin{proof} We proceed by steps.

{\bf Step 1}. Performing the $L^2$ scalar product of the first equation with $\overline{v}_\mu$ and integrating by parts over $\RR^3$. Then in view of $\Div \overline{v}_\mu=\Div \overline{B}_\mu$, we obtain 
\begin{eqnarray*}
\frac{1}{2}\frac{d}{dt}\|\overline{v}_\mu(t)\|^2_{L^2}   &=&\mu\int_{\RR^3} \Delta v_\mu \cdot \overline{v}_{\mu} dx-\int_{\RR^3}(\overline{v}_\mu \cdot\nabla v)\cdot \overline{v}_{\mu}dx +\int_{\RR^3} (\overline{B}_\mu\cdot\nabla B)\cdot \overline{v}_{\mu}dx -\int_{\RR^3} (B_\mu \cdot\nabla \overline{v}_\mu)\cdot  \overline{B}_\mu  dx,
\end{eqnarray*}
where, we have used $\int_{\RR^3} (B_\mu \cdot\nabla \overline{B}_\mu )\cdot \overline{v}_{\mu}dx=-\int_{\RR^3} (B_\mu \cdot\nabla \overline{v}_\mu)\cdot  \overline{B}_\mu  dx$.\\
Likewise for $\overline{B}_\mu$, we also get  
\begin{equation*}
\frac{1}{2}\frac{d}{dt}\|\overline{B}_\mu(t)\|^2_{L^2}+\int_{\RR^3} |\nabla \overline{B}_{\mu}(t,x)|^2dx = - \int_{\RR^3}(\overline{v}_\mu \cdot\nabla B)\cdot \overline{B}_\mu  dx +\int_{\RR^3} (\overline{B}_\mu\cdot\nabla v)\cdot \overline{B}_\mu dx +\int_{\RR^3} (B_\mu \cdot\nabla\overline{v}_\mu )\cdot \overline{B}_\mu dx.
\end{equation*}
By summing the last two estimates, so, the fact $\int_0^t\int_{\RR^3} |\nabla \overline{B}_{\mu}(t,x)|^2dxd \tau \ge0 $ yields  
\begin{eqnarray*}
\frac{1}{2}\frac{d}{dt}\Big(\|\overline{v}_\mu(t)\|^2_{L^2}+\|\overline{B}_\mu(t)\|^2_{L^2}\Big) &=& \mu\int_{\RR^3} \Delta v_\mu \cdot \overline{v}_{\mu} dx-\int_{\RR^3}(\overline{v}_\mu \cdot\nabla v)\cdot \overline{v}_{\mu}dx +\int_{\RR^3} (\overline{B}_\mu\cdot\nabla B)\cdot \overline{v}_{\mu}dx\\
&&- \int_{\RR^3}(\overline{v}_\mu \cdot\nabla B)\cdot \overline{B}_\mu  dx +\int_{\RR^3} (\overline{B}_\mu\cdot\nabla v)\cdot \overline{B}_\mu dx.
\end{eqnarray*}

Thanks to the Cauchy-Shwartz inequality, it holds
\begin{align}
\begin{split}
\frac{1}{2}\frac{d}{dt}\Big(\|\overline{v}_\mu(t)\|^2_{L^2}+\|\overline{B}_\mu(t)\|^2_{L^2}\Big)&\le  \| \overline{v}_\mu\|_{ L^2} \big(\mu  \| \Delta v_\mu\|_{L^2}\big) +\| \overline{v}_\mu\|_{L^2}  \|\overline{v}_\mu \cdot\nabla v\|_{L^2} + \| \overline{v}_\mu\|_{ L^2} \|\overline{B}_\mu\cdot\nabla B\|_{L^2}\\& +\| \overline{B}_\mu\|_{L^2} \|\overline{v}_\mu\cdot\nabla B\|_{L^2}+\| \overline{B}_\mu\|_{ L^2} \| \overline{B}_\mu\cdot\nabla v \|_{L^2}.
\end{split}
\end{align}
Now, we intgrate in time to get 
\begin{align}
\begin{split}
\| \overline{v}_\mu {\|}^2_{L^\infty_t L^2}  + \| \overline{B}_\mu {\|}^2_{L^\infty_t L^2}  &\lesssim  \|\overline{v}^0_\mu\|^2_{L^2} +\| \overline{B}^0_\mu\|^2_{L^{2}}+ \| \overline{v}_\mu\|_{L^\infty_t L^2} \big(\mu  \| \Delta v_\mu\|_{L^1_{t}L^2}\big) +\| \overline{v}_\mu\|_{L^\infty_t L^2}  \|\overline{v}_\mu \cdot\nabla v\|_{L^1_{t}L^2} \\&+
\| \overline{v}_\mu\|_{L^\infty_t L^2} \|\overline{B}_\mu\cdot\nabla B\|_{L^1_{t}L^2} +\| \overline{B}_\mu\|_{L^\infty_t L^2} \|\overline{v}_\mu\cdot\nabla B\|_{L^1_{t}L^2} \\ &+
   \| \overline{B}_\mu\|_{L^\infty_t L^2} \| \overline{B}_\mu\cdot\nabla v \|_{L^1_{t}L^2}  .
\end{split}
\end{align}
Young's inequality ensures that 
\begin{align}
\begin{split}
\| \overline{v}_\mu {\|}^2_{L^\infty_t L^2}  + \| \overline{B}_\mu {\|}^2_{L^\infty_t L^2}  &\lesssim  \|\overline{v}^0_\mu\|^2_{L^2} +\| \overline{B}^0_\mu\|^2_{L^{2}} + \big(\mu  \| \Delta v_\mu\|_{L^1_{t}L^2}\big)^2 + \|\overline{v}_\mu \cdot\nabla v\|^2_{L^1_{t}L^2}+\|\overline{B}_\mu\cdot\nabla B\|^2_{L^1_{t}L^2}  \\&+
 \|\overline{v}_\mu\cdot\nabla B\|^2_{L^1_{t}L^2} +\| \overline{B}_\mu\cdot\nabla v \|^2_{L^1_{t}L^2}  .
\end{split}
\end{align}
Again from Young's inequality and last estimate, we find 
\begin{align*}
\begin{split}
2\| \overline{v}_\mu {\|}_{L^\infty_t L^2}  \| \overline{B}_\mu {\|}_{L^\infty_t L^2}  &\le \| \overline{v}_\mu {\|}^2_{L^\infty_t L^2}+ \| \overline{B}_\mu {\|}^2_{L^\infty_t L^2}\\&\lesssim 
\|\overline{v}^0_\mu\|^2_{L^2} +\| \overline{B}^0_\mu\|^2_{L^{2}} + \big(\mu  \| \Delta v_\mu\|_{L^1_{t}L^2}\big)^2 + \|\overline{v}_\mu \cdot\nabla v\|^2_{L^1_{t}L^2}+\|\overline{B}_\mu\cdot\nabla B\|^2_{L^1_{t}L^2}  \\&+
 \|\overline{v}_\mu\cdot\nabla B\|^2_{L^1_{t}L^2} +\| \overline{B}_\mu\cdot\nabla v \|^2_{L^1_{t}L^2}  .
\end{split}
\end{align*}
Gathering the last two estimate and employ that $a^2+b^2+2ab=(a+b)^2$, we end up with
\begin{align*}
\begin{split}
\big(\| \overline{v}_\mu {\|}_{L^\infty_t L^2}  + \| \overline{B}_\mu {\|}_{L^\infty_t L^2}\big)^2  &\lesssim 
\|\overline{v}^0_\mu\|^2_{L^2} +\| \overline{B}^0_\mu\|^2_{L^{2}} + \big(\mu  \| \Delta v_\mu\|_{L^1_{t}L^2}\big)^2 + \|\overline{v}_\mu \cdot\nabla v\|^2_{L^1_{t}L^2}+\|\overline{B}_\mu\cdot\nabla B\|^2_{L^1_{t}L^2}  \\&+
 \|\overline{v}_\mu\cdot\nabla B\|^2_{L^1_{t}L^2} +\| \overline{B}_\mu\cdot\nabla v \|^2_{L^1_{t}L^2}  .
\end{split}
\end{align*}
Consequently,
\begin{align*}
\begin{split}
\mathscr{A}(t) \lesssim  \mathscr{A}(0)+ \mu \int_0^t\| \Delta v_\mu(\tau) {\|}_{L^2}d\tau  + \int_0^t \Big( \|\nabla v (\tau) {\|}_{L^\infty}+\|\nabla B(\tau) {\|}_{L^\infty} \Big)\mathscr{A}(\tau)  d \tau .
\end{split}
\end{align*}
with $\mathscr{A}(t)=\| \overline{v}_\mu {\|}_{L^\infty_t L^2}  + \| \overline{B}_\mu {\|}_{L^\infty_t L^2} $ and $ \mathscr{A}(0)=0$.  Gronwall's inequality leads to
\begin{equation}\label{Est:L^2-con}
\mathscr{A}(t)   \lesssim e^{\int_0^t ( \|\nabla v (\tau) {\|}_{L^\infty}+\|\nabla B(\tau) {\|}_{L^\infty})d \tau } \Big(\mu \int_0^t\| \Delta v_\mu(\tau) {\|}_{L^2} d\tau \Big).
\end{equation}
For the last term of r.h.s. of \eqref{Est:L^2-con}, H\"older's inequality in time variable, the embedding $\mathscr{B}^0_{2,1} \hookrightarrow L^2$ and the continuity of $\Delta : \mathscr{B}^2_{2,1} \rightarrow  \mathscr{B}^0_{2,1} $ allow us to write
$$
\| \overline{v}_\mu {\|}_{ L^{\infty}_t L^2}  + \| \overline{B}_\mu  {\|}_{L^{\infty}_t L^2}  \lesssim e^{\int_0^t ( \|\nabla v (\tau) {\|}_{L^\infty}+\|\nabla B(\tau) {\|}_{L^\infty})d \tau } \Big( (\mu t) \|  v_\mu {\|}_{L^{\infty}_t \mathscr{B}^2_{2,1}} \Big).
$$
Concerning the term $\|v_\mu {\|}_{L^{\infty}_t \mathscr{B}^{2}_{2,1}}$, we explore the fact that $\mathscr{B}^{\frac{5}{2}}_{2,1}\hookrightarrow  \mathscr{B}^2_{2,1} $ and  Proposition \ref{Prop:4.4} for $p=2$ to obtain
$$ \|  v_\mu {\|}_{L^{\infty}_t \mathscr{B}^2_{2,1}}  \le \|  v_\mu \|_{L^{\infty}_t \mathscr{B}^{\frac{5}{2}}_{2,1}} \le \Phi_6(t). $$ 
From the  last three estimates and  \eqref{Es:B,b}, \eqref{b,v-Lip}, we infer that
\begin{equation}\label{rote-L^2}
\| \overline{v}_\mu {\|}_{L^{\infty}_t L^2}  + \| \overline{B}_\mu {\|}_{L^{\infty}_t L^2}   \le  (\mu t) \Phi_6(t).
\end{equation}
Now, we give the rate of convergence of velocities and magnetic fields in $L^\infty(\RR_+, \mathcal{B}^{0}_{2,1})$. To do so, using  the definition of $\mathcal{B}^{0}_{2,1}$ to write 
\begin{eqnarray*}
\|\overline{B}_\mu\|_{\mathcal{B}^{0}_{2,1}}= \sum _{q\ge -1} \|\Delta_{q} \overline{B}_\mu\|_{L^2}=\sum _{q\ge -1} 2^{-\frac{q}{4}}\|\Delta_{q} \overline{B}_\mu\|^{\frac{1}{2}}_{L^2}2^{\frac{q}{4}}\|\Delta_{q} \overline{B}_\mu\|^{\frac{1}{2}}_{L^2}
\end{eqnarray*}
From Cauchy-Shwartz's inequality for the series, it follows  
\begin{eqnarray*}
\|\overline{B}_\mu\|_{\mathcal{B}^{0}_{2,1}}&\lesssim& \Big( \sum _{q\ge -1} 2^{-\frac{q}{2}} \|\Delta_{q} \overline{B}_\mu\|_{L^2} \Big)^{\frac{1}{2}} \Big( \sum _{q\ge -1} 2^{\frac{q}{2}} \|\Delta_{q} \overline{B}_\mu\|_{L^2} \Big)^{\frac{1}{2}}\nonumber\\
  &\lesssim& \|\overline{B}_\mu\|^\frac{1}{2}_{\mathcal{B}^{-\frac{1}{2}}_{2,1}} \|\overline{B}_\mu\|^\frac{1}{2}_{\mathcal{B}^{\frac{1}{2}}_{2,1}}.
\end{eqnarray*}
Further, the embedding $L^2 \hookrightarrow \mathcal{B}^{-\frac{1}{2}}_{2,1}$, implies
\begin{eqnarray}\label{Est:B-0}
\|\overline{B}_\mu\|_{\mathcal{B}^{0}_{2,1}} \lesssim \|\overline{B}_\mu\|^\frac{1}{2}_{L^2}\big( \|B_\mu\|_{\mathcal{B}^{\frac{1}{2}}_{2,1}}+\|B\|_{\mathcal{B}^{\frac{1}{2}}_{2,1}}\big)^\frac{1}{2}.
\end{eqnarray}
Similarly for $\overline{v}_\mu$, we write 
\begin{eqnarray*}
\|\overline{v}_\mu\|_{\mathcal{B}^{0}_{2,1}} \lesssim \|\overline{v}_\mu\|^\frac{1}{2}_{L^2}\big( \|v_\mu\|_{\mathcal{B}^{\frac{1}{2}}_{2,1}}+\|v\|_{\mathcal{B}^{\frac{1}{2}}_{2,1}}\big)^\frac{1}{2}.
\end{eqnarray*}
Or, the embedding $\mathcal{B}^{\frac{5}{2}}_{2,1} \hookrightarrow \mathcal{B}^{\frac{1}{2}}_{2,1}$, implies that
\begin{eqnarray}\label{v-0-2}
\|\overline{v}_\mu\|_{\mathcal{B}^{0}_{2,1}} \lesssim  \|\overline{v}_\mu\|^\frac{1}{2}_{L^2}\big( \|v_\mu\|_{\mathcal{B}^{\frac{5}{2}}_{2,1}}+\|v\|_{\mathcal{B}^{\frac{5}{2}}_{2,1}}\big)^\frac{1}{2}.
\end{eqnarray}
Gathering \eqref{Est:B-0} and \eqref{v-0-2}, next we apply {\bf(ii)-}Preposition's \ref{Prop:4.4} for $p=2$, we find 
\begin{equation*}
\|\overline{v}_\mu\|_{L^{\infty}_{t} \mathcal{B}^{0}_{2,1}} +\|\overline{B}_\mu\|_{L^{\infty}_{t} \mathcal{B}^{0}_{2,1}}\le \big( \|\overline{v}_\mu\|^\frac{1}{2}_{L^{\infty}_{t} L^2} +\|\overline{B}_\mu\|^\frac{1}{2}_{L^{\infty}_{t} L^2} \big)\Phi_6(t)
\end{equation*}
Thanks to \eqref{rote-L^2}, we deduce that
\begin{eqnarray}\label{B-2.0}
\|\overline{v}_\mu\|_{L^{\infty}_{t} \mathcal{B}^{0}_{2,1}} +\|\overline{B}_\mu\|_{L^{\infty}_{t} \mathcal{B}^{0}_{2,1}} \le (\mu t)^\frac{1}{2} \Phi_6(t) .
\end{eqnarray}
{\bf Step 2.} In this step we will evaluate the rate convergence \eqref{Rate} in the resolution space via the complex interpolation. For this aim,  let $N\in\NN$ be an integer will be judiciously chosen. By definition of $\mathscr{B}^{0}_{p,1}$ and Bernstein's inequality, we have     
\begin{eqnarray}\label{Int-v}
\|\overline{v}_\mu\|_{\mathscr{B}^{0}_{p,1}}&= & \sum _{q\ge -1} \|\Delta_{q} \overline{v}_\mu\|_{L^p}\nonumber\\
  &\le&  \sum _{q\le N} \|\Delta_{q} \overline{v}_\mu\|_{L^p}+\sum _{q>N} \|\Delta_{q} \overline{v}_\mu\|_{L^p} \nonumber \\
&\lesssim& \sum _{q\le N} 2^{q(\frac{3}{2}-\frac{3}{p})}\|\Delta_{q} \overline{v}_\mu\|_{L^2}+\sum _{q > N}2^{-q(1+\frac{3}{p})} 2^{q\frac{3}{p}} \|\nabla \Delta_{q} \overline{v}_\mu\|_{L^p} \nonumber \\
&\lesssim& 2^{N(\frac{3}{2}-\frac{3}{p})} \|\overline{v}_\mu\|_{\mathcal{B}^{0}_{2,\infty}}+2^{-N(1+\frac{3}{p})}  \|\overline{v}_\mu\|_{\mathcal{B}^{1+\frac{3}{p}}_{p,\infty}}.
\end{eqnarray}
Taking $N$ such that 
$$2^{\frac{5}{2} N}=\frac{ \|\overline{v}_\mu\|_{\mathcal{B}^{1+\frac{3}{p}}_{p,1}} }{\|\overline{v}_\mu\|_{\mathcal{B}^{0}_{2,\infty}}}.$$
Inserting the last estimate in \eqref{Int-v}, we obtain 
\begin{equation*}
\|\overline{v}_\mu\|_{\mathscr{B}^{0}_{p,1}} \lesssim  \|\overline{v}_\mu\|^{\frac{2}{5}+\frac{6}{5p}}_{\mathcal{B}^{0}_{2,\infty}} \|\overline{v}_\mu\|^{\frac{3}{5}-\frac{6}{5p}}_{\mathcal{B}^{1+\frac{3}{p}}_{p,\infty}}.
\end{equation*}
Since $\mathcal{B}^{s}_{p,1} \hookrightarrow \mathcal{B}^{s}_{p,\infty} $, then Proposition's \ref{Prop:4.4} yields
\begin{eqnarray*}
\|\overline{v}_\mu\|_{L^{\infty}_{t} \mathscr{B}^{0}_{p,1}} &\lesssim&  \|\overline{v}_\mu\|^{\frac{2}{5}+\frac{6}{5p}}_{L^{\infty}_{t}\mathcal{B}^{0}_{2,1}} \big( \|v \|_{L^{\infty}_{t}\mathcal{B}^{1+\frac{3}{p}}_{p,1}} + \|v_\mu\|_{L^{\infty}_{t}\mathcal{B}^{1+\frac{3}{p}}_{p,1}} \big)^{\frac{3}{5}-\frac{6}{5p}}\\&\le&  \|\overline{v}_\mu\|^{\frac{2}{5}+\frac{6}{5p}}_{L^{\infty}_{t}\mathcal{B}^{0}_{2,1}} \Phi_6(t).
\end{eqnarray*}
Owing to \eqref{B-2.0}, we infer that 
\begin{equation}\label{rote-v-L^p}
\|\overline{v}_\mu\|_{L^{\infty}_{t}\mathscr{B}^{0}_{p,1}} \le  (\mu t)^{\frac{1}{5}+\frac{3}{5p}} \Phi_6(t). 
\end{equation}
For the term $\|\overline{B}_\mu\|_{L^{\infty}_{t}\mathcal{B}^{-1}_{p,1}}$, we distinguish two cases, first $  p > 6$ and $p \le 6$.\\  For $p>6$ we proceed by the same argument as above combined with Bernstein's inequality implies that 
\begin{eqnarray*}\label{Int-B-1}
\|\overline{B}_\mu\|_{\mathcal{B}^{-1}_{p,1} }  &\le&  \sum _{q\le N} 2^{-q}\|\Delta_{q} \overline{B}_\mu\|_{L^p}+\sum _{q > N} 2^{-q} \|\Delta_{q} \overline{B}_\mu\|_{L^p} \nonumber \\
&\lesssim& \sum _{q\le N} 2^{q(\frac{1}{2}-\frac{3}{p})}\|\Delta_{q} \overline{B}_\mu\|_{L^2}+\sum _{q > N}  2^{-q\frac{3}{p}} 2^{q(\frac{3}{p} -1)} \| \Delta_{q} \overline{B}_\mu\|_{L^p} \nonumber \\
&\lesssim& 2^{N(\frac{1}{2}-\frac{3}{p})} \|\overline{B}_\mu\|_{\mathcal{B}^{0}_{2,\infty}}+2^{-N \frac{3}{p}}  \|\overline{B}_\mu\|_{\mathcal{B}^{\frac{3}{p}-1}_{p,\infty}}.
\end{eqnarray*}
Choosing  $N$ such that 
$$2^{\frac{1}{2} N}=\frac{ \|\overline{B}_\mu\|_{\mathcal{B}^{\frac{3}{p}-1}_{p,\infty}}}{\|\overline{B}_\mu\|_{\mathcal{B}^{0}_{2,\infty}}}.$$
Combining with the last two estimates, it holds 
\begin{eqnarray}
\|\overline{B}_\mu\|_{\mathcal{B}^{-1}_{p,1} }  &\lesssim& \|\overline{B}_\mu\|^{\frac{6}{p}}_{\mathcal{B}^{0}_{2,\infty}} \|\overline{B}_\mu\|^{1-\frac{6}{p}}_{\mathcal{B}^{\frac{3}{p}-1}_{p,\infty}}  . 
\end{eqnarray}

From {\bf(ii)-}Proposition \ref{Prop:4.4}, it follows
\begin{eqnarray*}\label{p>6}
\|\overline{B}_\mu\|_{L^{\infty}_{t}\mathcal{B}^{-1}_{p,1} }  &\le & \|\overline{B}_\mu\|^{\frac{6}{p}}_{L^{\infty}_{t}\mathcal{B}^{0}_{2,1}} \Phi_6(t) . 
\end{eqnarray*}  
For the second case, $2<p \le 6$, using the embeddings  $\mathcal{B}^{0}_{2,1} \hookrightarrow \mathcal{B}^{\frac{1}{2}-\frac{3}{p}}_{2,1} \hookrightarrow  \mathcal{B}^{-1}_{p,1}$ to write 
\begin{equation*}\label{p=<6}
\|\overline{B}_\mu\|_{L^{\infty}_{t} \mathcal{B}^{-1}_{p,1}} \lesssim \|\overline{B}_\mu\|_{ L^{\infty}_{t} \mathcal{B}^{\frac{1}{2}-\frac{3}{p}}_{2,1}} \lesssim \|\overline{B}_\mu\|_{L^{\infty}_{t} \mathcal{B}^{0}_{2,1}},
\end{equation*}
combined the last two estimates with \eqref{B-2.0} to conclude that
\begin{equation*}
\|\overline{B}_\mu\|_{L^{\infty}_{t} \mathcal{B}^{-1}_{p,1}}\le
\left\{\begin{array}{ll}
(\mu t)^{\frac{1}{2}} \Phi_6(t), \quad& \textrm{if $p \le 6 ,$}  \\(\mu t)^{\frac{6}{2p}} \Phi_6(t),&\textrm{if $p>6$}.
\end{array}
\right.
\end{equation*}
Consequently, we get 
\begin{equation}\label{rote-B-1}
\|\overline{B}_\mu\|_{L^{\infty}_{t} \mathcal{B}^{-1}_{p,1}}\le (\mu t)^{\frac{6}{2\max(p,6)}} \Phi_6(t) .
\end{equation}
To finalize, let us  move to the term $\|\overline{B}_\mu\|_{L^{1}_{t} \mathcal{B}^{1}_{p,1}}$. Again Bernstein's inequality yields 
\begin{eqnarray}\label{Int-B+1}
\|\overline{B}_\mu\|_{\mathcal{B}^{1}_{p,1} }  &\le&  \sum _{q\le N} 2^{q}\|\Delta_{q} \overline{B}_\mu\|_{L^p}+\sum _{q> N} 2^{q} \|\Delta_{q} \overline{B}_\mu\|_{L^p} \nonumber \\
&\lesssim& \sum _{q\le N} 2^{q(\frac{5}{2}-\frac{3}{p})}\|\Delta_{q} \overline{B}_\mu\|_{L^2}+\sum _{q > N}  2^{-q\frac{3}{p}} 2^{q\frac{3}{p}} \| \Delta_{q} \overline{B}_\mu\|_{L^p} \nonumber \\
&\lesssim& 2^{N(\frac{5}{2}-\frac{3}{p})} \|\overline{B}_\mu\|_{\mathcal{B}^{0}_{2,\infty}}+2^{-N \frac{3}{p}}  \|\overline{B}_\mu\|_{\mathcal{B}^{\frac{3}{p}+1}_{p,\infty}}.
\end{eqnarray}
We choose  $N$ such that 
$$2^{\frac{5}{2} N}=\frac{ \|\overline{B}_\mu\|_{\mathcal{B}^{1+\frac{3}{p}}_{p,\infty}} }{\|\overline{B}_\mu\|_{\mathcal{B}^{0}_{2,\infty}}}.$$
Plugging the above estimate int  \eqref{Int-B+1}, we get 
\begin{eqnarray*}
\|\overline{B}_\mu\|_{\mathcal{B}^{1}_{p,1} }  &\lesssim& \|\overline{B}_\mu\|^{\frac{6}{5p}}_{\mathcal{B}^{0}_{2,\infty}} \|\overline{B}_\mu\|^{1-\frac{6}{5p}}_{\mathcal{B}^{1+\frac{3}{p}}_{p,\infty}}  .
\end{eqnarray*}
In view of \eqref{B-2.0} and {\bf(ii)-}Proposition \ref{Prop:4.4}, we deduce that
\begin{equation}\label{rote-B+1}
\|\overline{B}_\mu\|_{L^{1}_{t} \mathcal{B}^{1}_{p,1}}  \le  (\mu t)^{\frac{3}{5p}} \Phi_6(t).
\end{equation}
Finally gathering \eqref{rote-v-L^p},\eqref{rote-B-1} and \eqref{rote-B+1}, we find the desired result .\\
This completes the proof of  Theorem \ref{The:1.2}.
\end{proof}
\section{Appendix}
\hspace{0.5cm}
We state a technical lemma about the scaling in Besov space. For the proof, we refer to \cite[Proposition A.1]{AKH}.
\begin{lem}\label{Lem:6.1}
 Let $ f : \mathbb{R}^3 \rightarrow \mathbb{R}$ be a function belonging to $\mathscr{B}^{0}_{\infty
,1}$ and take $f_{\lambda}( x_1; x_2; x_3)
= f(\lambda x_1; x_2; x_3) $ with $ \lambda \in  (0; 1)$. Then, there exists an absolute constant $C>0$ such that the following inequality holds:
$$ \|f_\lambda\|_{\mathscr{B}^{0}_{\infty ,1}} \le C(1-\log\lambda) \|f\|_{\mathscr{B}^{0}_{\infty ,1}} . $$
\end{lem}
Next, we give boundedness for a product between two quantities in Besov and $L^\infty$ spaces; for the proof, we refer again to \cite[Proposition A.1.]{Hassainia}.
\begin{prop}\label{prop:7.10}Let $B$ be a smooth axisymmetric vector field with a trivial radial component,$B^r=0$. Let $p\in[1,\infty]$. Then we have the following estimate: 

\begin{equation*}
\Big\|\frac{B_{\theta}}{r}B\Big\|_{\mathscr{B}_{p, 1}^{1+\frac{3}{p}}}\lesssim\Big\|\frac{B}{r}\Big\|_{L^\infty}\|B\|_{ \mathscr{B}_{p, 1}^{1+\frac{3}{p}}}.
\end{equation*}
¨
\end{prop}

\subsection{Commutator estimates}In this subsection, we state three different commutator estimates. The first is classical and deals with the commutator with the dyadic block and advection operator. Its proof can be found in \cite[Proposition 5.4]{hk1}. 
\begin{prop}\label{prop:7.1}
 Let $u$ be a smooth function and $v$ be a smooth vector field in divergence-free of $\RR^3$ such that it's vorticity $ \omega \triangleq \cur v$ belongs to $L^\infty$. Then for every $p\in[1,\infty]$ and  $q \geq -1$  we have
\begin{equation*}
\|[\Delta_q,v\cdot \nabla]u\|_{L^{p}} \le  C \|u \|_{L^{p}}\big(\|\nabla \Delta_{-1} v \|_{ L^{\infty}} +(q+2) \|\omega \|_{L^{\infty}}\big).
\end{equation*}
\end{prop}
The second result cares with the same commutator in $L^2-$space, taking into account the axisymmetric structure. For the detailed proof, we can refer \cite[Proposition 3.2]{HKR1}.

\begin{prop}\label{prop:2.1}
Let $v$ be an axisymmetric smooth vector field without swirl in divergence-free and let $u$ be a smooth scalar function. Then there exists $C >0$ such that for every $q \in  \NN \cup {-1}$  we have:
\begin{equation*}
\|[\Delta_q,v\cdot \nabla]u\|_{L^{2}} \le C  \|\frac{\omega}{r}\Big\|_{ L^{3,1}} \big( \|x_{h}u\|_{L^{6}}+\|u\|_{L^{2}}\big).
\end{equation*}
where $ \omega^\theta$ is the angular component of $\omega=\nabla \times v$.
\end{prop}
We finish this subsection by estimating the commutator in Besov space. More precise, we have. 
\begin{lem}\label{Lem:2.3} Let $u$ be a smooth function and $v$ be a smooth vector field of $\RR^3$ in divergence-free. Then for every $p\in [1,\infty]$ and $ s\in(-1,1)$, we have:
\begin{equation*}
\sum_{q\geq 0} 2^{qs} \|[\Delta_q,v\cdot \nabla]u\|_{L^{p}} \lesssim  \|\nabla u\|_{L^{\infty}} \| v\|_{\mathscr{B}^{s}_{p,1}} 
\end{equation*}
\end{lem}
The proof can be found in \cite{che1}.

\end{document}